\documentclass[11pt, oneside,british,english,american]{amsart}
\usepackage[T1]{fontenc}
\usepackage[utf8]{inputenc}
\usepackage{color}
\usepackage{refstyle}
\usepackage{amstext}
\usepackage{amsthm}
\usepackage{amssymb}
\usepackage[normalem]{ulem}
\usepackage{cancel}
\usepackage{babel}
\usepackage{graphicx}
\usepackage{subcaption} 
\usepackage{float}

\usepackage{enumitem}
\usepackage{subcaption}

\makeatletter

\AtBeginDocument{\providecommand\eqref[1]{\ref{eq:#1}}}
\RS@ifundefined{subsecref}
  {\newref{subsec}{name = \RSsectxt}}
  {}
\RS@ifundefined{thmref}
  {\def\RSthmtxt{theorem~}\newref{thm}{name = \RSthmtxt}}
  {}
\RS@ifundefined{lemref}
  {\def\RSlemtxt{lemma~}\newref{lem}{name = \RSlemtxt}}
  {}

\makeatother

\newtheorem{theorem}{Theorem}[section]
\newtheorem{thm}[theorem]{Theorem}

\newtheorem{prop}[theorem]{Proposition}
\newtheorem{lemma}[theorem]{Lemma}

\newtheorem{remark}{Remark}
\newtheorem{notation}{Notation}
\numberwithin{equation}{section}
\numberwithin{figure}{section}

\newcommand{\RR}{\mathbb{R}}
\newcommand{\NN}{\mathbb{N}}

\newcommand{\Om}{\Omega}
\newcommand{\p}{\partial}
\newcommand{\pd}[2]{\frac {\p #1}{\p #2}}
\newcommand{\ds}{\displaystyle}
\newcommand{\beq}{\begin{equation}}
\newcommand{\eeq}{\end{equation}}  
\newcommand{\eqnref}[1]{(\ref {#1})}

\begin{document}
 
\global\long\def\Omd{\Omega_{\delta}}
\title{
Cloaking laminate design based on GPT-vanishing structures
}
\author{
Eleanor Gemida$^{1\text{,2}}$ and Mikyoung LIM$^{3}$}

\begin{abstract}

We propose a near-cloaking design which is a lamination of a finite number of layers of isotropic materials. The proposed design is an approximation of a cloaking material obtained by pushing forward a multi-coated structure for which the coating cancels the generalized polarization tensors (GPTs) up to several leading orders. The enhanced cloaking effect achieved by the GPT-vanishing structure permits a coarser microscale requirement and reduces the contrast in the constituent isotropic materials, thereby improving constructibility of cloaking laminates compared with designs based on a non-coated structure.
\end{abstract}

\address{ 
$^{1}$Université Paris Cité, CNRS, Sorbonne Université, Laboratoire
Jacques-Louis Lions UMR 7598, F-75006 Paris, France; $^{2}$Institute
of Mathematical Sciences, College of Arts and Sciences, University
of the Philippines Los Ba\~{n}os, 4031 Laguna, Philippines; $^{3}$Department of Mathematical Sciences, Korean Advanced Institute of Science and Technology, Daejeon 34141, Republic of Korea}
\thanks{ 
This project has received funding from the European Union's Horizon
2020 research and innovation programme under the Marie Skłodowska-Curie
grant agreement No 945332.}
\thanks{ 
This project has received funding from the College of Doctoral Schools
and the SMARTS-UP project of the Université Paris Cité.}
\thanks{ 
This work was supported by a National Research Foundation of Korea (NRF) grant funded by the Korean government (MSIT) No. RS-2024-00359109}
\email{ 
eleanor.gemida@etu.u-paris.fr}
\email{ 
mklim@kaist.ac.kr}
\maketitle

\section{Introduction}

Passive cloaking involves surrounding an object with a
specially designed material, commonly referred to as a metamaterial,
that renders both the object and the cloak itself undetectable
to external sensing, typically via electromagnetic waves. Significant
progress in this field  has been achieved through the change of variables scheme, also known as transformation optics
\cite{Greenleaf:2003:NCI,Kohn:2008:CCV,Milton:2006:CEA,Pendry:2006:CEF,Chen:2010:ACT,Kohn:2010:CCV,Bao:2014:NCE,Hu:2015:NCE}.
The core component of this method is a coordinate transformation that preserves a fixed outer boundary, such as $\partial B_1$, where $B_s$ denotes a ball of radius $s$ centered at the origin.

More specifically, approximate cloaking (or near cloaking) employs a transformation that expands a small region, $B_{\rho}$, into a region of a finite size,  such as  $B_{{1}/{2}}$.
This expanded region $B_{{1}/{2}}$, functions as the cloaked zone. In the context of the conductivity
equation, the approximate cloak--denoted by $\sigma_{c}$--is generated by pushing forward the background conductivity through this transformation. This approach achieves a cloaking effect up to order $\rho^{d}$ in dimension $d$
\cite{Kohn:2008:CCV}. For the electric impedance tomography (EIT), this implies that the Dirichlet-to-Neumann data measured on the outer boundary is nearly identical to that of a uniform ball, subject only to perturbations on the order of those produced by a negligible inclusion of size $\rho$. However, while the technique
is theoretically straightforward, the required metamaterials are anisotropic, posing significant challenges for manufacturing. 

 To address this, a recent study \cite{Capdeboscq:2025:CCC} proposes a near-cloaking design achieving invisibility order $\rho^d$ using finitely many isotropic layers composed of three or more distinct isotropic homogeneous materials.  Structured as a radial laminate, this design is suitable
for fabrication via additive manufacturing. The maximal and minimal isotropic conductivities in this design are of the same order as the maximal and minimal
directional eigenvalues of the anisotropic near-cloak $\sigma_{c}$. 
The foundational tool in this derivation is homogenization theory (see, e.g., \cite{Cioranescu:1999:IH}, \cite{Allaire:2002:SOH}), wherein $\sigma_c$ is interpreted as an effective medium to be realized by appropriately configured isotropic constituent materials. 
 More precisely, the realization of $\sigma_c$ involves constructing a sequence of heterogeneous, locally isotropic media with finite-scale microstructures that converges to the target anisotropic conductivity $\sigma_c$ in the sense of the homogenization. By rigorously quantifying the homogenization error, the study establishes an estimate for the necessary microscale. Homogenization-based approach to cloaking has been discussed in earlier works \cite{Greenleaf:2008:ITO,Farhat:2008:AIF},  however,  rigorous analysis of the homogenization error was not addressed. 
A critical consideration in this design is the utilization of coarser microstructures, which significantly facilitates more practical implementation.
 
In this paper, we propose a cloaking laminate design that utilizes a coarser microscale while maintaining the target invisibility order $\rho^d$. This is accomplished by integrating the GPT-vanishing multi-coating scheme \cite{Ammari:2013:ENC_1} with the homogenization-based laminate framework \cite{Capdeboscq:2025:CCC}. 

The multi-coating technique, originally developed for the conductivity equation \cite{Ammari:2013:ENC_1} and subsequently extended to various physical settings \cite{Ammari:2013:ENCFM, Abbas:2017:TDE}, enhances invisibility by combining the blow-up transformation with a specialized coating structure.
In this framework, the coating occupies the annular region $B_{2\rho}\setminus B_{\rho}$ prior to the transformation. It consists of a finite number of layers with constant conductivities, specifically selected to cancel the leading order generalized polarization tensors (GPTs).

By neutralizing these leading order terms before applying the blow-up transformation, the resulting structure yields a near-cloak with an enhanced invisibility order of $\rho^{d+2N}$ for some integer $N$ referred to as the order of the GPT-vanishing structure. 
This improvement allows for a relation of the smallness constraint on the pre-transformation structure; specifically, if the pre-transformation inclusion size is increased to $\rho^{\frac{d}{d+2N}}$ (rather than $\rho$), the resulting invisibility order remains $\rho^{d}$. 
Notably, our proposed cloaking laminate based on GPT-vanishing structures achieves a cloaking effect comparable to the design in \cite{Capdeboscq:2025:CCC} while permitting a significantly larger inclusion, thereby enhancing its practical feasibility.   

To state the main results more precisely, let $\Omega=B_{1}\subset\mathbb{R}^{d}$ for $d=2,3$. 
The proposed cloaking laminate is designed to hide the core region $B_{{1}/{2}}$ from boundary measurements on $\p B_1$. The laminate is a radial structure occupying the annulus $B_{{3}/{4}}\setminus B_{{1}/{2}}$, where the background medium $B_1\setminus B_{{3}/{4}}$ has a constant conductivity of $1$. 

The laminate is constructed using the GPT-vanishing structure. 
Let $\sigma$ denote a conductivity profile exhibiting the GPT-vanishing property of order $N$ and containing an insulating core, as detailed in Propositions \ref{prop:radial:2D} and \ref{prop:radial:3D}.
Define the contrast parameter
\begin{equation}\label{kappa}
\kappa:=\max\big\{ \max\sigma,\left(\min\sigma\right)^{-1}\big\},
\end{equation}
where the maximum and minimum of $\sigma$ are taken over the coating layers, assuming a positive and piecewise constant conductivity profile therein. 
Given a small parameter $\rho\in\left(0,1/2\right)$, we construct the laminates using the conductivities $1$ (matching that of the background medium), $\alpha$ (a low-conductivity material), and a set of high-conductivity materials $\gamma^{(1)},\dots, \gamma^{(m)}$, where the number of high-conductivity materials, $m\geq 1$, depends on $\sigma$ and $\rho$. 
We take $\alpha$ to be a positive constant satisfying
\begin{equation}\label{eq:boundalpha}
\alpha=s\left(\kappa\lvert\ln\rho\rvert\right)^{-1}\left(\rho^{\frac{d}{d+2N}}\right)^{d-2}\quad\mbox{for some }0<s<1
\end{equation}
and select positive constant conductivities $\gamma^{(i)}, i=1,\dots m$, which satisfy
\begin{equation}\label{eq:boundgamma}
\gamma^{(i)}=t^{(i)}\,\kappa\,\lvert\ln\rho\rvert\quad\text{for some }t^{(i)}>1.
\end{equation}

The cloaking laminate is a heterogeneous structure defined on the radial interval $[1/2, 3/4]$ with a conductivity profile $A_{\varepsilon}$. Although the heterogeneity is confined to this interval and the conductivity remains constant in the background region $B_1\setminus B_{3/4}$, we present $A_\varepsilon$ in the following unified form for notational simplicity:
for a lamination scale $\varepsilon>0$, 
\beq\label{A_ep:intro}
A_{\varepsilon}\left(x\right)=\sum_{k=0}^{N_{\varepsilon}}a_{k}\left(\frac{\left|x\right|}{\varepsilon}\right)1_{I_{k}}(x),
\eeq
where $\left\{ I_{k}\right\}_{k=0}^{N_\varepsilon}$ partitions the radial interval $\left[{1}/{2},1\right]$ into uniform sub-intervals of length $\varepsilon$. 
For each $k$, the function $s\mapsto a_{k}^{}(s)$ represents a $1$-periodic laminate composed of isotropic materials $\alpha$, $1$, and $\gamma$ (where $\gamma\in\left\{ \gamma^{(i)}\right\}_{i=1}^{m}$). The volume fractions of these materials vary for each $k$ (refer to \eqnref{eq:Aeps_def} and \eqnref{eq:proportions_3mat} for more details). For $k$ such that $I_k\subset[3/4,1]$, $a_k$ is constantly $1$.

Regarding the boundary measurements, we consider the Dirichlet-to-Neumann map on $B_1$ for two cases: one with an insulating core and the other with a core of general conductivity. 
The main results of this paper address near-cloaking for these two cases, with the formulation for the insulating core presented below in Theorem \ref{thm:main}. 
Let $ \Lambda_{B_{1},B_{{1}/{2}}}\left[A_{\varepsilon}\right]:H^{\frac{1}{2}}\left(\partial B_{1}\right) \to H^{-\frac{1}{2}}\left(\partial B_{1}\right)$ be the operator defined by
\begin{align*}
\varphi & \mapsto A_{\varepsilon}\frac{\partial u^{\varepsilon}}{\partial\nu}\Big|_{\p B_1},
\end{align*}
where $u^\varepsilon$ is the solution to
$$\begin{cases}
\ds\nabla\cdot A_{\varepsilon}\nabla u^{\varepsilon}=0 & \text{in }B_{1}\setminus\overline{B_{{1}/{2}}},\\
\ds A_{\varepsilon}\frac{\partial u^{\varepsilon}}{\partial\nu}=0 & \text{on }\partial B_{{1}/{2}}\\
\ds u^{\varepsilon}=\varphi & \text{on }\partial B_{1}.
\end{cases}$$

\begin{thm}
\label{thm:main} 
Let $\sigma$ be a conductivity profile exhibiting the GPT-vanishing property of order $N$ and containing an insulating core in $\RR^d$ ($d=2,3$), as detailed in Propositions \ref{prop:radial:2D} and \ref{prop:radial:3D}. 
Let $A_\varepsilon$ be given by \eqnref{A_ep:intro} for a lamination scale $\varepsilon$ satisfying 
\begin{equation}\label{thm:cond:ep}
\varepsilon=\begin{cases}
O\left(\kappa^{-3}\rho^2\left|\ln\rho\right|^{-3}\right) & \text{for }d=2,\\[2mm]
O\left(\kappa^{-3}\rho^{3+\frac{3}{2N+3}}\left|\ln\rho\right|^{-1}\right) & \text{for }d=3.
\end{cases}
\end{equation}
Then, the following estimate holds:
\[
\left\Vert \Lambda_{B_{1},B_{{1}/{2}}}\left[A_{\varepsilon}\right]-\Lambda_{B_{1}}\left[1\right]\right\Vert _{\mathcal{L}\left(H^{\frac{1}{2}}\left(\partial B_{1}\right),H^{-\frac{1}{2}}\left(\partial B_{1}\right)\right)}\le C\rho^{d},
\]
where $C$ is a positive constant independent of $\rho$ but dependent on $\sigma$. Here, $\Lambda_{B_{1}}[1]$ is the Dirichlet-to-Neumann
map for the homogeneous background conductivity (Laplacian) on $B_{1}$, and $\|\cdot\|_{\mathcal{L}(H_1,H_2)}$ denotes the standard linear operator norm.
\end{thm}

We now discuss the cloaking effect of the laminate structure $A_\varepsilon$ given in \eqnref{A_ep:intro} for a core possessing general conductivity. The enhancement strategy introduced in \cite{Ammari:2013:ENC_1} addresses the cases where the conductivity within the cloaked region $B_{1/2}$ is either constant and known, or where the boundary $\p B_{1/2}$ is perfectly insulated; the latter case is the focus of Theorem \ref{thm:main} in this paper.
This approach was extended to a more general setting by Heumann and Vogelius
\cite{Heumann:2014:AEA}, who achieved enhanced cloaking for inclusions of arbitrary conductivity. Their strategy involves introducing a low-conductivity layer occupying $B_{\rho}\setminus B_{\frac{\rho}{2}}$ prior to the blow-up transformation, which corresponds to $B_{{1}/{2}}\setminus B_{{1}/{4}}$ in the physical domain. 
While this method does not increase anisotropy level according to certain measures proposed in \cite{Heumann:2014:AEA}, it results in a lower minimal directional eigenvalue. In Section
\ref{sec:arbitrary_inclusion}, we apply a similar approach to our proposed laminate $A_\varepsilon$. We show that by incorporating a low-conductivity material within the region $B_{{1}/{2}}\setminus B_{{1}/{4}}$, the laminate $A_{\varepsilon}$ cloaks inclusions of arbitrary conductivity located in $B_{{1}/{4}}$ with an invisibility level of $\rho^{2}$.

The remainder of this paper is organized as follows. In Section \ref{sec:GPT}, we review
the multipole expansion for conductivity inclusion problems, the concept of (contracted) GPTs, and the GPT-vanishing structures. Section \ref{sec:BVP} provides an estimate of the Dirichlet-to-Neumann map for a domain containing a small GPT-vanishing structure. 
In Section \ref{sec:laminate}, we detail the construction
of the proposed radial laminate for an insulated core. Section \ref{sec:main_result} is devoted to prove the proof of Theorem \ref{thm:main}, which primarily focuses on estimating the homogenization error. Section \ref{sec:arbitrary_inclusion} discusses the cloaking effect of the laminate $A_\varepsilon$ when inclusions of arbitrary conductivity are located in the core. In Section \ref{sec:Illustrative-examples}, we present numerical
examples of the proposed cloaking laminates. 
Finally, Section \label{sec:conclusion} provides concluding remarks.

\section{Preliminaries: GPT-vanishing structures}\label{sec:GPT}

\subsection{The concept of GPT-vanishing structures}
Let $H$ be a given harmonic function in $\RR^d$, where $d=2,3$. We consider the following conductivity problem for an inclusion $\Om$ embedded in an otherwise homogeneous medium:
\begin{equation}\label{eq:pb_main}
\begin{cases}
\ds\nabla\cdot \sigma\nabla u=0 \quad& \text{in }\mathbb{R}^{d},\\
\ds u(x)-H(x)=O\big(|x|^{-1}\big)\quad & \text{as }\left|x\right|\rightarrow\infty,
\end{cases}
\end{equation}
where the conductivity distribution $\sigma$ is given by $$\sigma=\sigma_{0}\,\chi_{\mathbb{R}^{d}\setminus\overline{\Omega}}
+\sigma_1\,\chi_\Omega.$$
Here, $\sigma_0$ and $\sigma_1$ are the conductivities (positive
constants) of the background medium and the inclusion $\Om$, respectively, and $\chi$ denotes the indicator function.

The presence of the inclusion $\Om$ perturbs the background field $H$, which would otherwise solve the problem in the absence of the inclusion. We have the multipole expansion \cite{Ammari:2007:PMT}:
\begin{equation}\label{eq:expansion_GPT}
\left(u-H\right)\left(x\right)=\sum_{\left|\alpha\right|=0}^{\infty}\sum_{\left|\beta\right|=0}^{\infty}\frac{\left(-1\right)^{\left|\alpha\right|}}{\alpha!\beta!}\,\partial^{\alpha}\Gamma\left(x\right)M_{\alpha\beta}\,\partial^{\beta}H\left(0\right),\quad\left|x\right|\gg 1,
\end{equation}
where $\alpha=(\alpha_1,\dots,\alpha_d)$ (and similarly for $\beta$)  is the multi-index with $|\alpha|=\alpha_1+\cdots+\alpha_d$ and $x^\alpha=x_1^{\alpha_1}\cdots x_d^{\alpha_d}$ for $x=(x_1,\dots,x_d)$. The coefficients $M_{\alpha\beta}$ are referred to as the \textit{generalized polarization tensors (GPTs)} and depend on the
geometry of the inclusion $\Omega$ and the conductivity contrast ${\sigma}_1/{\sigma_{0}}$.
The function $\Gamma\left(x\right)$ is the fundamental solution of the
Laplacian, that is,
\[
\Gamma\left(x\right)=\begin{cases}
\ds\frac{1}{2\pi}\ln\left|x\right|, & d=2,\\
\ds\frac{1}{\left(2-d\right)\omega_{d}}\left|x\right|^{2-d}, & d=3,
\end{cases}
\]
where $\omega_{d}$ denotes the area of the unit sphere in $\mathbb{R}^{d}$. 

We first discuss the two dimensional case. 
The entire harmonic function $H(x)$ admits the following expansion:
\[
H\left(x\right)=H\left(0\right)+\sum_{k=1}^{\infty}\left(a_{k}^{c}\,r^k\cos (k\theta)+a_{k}^{s}\,r^k\sin (k\theta)\right),
\]
where $x=\left(r\cos\theta,r\sin\theta\right)$.
For each $m\in\NN$, define the coefficients $b_{\alpha}^{c}$ and $b_{\alpha}^{s}$
for multi-indices $\alpha$ satisfying $\left|\alpha\right|=m$ via the identities:
\[
\sum_{\left|\alpha\right|=m}b_{\alpha}^{c}\,x^{\alpha}=r^{m}\cos (m\theta)\quad\mbox{and}\quad\sum_{\left|\alpha\right|=m}b_{\alpha}^{s}\,x^{\alpha}=r^{m}\sin (m\theta).
\]
Following \cite{Ammari:2013:ENC_1}, we define the \textit{contracted generalized polarization tensors (CGPTs)} $M_{mk}^{\text{cc}},\, M_{mk}^{\text{cs}},\, M_{mk}^{\text{sc}}$ and $M_{mk}^{\text{ss}}$ as
\begin{align*}
M_{mk}^{ij}&:=\sum_{\left|\alpha\right|=m}\sum_{\left|\beta\right|=k}b_{\alpha}^{i}\,b_{\beta}^{j}\,M_{\alpha\beta}\qquad \mbox{for all $m,k\in\NN$, $i,j=$c, s},
\end{align*}
which are called the CGPTs. Equation \eqnref{eq:expansion_GPT} becomes
\begin{equation}\label{eq:expansion_d=2}
  \begin{aligned}
\left(u-H\right)\left(x\right)= & -\sum_{m=1}^{\infty}\frac{\cos m\theta}{2\pi mr^{m}}\,\sum_{k=1}^{\infty}\left(M_{mk}^{cc}\,a_{k}^{c}+M_{mk}^{cs}\,a_{k}^{s}\right)\\
 & -\sum_{m=1}^{\infty}\frac{\sin m\theta}{2\pi mr^{m}}\,\sum_{k=1}^{\infty}\left(M_{mk}^{sc}\,a_{k}^{c}+M_{mk}^{ss}\,a_{k}^{s}\right),\qquad\left|x\right|\gg 1.
\end{aligned}  
\end{equation}
We say that the inclusion $\Om$ with conductivity distribution $\sigma$ is a \textit{GPT-vanishing structure} of order $N$ if $$M_{mk}^{ij}=0\quad\mbox{for all }m,k\leq N, \mbox{ and }i,j=\mbox{c},\mbox{s}.$$

\smallskip

Now, we consider the three-dimensional case. The entire harmonic function  {$H(x)$ admits the following expansion:
\beq\label{H:expan}
H(x)=H(0)+\sum_{k=1}^{\infty}\sum_{l=-k}^{k}a_{k}^{l}\,r^{k}\,Y_{k}^{l}\left(\theta,\phi\right),
\eeq
where $x=(r\sin\phi\cos\theta,r\sin\phi\sin\theta,r\cos\phi)$, and $Y_{k}^{l}\left(\theta,\phi\right)$ are the spherical
harmonics of degree $k$ and order $l$.
For each $m\in\NN$ and $-m\leq k\leq m$, define the coefficients $b_{\alpha}^{k}$ for multi-indices $\alpha$ satisfying $\left|\alpha\right|=m$ 
via the identities:
\[
\sum_{\left|\alpha\right|=m}b_{\alpha}^{k}\,x^{\alpha}=r^{m}\,Y_{m}^{k}\left(\theta,\phi\right).
\]
We defined the CGPTs $M_{mk}^{nl}$ as 
\[
M_{mk}^{nl}:=\sum_{\left|\alpha\right|=m}\sum_{\left|\beta\right|=k}b_{\alpha}^{n}\,b_{\beta}^{l}\,M_{\alpha\beta}\quad\mbox{for all }m,k\in\NN,\left|n\right|\le m,\left|l\right|\le k.
\]
Equation \eqnref{eq:expansion_GPT} becomes
\beq\label{multipole:3D}
\left(u-H\right)(x)=\sum_{m=1}^{\infty}\sum_{n=-m}^{m}\frac{Y_{m}^{n}\left(\theta,\phi\right)}{\left(2m+1\right)r^{m+1}}\,\sum_{k=1}^{\infty}\sum_{l=-k}^{k}M_{mk}^{nl}\,a_{k}^{l}.
\eeq
We ssay that an inclusion $\Om$ with conductivity distribution $\sigma$ is a \textit{GPT-vanishing structure} of order $N$ if $$M_{mk}^{nl}=0\quad\mbox{for all }m,k\leq N\mbox{ and all }n,l.$$

\subsection{Construction of GPT-vanishing structures with multicoated balls} 
\label{subsec:GPT_vanishing}

Following the construction in \cite{Ammari:2013:ENC_1,arXiv:2412.18809}, we model the GPT-vanishing structure $\Om$ as a ball with a spherical core and radially symmetric coating layers. We set $\Omega=B_{r_1}(0)\subset\RR^d$ for $d=2,3$, and assume that the coating consists of $L$ concentric layers ($L\in\NN$) with radii ordered $0<r_{L+1}<r_{L}<...<r_{1}$. Let $A_0$ denote the exterior of the structure, $A_{L+1}$ the core, and $A_1,\dots,A_L$ the coating layers; explicitly, 
\beq\label{eq:coating_set}
\begin{aligned}
 A_{0}&=\mathbb{R}^{d}\setminus\overline{\Omega},\qquad 
  A_{L+1}=B_{r_{L+1}}(0),\\
 A_{j}&=\left\{x\in\RR^d\,:\, r_{j+1}<r\le r_{j}\right\}\quad\mbox{for }j=1,...,L.
\end{aligned}
\eeq
We normalize the background conductivity to $\sigma_{0}=1$, and 
let $\sigma_{j}$ denote the constant conductivity in $A_{j}$ for $j=1,...,L+1$.
We set the overall conductivity distribution as
\begin{equation}\label{eq:coating_sigma:general}
\sigma=\sigma_{0}\,\chi_{\mathbb{R}^{d}\setminus\overline{\Omega}}+\sigma_{L+1}\chi_{A_{L+1}}+\sum_{j=1}^{L}\sigma_{j}\chi_{A_{j}}.
\end{equation}
We assume $\sigma_j>0$ for $j=1,\dots,L$, and $\sigma_{L+1}\geq 0$, where $\sigma_{L+1}=0$ corresponds to an insulating core. 

\subsubsection{Two-dimensional case} 
For the radially symmetric structure $\Om$ with the conductivity distribution $\sigma$ given by \eqnref{eq:coating_sigma:general}, 
we have $M_{kk}^{cc}=M_{kk}^{ss}$ for $k=1,2,\dots$, and all other CGPT terms vanish. Hence, the expansion in \eqnref{eq:expansion_d=2} reduces to 
\begin{equation}\label{eq:expansion_d=2_simple}
 \left(u-H\right)\left(x\right)=	-\sum_{k=1}^{\infty}\frac{M_{k}}{2\pi kr^{k}}\left(a_{k}^{c}\,\cos (k\theta) +a_{k}^{s}\,\sin (k\theta)
 \right),\qquad\left|x\right|\gg1,   
\end{equation}
where $M_{k}:=M_{kk}^{cc}=M_{kk}^{ss}$.
The inclusion $\Om$ with the conductivity distribution $\sigma$ is a GPT-vanishing structure of order $N$ if $$M_k=0\quad\mbox{for all }k\leq N.$$

 Ammari et al. \cite{Ammari:2013:ENC_1} provided a characterization of GPT-vanishing structures in the radially symmetric setting, together with numerical examples. We briefly state the result below and refer the reader to that paper for a detailed discussion (see also \cite{arXiv:2412.18809}).

\begin{prop}[\cite{Ammari:2013:ENC_1}]\label{prop:radial:2D} Let $N$ be a positive integer. 
Suppose there exist nonzero constants $\lambda_{1},\ldots,\lambda_{L+1}$ with
$|\lambda_{j}|<1$, and radii $0<r_{L+1}<r_{L}<...<r_{1}$
such that, for every $k=1,\dots,N$,
\[
\prod_{j=1}^{L+1}
\begin{pmatrix}
1 & \lambda_{j}r_{j}^{-2k}\\
\lambda_{j}r_{j}^{2k} & 1
\end{pmatrix}=
\begin{pmatrix}
* & *\\
0 & *
\end{pmatrix},
\quad \mbox{i.e., an upper triangluar matrix}.
\]
Here, the matrix product is taken in left-multiplicative order, that is, $\prod_{j=1}^{L+1}A_j=A_{L+1}\cdots A_1$.
Then the radially symmetric domain $\Om$ with conductivity distribution $\sigma$ defined in (\ref{eq:coating_set})
and (\ref{eq:coating_sigma:general}), whose conductivity values $\sigma_0,\dots,\sigma_{L+1}$ satisfies
\begin{equation}\label{GPT_vanishing_d=2}
  \frac{\sigma_{j}-\sigma_{j-1}}{\sigma_{j}+\sigma_{j-1}}=\lambda_{j},\quad j=1,...,L+1,  
\end{equation}
forms a GPT-vanishing structure of order $N$, that is, $M_{k}=0$ for all
$k\le N$. 
\end{prop}

This characterizations also extends to the case of an insulating core, that is, when $\sigma_{L+1}=0$ and $\lambda_{L+1}=-1$.

\subsubsection{Three-dimensional case} \label{subsub:3D}
An analogous of Proposition~\ref{prop:radial:2D} holds in three dimensions \cite{arXiv:2412.18809}. For completeness, we present the derivation below; the two-dimensional cases follows by the similar argument.

For a radially symmetric structure $\Om$ in three dimensions with the conductivity distribution $\sigma$ given by (\ref{eq:coating_set})
and (\ref{eq:coating_sigma:general}), the CGPTs satisfy the following properties:
\begin{align*}
&M_{mk}^{nl}=0 \ \quad\qquad\text{if }m\ne k\mbox{ or }n\ne l,\\
&M_{kk}^{nn}=M_{kk}^{ll}      \qquad\text{for all } k,\mbox{ and for all }n,l.
\end{align*}
Then the far-field expansion \eqnref{multipole:3D} reduces to
\beq\label{eq:expansion}
\left(u-H\right)(x)=\sum_{k=1}^{\infty}\sum_{l=-k}^{k}M_{k}\,a_{k}^{l}\,\frac{Y_{k}^{l}\left(\theta,\phi\right)}{\left(2k+1\right)r^{k+1}},
\eeq
where $M_{k}:=M_{kk}^{ll}$, which are independent of $l$, and $a_k^l$ are expansion coefficients of $H$ as in \eqnref{H:expan}.

Thanks to the linearity of \eqnref{eq:pb_main} with respect to the background field $H$, it suffices to consider the problem for each basis function in the spherical harmonic expansion. 
Specifically, for each fixed $k\in\NN$, we consider the harmonic background field
$$H(x)=r^k\,Y_k^l(\theta,\phi), \quad l=-k,\dots,k,$$
that is, $a_k^l=1$ and other coefficients in \eqnref{H:expan} are zero.
Then the multipole expansion with the CGPTs given in \eqnref{eq:expansion} leads to 
$$u(x)=r^k\,Y_k^l(\theta,\phi)+\frac{M_k}{2k+1}r^{-k-1}Y_k^l(\theta,\phi),\quad |x|\gg 1.$$

Since $\sigma$ is radially symmetric, we can express $u$ in the form

\beq \label{u:H_k}
u(x)=\left(a_{j}\,r^{k}+ b_{j}\,r^{-k-1}\right)Y_{k}^{l}(\theta,\phi)\quad\mbox{ in }A_j,\ 0\leq j\leq L+1,
\eeq
where
\begin{align} \label{b0:cgpt}
\begin{cases}
\ds a_0=1,\quad b_{0}=\frac{M_k}{2k+1}\quad &\mbox{in }A_0\,(=\RR^3\setminus\overline{\Om}),\\[1mm]
\ds b_{L+1}=0\quad &\mbox{in }A_{L+1}, \mbox{ which is the core}.
\end{cases}    
\end{align}

It holds that
\begin{align*}
r^{k}+\frac{M_{k}}{\left(2k+1\right)r^{k+1}}&\ne 0, \\
kr^{k}-(k+1)\,\frac{M_{k}}{\left(2k+1\right)r^{k+1}}
&\ne0\quad\mbox{for all }r\geq r_1,
\end{align*}
since otherwise either $u(x)=0$ or $\pd{u}{n}(x)=0$ on $|x|=r$, which contradict the fact that $u$ is a nontrivial solution to the elliptic equation \eqnref{eq:pb_main}. 
Here, $n$ denotes the outward unit normal vector.
In particular, we have
\begin{equation}
\mathinner{\left|M_{k}\right|}{\le\left(2k+1\right)r_{1}^{2k+1}}.\label{eq:bound_for_M_kl}
\end{equation}

Now, we consider the transmission conditions for $u$ given in \eqnref{u:H_k}: for $j=1,\dots,L+1$,
\begin{align*}
u_{k}\big|^+&=u_{k}\big|^-,\\
\sigma_{j-1}\frac{\partial u_{k}}{\partial n}\Big|^+&=\sigma_j\frac{\partial u_{k}}{\partial n}\Big|^-\quad\mbox{on } \left\{ r=r_{j}\right\}.
\end{align*}
This leads to the matching conditions:
\[
\begin{aligned}
&\begin{bmatrix}
r_{j}^{k} & {r_{j}^{-k-1}}\\
\sigma_{j}kr_{j}^{k-1} & -{\sigma_{j}\left(k+1\right)}{r_{j}^{-k-2}}
\end{bmatrix}
\begin{pmatrix}
a_{j}\\[1mm]
b_{j}
\end{pmatrix}
=\begin{bmatrix}
r_{j}^{k} & {r_{j}^{-k-1}}\\
\sigma_{j-1}kr_{j}^{k-1} & -{\sigma_{j-1}\left(k+1\right)}{r_{j}^{-k-2}}
\end{bmatrix}
\begin{pmatrix}
a_{j-1}\\[1mm]
b_{j-1}
\end{pmatrix}
\end{aligned}.
\]
Equivalently, this yields the recurrence relation:
\begin{align}\notag
\begin{pmatrix}
a_{j}\\[1mm]
b_{j}
\end{pmatrix}
&=\frac{1}{\left(2k+1\right)\sigma_{j}}
\begin{bmatrix}
k\sigma_{j-1}+\left(k+1\right)\sigma_{j} & \left(k+1\right)\left(\sigma_{j}-\sigma_{j-1}\right)r_{j}^{-\left(2k+1\right)}\\
k\left(\sigma_{j}-\sigma_{j-1}\right)r_{j}^{2k+1} & \left(k+1\right)\sigma_{j-1}+k\sigma_{j}
\end{bmatrix}
\begin{pmatrix}
a_{j-1}\\[1mm]
b_{j-1}
\end{pmatrix}.
\end{align}

By successively applying this relation for $j=1,\dots,L+1$, we obtain the global relation
\begin{equation}\label{eq:GPT_charac}
\begin{pmatrix}
a_{L+1}\\[1mm]
0
\end{pmatrix}
=\mathbf{P}^{\left(k\right)}
\begin{pmatrix}
1\\
b_{0}
\end{pmatrix},
\end{equation}
where the matrix $\mathbf{P}^{\left(k\right)}\in\RR^{2\times 2}$ is defined by
\begin{equation}\notag
\mathbf{P}^{\left(k\right)}
:=\prod_{j=1}^{L+1}P_j^{(k)}
=\begin{bmatrix}
{p}^{(k)}_{11} & {p}^{(k)}_{12}\\[2mm]
{p}^{(k)}_{21} & {p}^{(k)}_{22}
\end{bmatrix},
\end{equation}
with each interface matrix $P_{j}^{(k)}$ given by
\begin{equation}
P_{j}^{(k)}:=\frac{1}{\left(2k+1\right)\sigma_{j}}
\begin{bmatrix}
k\sigma_{j-1}+\left(k+1\right)\sigma_{j} & \left(k+1\right)\left(\sigma_{j}-\sigma_{j-1}\right)r_{j}^{-\left(2k+1\right)}\\
k\left(\sigma_{j}-\sigma_{j-1}\right)r_{j}^{2k+1} & \left(k+1\right)\sigma_{j-1}+k\sigma_{j}
\end{bmatrix}\label{eq:GPT_matrix}. 
\end{equation}
From \eqnref{b0:cgpt} and the second component of \eqnref{eq:GPT_charac}, we deduce the expression for the CGPTs $M_k$: 
\beq\label{M_k:b_0}
M_k=(2k+1)b_{0}=-(2k+1)\frac{{p}_{21}^{(k)}}{{p}_{22}^{(k)}}.
\eeq
Similar to the two dimensional result, $M_{k}=0$ if and only if $\mathbf{P}^{\left(k\right)}$ is an upper triangular matrix. In other words, the conductivity profile $\sigma$ given by \eqnref{eq:coating_sigma:general} is a GPT-vanishing structure of order $N$ if 
\begin{equation}\label{GPT_vanishing:character}
  {p}^{(k)}_{21}=0 \quad\text{for all }k\le N.  
\end{equation}

We now state a characterization of a GPT-vanishing structure in $d=3$.

\begin{prop}\label{prop:radial:3D}
    Let $L$ be a given positive integer. Suppose there exist positive constants $\sigma_{1},\ldots,\sigma_{L+1}$ and $0<r_{L+1}<r_{L}<...<r_{1}$ such that, for all $k=1,\dots,L$, 
    \begin{equation*}   \mathbf{P}^{\left(k\right)}:=\prod_{j=1}^{L+1}P_{j}^{(k)}=\begin{bmatrix}p_{11}^{(k)} & p_{12}^{(k)}\\[2mm]
p_{21}^{(k)} & p_{22}^{(k)}
\end{bmatrix}
\end{equation*}
with 
\begin{equation*}
     P_{j}^{(k)}:=\frac{1}{\left(2k+1\right)\sigma_{j}}\begin{bmatrix}k\sigma_{j-1}+\left(k+1\right)\sigma_{j} & \left(k+1\right)\left(\sigma_{j}-\sigma_{j-1}\right)r_{j}^{-\left(2k+1\right)}\\
k\left(\sigma_{j}-\sigma_{j-1}\right)r_{j}^{2k+1} & \left(k+1\right)\sigma_{j-1}+k\sigma_{j}
\end{bmatrix}
\end{equation*}
is an upper triangular matrix for all $k\le L$. That is, $p_{21}^{(k)}=0$ for all $k\le L$. Then the radially symmteric domain $\Omega$  with conductivity distribution $\sigma$  defined in \eqnref{eq:coating_set} and \eqnref{eq:coating_sigma:general} forms a GPT-vanishing structure of order $L$, that is, $M_{k}=0$ for all $k\le L$.
\end{prop} 

This characterization can be extended to to the case of an insulating core, that is, when $\sigma_{L+1}=0$. In this case, instead of using $P_{L+1}^{(k)}$ as defined in \eqnref{eq:GPT_matrix}, we take
$$P_{L+1}^{(k)}=
\begin{bmatrix}
0&0\\[1mm]
-k\,r_{L+1}^{2k+1}&k+1
\end{bmatrix}
\quad\mbox{when }\sigma_{L+1}=0.
$$

\section{Boundary value problem with a small GPT-vanishing structure}\label{sec:BVP}
The near-cloaking scheme based on the blow-up transformation of a ball of radius $\rho$ achieves an invisibility level of order $\rho^{d}$ \cite{Kohn:2008:CCV}, where $\rho$ is a small parameter.
This scheme is implemented by applying the pushforward of the conductivity distribution with a small inclusion via the blow-up transformation (see Section~\ref{sec:laminate} for details).
Incorporating a GPT-vanishing structure of order $N$ further enhances the cloaking effect, yielding an invisibility level of order $\rho^{d+2N}$ \cite{Ammari:2013:ENC_1}. This improvement follows from the small-volume expansion for the scaled GPT-vanishing structures.} 

We derive the small-volume expansion for the boundary value problem with a GPT-vanishing structure in three dimensions; the two-dimensional case is analogous.

\subsection{Small volume expansion for the Dirichlet-to-Neumann map}
For a domain $D$ and a conductivity distribution $\gamma$, we denote by $\Lambda_{D}\left[\gamma\right]:H^{\frac{1}{2}}(\p D)\to H^{-\frac{1}{2}}(\p D)$ the Dirichlet-to-Neumann (DtN) map, that is,
$$\Lambda_{D}\left[\gamma\right](\varphi)=\gamma \frac{\p v}{\p n}\Big|_{\p D}\quad\mbox{for }\varphi\in H^{\frac{1}{2}}(\p D),$$
where $n$ means the outward normal vector to $\p\Omega$ and $v$ is the solution to the problem
\beq\notag%\label{eq:bvp_main}
\begin{cases}
\nabla\cdot\gamma\nabla v=0 \quad& \text{in }D,\\
v=\varphi \quad& \text{on }\partial D.
\end{cases}
\eeq

Let $\sigma$ denote the GPT-vanishing structure 
of the form given in \eqnref{eq:coating_sigma:general} and constructed in Section \ref{subsec:GPT_vanishing}.
As the main focus of this section, we estimate the DtN map of $\sigma\circ\Psi_{1/\rho}$ on $B_1$,
where $\rho>0$ is a small parameter and $\Psi_{1/\rho}$ is the scaling transformation given by
$$\Psi_{{1}/{\rho}}(x)=\frac{1}{\rho}x.$$
The following relation is useful to derive the expansion for the DtN map of $\sigma\circ\Psi_{1/\rho}$:
\begin{lemma}\label{lemma:eq:dtn}
Fix $s\in(\frac{1}{2},1)$. For $f\in H^{\frac{1}{2}}(\p B_s)$, we have
\begin{equation}\label{eq:dtn}
\Lambda_{B_{\frac{s}{\rho}}}\left[\gamma\right]\left(f\circ\Psi_{\rho}\right)=\rho\,\Lambda_{B_{s}}\big[\gamma\circ\Psi_{{1}/{\rho}}\big]\left(f\right)\circ\Psi_{\rho}\quad \mbox{on }\p B_{\frac{s}{\rho}}.
\end{equation}
\end{lemma}
\begin{proof}
Let $v$ be the solution to the boundary value problem
\begin{equation}\label{eq:u in B_1}
\begin{cases}
\nabla\cdot\left(\gamma\circ\Psi_{\frac{1}{\rho}}\right)\nabla v=0 & \text{in }B_{s},\\
v=f & \text{on }\p B_s.
\end{cases}
\end{equation}
Define $\tilde{v}$ to be the solution to the following problem:
\begin{equation}\notag
\begin{cases}
\nabla\cdot\gamma\nabla\tilde{v}=0 & \text{in }B_{\frac{s}{\rho}},\\
\tilde{v}=f\circ\Psi_{\rho} & \text{on }\p B_{\frac{s}{\rho}},
\end{cases}
\end{equation}
where $\Psi_\rho(s)=\rho x$ is the inverse of $\Psi_{1/\rho}$. 
Then, by a change-of-variable argument, we have the relation
\[
\frac{\partial\tilde{v}}{\partial n}\Big|_{\p B_{\frac{s}{\rho}}}(y)=\rho\,\frac{\partial v}{\partial n}\Big|_{\p B_s}\left(\rho\, y\right),
\]
which implies \eqnref{eq:dtn}.
\end{proof}

Fix $s\in(\frac{1}{2},1)$. 
To simplify the computation, we consider a basis function of the form
\beq\label{def:f_k:spherical}
f_k\left(\theta,\phi\right)=Y_{k}^{l}\left(\theta,\phi\right)\qquad\mbox{for }k\in\NN,\ l=-k,\dots,k.
\eeq
Since $(\frac{r}{s})^{k}Y_{k}^{l}(\theta,\phi)$ is the harmonic function with the Dirichlet data $f_k$ on $\p  B_s$, one can easily find that 
the DtN map for the identity conductivity in $B_s$ is
\beq\label{DtN:identity}
\Lambda_{B_{s}}[1]\left(f_{k}\right)=\frac{k}{s}\, Y_{k}^{l}\left(\theta,\phi\right).
\eeq

We first compute the DtN map $\Lambda_{B_{s}}\left[\sigma\circ\Psi_{{1}/{\rho}}\right]$.
In view of Lemma~\ref{lemma:eq:dtn}, we consider  the problem 
\begin{equation}\label{eq:pB_tilde}
\begin{cases}
\nabla\cdot\sigma\nabla\tilde{u}=0 & \text{in }B_{\frac{s}{\rho}},\\
\tilde{u}=f_k\circ\Psi_{\rho} & \text{on }\p B_{\frac{s}{\rho}},
\end{cases}
\end{equation}
where $f_k$ is given by \eqnref{def:f_k:spherical}. 
Note that $\tilde{u}$ takes the form
\[
\tilde{u}=\left(a_{j}\,r^{k}+b_{j}\,r^{-(k+1)}\right)Y_{k}^{l}\left(\theta,\phi\right)\quad\text{ in }A_{j},\ j=0,1,...,N+1,
\]
where the coefficients $a_j$ and $b_j$ are determined by the transmission conditions on $\p A_j$. In particular, from the boundary condition on $\p B_{\frac{s}{\rho}}$ and the regularity in the core, $A_{N+1}$, for $\tilde{u}$, the outmost and innermost coefficients satisfy the relations:
\begin{align}\label{eq:coeff_for_0}
 & a_{0}\left(\frac{s}{\rho}\right)^{k}+b_{0}\left(\frac{s}{\rho}\right)^{-(k+1)}=1,\quad b_{N+1}=0.
\end{align}
Furthermore, by \eqnref{eq:expansion}, the CGPTs $M_k=M_k(\sigma)$ satisfies
\begin{equation}\label{eq:b_0_intermsof_a_0}
b_{0}=\frac{a_0}{2k+1}M_{k}.
\end{equation}

Solving the system (\ref{eq:coeff_for_0}) and (\ref{eq:b_0_intermsof_a_0}), we can express $b_0$ in terms of $M_k$ as 
\begin{equation}\label{eq:b_0}
b_{0}=\frac{M_{k}\rho^{k}s^{k+1}}{M_{k}\rho^{2k+1}+\left(2k+1\right)s^{2k+1}}.
\end{equation}
By \eqnref{DtN:identity} and \eqnref{eq:coeff_for_0}, we obtain that
\begin{align*}
\left(\Lambda_{B_{\frac{s}{\rho}}}\left[\sigma\right]\left(f_{k}\circ\Psi_{\rho}\right)\right)(\theta,\phi) & =\frac{\partial\tilde{u}}{\partial r}\Big|_{\p B_{\frac{s}{\rho}}}\left(\theta,\phi\right)\\
 & =k\,a_{0}\left(\frac{s}{\rho}\right)^{k-1}Y_{k}^{l}\left(\theta,\phi\right)-\left(k+1\right)b_{0}\left(\frac{s}{\rho}\right)^{-(k+2)}Y_{k}^{l}\left(\theta,\phi\right).\\
 & =k\,\frac{\rho}{s}\,Y_{k}^{l}\left(\theta,\phi\right)-\left(2k+1\right)b_{0}\left(\frac{s}{\rho}\right)^{-(k+2)}Y_{k}^{l}\left(\theta,\phi\right)\\
 &=\rho \,\Lambda_{B_{s}}[1]\left(f_{k}\right)-\left(2k+1\right)b_{0}\left(\frac{s}{\rho}\right)^{-(k+2)}Y_{k}^{l}\left(\theta,\phi\right).
\end{align*}
This together with (\ref{eq:dtn}) leads to 
\begin{align*}
\Lambda_{B_{s}}\big[\sigma\circ\Psi_{{1}/{\rho}}\big]\left(f_{k}\right) & =\Lambda_{B_{s}}\left[1\right]\left(f_{k}\right)-\left(2k+1\right)b_{0}\,\frac{\rho^{k+1}}{s^{k+2}}\,Y_{k}^{l}\left(\theta,\phi\right),
\end{align*}
where $b_0$ satisfies \eqnref{eq:b_0}.

In general, for $f\in H^{\frac{1}{2}}(\p B_s)$ with spherical harmonic expansion 
\beq\label{f:harmonic:expan}
f=\sum_{k=1}^{\infty}\sum_{l=-k}^{k}c_{k}^{l}\, Y_{k}^{l}\left(\theta,\phi\right),
\eeq
linearity of the DtN map together with \eqnref{eq:b_0} yields
\begin{align*}
\left(\Lambda_{B_{s}}\big[\sigma\circ\Psi_{{1}/{\rho}}\big]-\Lambda_{B_{s}}\left[1\right]\right)\left(f\right) & =-\sum_{k=1}^{\infty}\sum_{l=-k}^{k}\frac{\left(2k+1\right)M_{k}\,\rho^{2k+1}}{M_{k}\rho^{2k+1}+\left(2k+1\right)s^{2k+1}}\frac{1}{s}\,c_{k}^{l}\,Y_{k}^{l}\left(\theta,\phi\right).
\end{align*}

For the GPT-vanishing structure of order $N$, we have $M_{k}=0$
for all $k\le N$. For $k\geq N+1$, by (\ref{eq:bound_for_M_kl}), 
\[
\left|M_{k}\right|\rho^{2k+1}\le\left(2k+1\right)\left(\rho r_{1}\right)^{2k+1}.
\]
Hence, given that $\rho\ll 1$ and $s\in(\frac{1}{2},1)$, we have
\begin{align}\notag
\left|\frac{ M_{k}\rho^{2k+1}}{M_{k}\rho^{2k+1}+\left(2k+1\right)s^{2k+1}}\right| & \le\frac{\left(2k+1\right)^{}\left(r_1 \rho \right)^{2k+1}}{\left(2k+1\right)s^{2k+1}-\left(2k+1\right)\left(r_{1}\rho\right)^{2k+1}}\\\notag
 &= \frac{\left(r_1 \rho \right)^{2N+3}}{s^{2N+3}\, (\frac{s}{r_1\rho})^{2k+1-2N-3}-\left(r_{1}\rho\right)^{2N+3}} \\\label{Mk:bound:rho_N}
 & \le C\rho^{2N+3},
\end{align}
 where $C$ is independent of $\rho$ and $k$. 

Recall that $\left\Vert g\right\Vert^2 _{H^{\frac{1}{2}}\left(\partial B_{s}\right)} \simeq \|g\|^2_{L^2(\p B_s)}+\underset{v|_{\p B_s}=g}{\min}\int_{B_s} |\nabla v|^2dx$ and this minimum is attained by the harmonic function $G$ on $B_s$ that solves $\Delta G=0$ in $B_s$ and $G=g$ on $\partial B_s$. Here, $\simeq$ denotes equivalence of norms. For $g$ in the form of \eqnref{f:harmonic:expan} with the coefficients $d_k^l$, we have
 \begin{equation*}
     G(r,\theta,\phi)=\sum_{k=1}^{\infty}\sum_{l=-k}^{k}d_{k}^{l}\left(\frac{r}{s}\right)^k Y_{k}^{l}\left(\theta,\phi\right)
 \end{equation*}
 and, using orthonormal properties of the spherical harmonics, $\lVert \nabla G \rVert^2_{L^2(B_s)}=\int_{\p B_s}G\,\partial_n G \, dS=\sum_{k=1}^{\infty}\sum_{l=-k}^{k} ks \lvert d_k^l \rvert^2$.
 Hence 
 $$c_2\sum_{k=1}^{\infty}\sum_{l=-k}^{k} k \lvert d_k^l \rvert^2\leq \|g\|^2_{H^{\frac{1}{2}}(\p B_s)}\leq c_1\sum_{k=1}^{\infty}\sum_{l=-k}^{k} k \lvert d_k^l \rvert^2$$ 
 for some constants $c_1,c_2>0$ depending on $s$, but not depending on $g$. 
This combined with \eqnref{Mk:bound:rho_N} leads to
\begin{align}\notag
 & \left|\left\langle \left(\Lambda_{B_{s}}\big[\sigma\circ\Psi_{{1}/{\rho}}\big]-\Lambda_{B_{s}}\left[1\right]\right)\left(f\right),\, g\right\rangle _{H^{-\frac{1}{2}}\left(\partial B_{s}\right),\,H^{\frac{1}{2}}\left(\partial B_{s}\right)}\right|\\\notag
 \leq & \, C \sum_{k=1}^\infty\sum_{l=-k}^k \frac{2k+1}{s}\left|\frac{ M_{k}\rho^{2k+1}}{M_{k}\rho^{2k+1}+\left(2k+1\right)s^{2k+1}}\right| |c_k^l\, d_k^l|\\ \notag
 \leq &\, C\rho^{2N+3}\sum_{k,l}k|c_k^l|\, |d_k^l|\\\label{DtN:opperator}
  \le &\, C\rho^{2N+3}\left\Vert f\right\Vert _{H^{\frac{1}{2}}\left(\partial B_{s}\right)}\, \left\Vert g\right\Vert _{H^{\frac{1}{2}}\left(\partial B_{s}\right)},
\end{align}
where $C$ is independent of $\rho$ and $f$.

Now, we consider the boundary value problem
\beq\label{eq:bvp_main}
\begin{cases}
\nabla\cdot\left(\sigma\circ\Psi_{{1}/{\rho}}\right)\nabla v=0 \quad& \text{in }B_1,\\
v=\varphi \quad& \text{on }\partial B_1,\\
\frac{\partial v}{\partial n }=0  \quad& \text{on } \partial B_\rho
\end{cases}
\eeq
for a given $\varphi\in H^{\frac{1}{2}}(\p B_1)$. 
By the Sobolev trace theorem and the $H^{1}$ energy estimate, there exists $C$ independent of $\rho$ (but may depend on $s$) such that
\begin{equation}
\left\Vert v|_{\p B_s}\right\Vert _{H^{\frac{1}{2}}\left(\p B_s\right)}
\le C\left\Vert v\right\Vert _{H^{1}\left(B_{1}\setminus \overline{B_{s}}\right)}\le C\left\Vert v\right\Vert _{H^{1}\left(B_{1}\setminus\overline{B_\rho}\right)} \le C\left\Vert \varphi\right\Vert_{H^{\frac{1}{2}}\left(\p B_1\right)}.\label{eq:f_le_varphi}
\end{equation}
Consequently, 
\begin{equation}\label{eq:cloaking_gpt_prelim}
\left\langle \left(\Lambda_{B_{s}}\big[\sigma\circ\Psi_{{1}/{\rho}}\big]-\Lambda_{B_{s}}\left[1\right]\right)(v|_{\p B_s}),\,v|_{\p B_s}\right\rangle _{H^{-\frac{1}{2}}\left(\partial B_{s}\right),\,H^{\frac{1}{2}}\left(\partial B_{s}\right)}\le C\rho^{2N+3}\left\Vert \varphi\right\Vert^2_{H^{\frac{1}{2}}\left(\p B_1\right)}.
\end{equation}
In what follows, we estimate the left-hand side in terms of $\varphi$. 

For $\varphi$ and $v$ given in \eqnref{eq:bvp_main}, we let $U_\rho$ and $U$ be respectively solutions to 
\beq\notag
\begin{cases}
\ds\Delta U_\rho=0 \quad& \text{in }B_s,\\
\ds U_\rho=v|_{\p B_s} \quad& \text{on }\partial B_s
\end{cases}
\eeq
and
\beq\notag
\begin{cases}
\ds \Delta U=0 \quad& \text{in }B_1,\\
\ds U=\varphi \quad& \text{on }\partial B_1.
\end{cases}
\eeq

We then define $w\in H^{1}\left(B_{1}\right)$ by
\beq\notag
w=\begin{cases}
U_{\rho} & \text{in }B_{s},\\
v & \text{in }B_{1}\setminus B_{s}.
\end{cases}
\eeq
Then, we have
\begin{equation}\label{DtN:varphi}
\begin{aligned} 
&  \left\langle \Lambda_{B_{1}}\big[\sigma\circ\Psi_{{1}/{\rho}}\big]\left(\varphi\right)- \Lambda_{B_{1}}\left[1\right]\left(\varphi\right),\,\varphi\right\rangle_{H^{-\frac{1}{2}}(\p B_1),\,H^{1/2}(\p B_1)}  \\ 
  =\,&\int_{B_{1}}\left(\sigma\circ\Psi_{{1}/{\rho}}\right)\nabla v\cdot\nabla v\, dx-\int_{B_{1}}\nabla U\cdot\nabla U\, dx\\ 
  =\,&\int_{B_{s}}\big(\sigma\circ\Psi_{{1}/{\rho}}\big)\nabla v\cdot\nabla v\, dx
  +\int_{B_1\setminus B_{s}}\nabla v\cdot\nabla v\, dx
  -\int_{B_{1}}\nabla U\cdot\nabla U\, dx\\ 
  =\,&A +B
  \end{aligned}
  \end{equation}
with 
\begin{align*}
    A:=&\left\langle \left(\Lambda_{B_{s}}\big[\sigma\circ\Psi_{{1}/{\rho}}\big]-\Lambda_{B_{s}}\left[1\right]\right)\left(v|_{\p B_s}\right),\, v|_{\p B_s}\right\rangle_{H^{-\frac{1}{2}}(\p B_s),H^{1/2}(\p B_s)},\\
B:= &\int_{B_{1}}\nabla w\cdot\nabla w\, dx-\int_{B_{1}}\nabla U\cdot\nabla U\, dx.
\end{align*}

We estimate $A$ and $B$. By \eqnref{eq:cloaking_gpt_prelim}, it holds that
\beq\label{estimate:A}
 |A|\leq C\rho^{2N+3}\left\Vert v|_{\p B_s}\right\Vert _{H^{\frac{1}{2}}\left(\partial B_{s}\right)}^{2}
\eeq
Since $\left(w-U\right)\in H_{0}^{1}\left(B_{1}\right)$, we have $2\int_{B_1}\nabla(w-U)\cdot\nabla U dx=0$. By subtracting this term from $A$, we obtain 
\begin{align*}
B=&\int_{B_{1}}\nabla \left(w-U\right)\cdot\nabla\left(w-U\right)
=\int_{B_{1}}\nabla w\cdot\nabla\left(w-U\right).
\end{align*}
Note that, because of $\rho\ll 1$, 
$$(\sigma\circ\Psi_{1/\rho})(x)=\sigma_0=1\quad\mbox{in an open set containing }B_1\setminus B_s.$$
By applying the Green's identity and continuity of flux, we derive
\begin{align*}\label{eq:energy_w-u}
B =&\int_{B_{s}}\nabla U_{\rho}\cdot\nabla\left(w-U\right)+\int_{B_{1}\setminus B_{s}}\nabla v\cdot\nabla\left(w-U\right)\\
=&\int_{\partial B_{s}}\frac{\partial\left(U_{\rho}\right)^{-}}{\partial\nu}\left(w-U\right)-\int_{\partial B_{s}}\frac{\partial v^{+}}{\partial\nu}\left(w-U\right)\\
 =&\int_{\partial B_{s}}\frac{\partial\left(U_{\rho}-v\right)^{-}}{\partial\nu}\left(w-U\right)\\
=&\left\langle \left(\Lambda_{B_{s}}\big[\sigma\circ\Psi_{{1}/{\rho}}\big]-\Lambda_{B_{s}}\left[1\right]\right)\left(v|_{\p B_s}\right),\,w-U\right\rangle _{H^{-\frac{1}{2}}(\p B_s),\,H^{1/2}(\p B_s)}.
\end{align*}

Consequently, by \eqnref{DtN:opperator}

\begin{align*}
B & \le C\rho^{2N+3} 
\left\Vert v|_{\p B_s}\right\Vert _{H^{1/2}(\p B_s)}
\left\Vert w-U\right\Vert _{H^{1/2}(\p B_s)}.
\end{align*}
By using the fact that $B=\left\Vert \nabla\left(w-U\right)\right\Vert^2_{L^{2}\left(B_{1}\right)}$ and $\left(w-U\right)\in H_{0}^{1}\left(B_{1}\right)$, 
we argue by the Poincaré inequality to deduce that

\begin{align*}
\sqrt{B}=\left\Vert \nabla\left(w-U\right)\right\Vert_{L^{2}\left(B_{1}\right)} 
\le C\rho^{2N+3} 
\left\Vert v|_{\p B_s}\right\Vert _{H^{1/2}(\p B_s)}
\end{align*}
and, thus
\beq\label{estimate:B}
B\le C \rho^{2(2N+3)}\|v|_{\p B_s}\|_{H^{\frac{1}{2}}(\p B_s)}^2.
\eeq

From \eqnref{eq:f_le_varphi}, \eqnref{DtN:varphi}, \eqnref{estimate:A} and \eqnref{estimate:B}, we obtain
\begin{equation}
\left|\left\langle \Lambda_{B_{1}}\big[\sigma\circ\Psi_{{1}/{\rho}}\big]\left(\varphi\right)-\Lambda_{B_{1}}\left[1\right]\left(\varphi\right),\varphi\right\rangle_{H^{-\frac{1}{2}}(\p B_1),\,H^{1/2}(\p B_1)}  \right|
\le C\rho^{2N+3}\left\Vert \varphi\right\Vert^2_{H^{\frac{1}{2}}\left(\partial B_{1}\right)}.\label{eq:cloaking_gpt_fin}
\end{equation}
Let us the polarization identity for an operator $\Lambda$ defined between Hilbert spaces $H_1$ and $H_2$: for any $\varphi, \psi \in H_1$ and $t>0$,
\begin{equation*}
\begin{aligned}
4\left\langle\Lambda\varphi,\psi\right\rangle
=\left\langle\Lambda\left(t\psi+t^{-1}\varphi\right),\, t\psi+t^{-1}\varphi\right\rangle
-\left\langle\Lambda\left(t\psi-t^{-1}\varphi\right),\, t\psi-t^{-1}\varphi\right\rangle.
\end{aligned}
\end{equation*}
By applying this relation to \eqnref{eq:cloaking_gpt_fin} with $t^{2}=\left\Vert \varphi\right\Vert {}_{H^{\frac{1}{2}}\left(\partial B_{1}\right)}/\left\Vert \psi\right\Vert {}_{H^{\frac{1}{2}}\left(\partial B_{1}\right)}$, 
we deduce that
\begin{equation*}
    \left\Vert \Lambda_{B_{1}}\big[\sigma\circ\Psi_{{1}/{\rho}}\big]-\Lambda_{B_{1}}\left[1\right]\right\Vert\le C\rho^{2N+3}.
\end{equation*}

\section{Design of a laminate cloaking structure for an insulated core}\label{sec:laminate}

In this section, we design a near-cloaking structure consisting of finitely many layers of isotropic materials, built from multi-layer GPT-vanishing constructions with a perfectly insulating core. To make this precise, we introduce the DtN map for domains containing a perfectly insulating subdomain.  
Let $\Om\subset\RR^d(d=2,3)$ and let $B\subset\Om$ be an insulating subdomain. We denote by $\Lambda_{\Omega,B}\left[\sigma\right]$ the DtN map on $\Omega$ associated with the conductivity distribution $\sigma$, in order to emphasize the insulating inclusion $B$. 

As shown in \cite{Kohn:2008:CCV}, if $F:\overline{\Om}\to\overline{\Om}$ is a $C^1$ diffeomorphism  with $F(x)=x$ on $\p \Om$ and $F(B)=B$, then
\beq\label{DtN:trans}
\Lambda_{\Omega,B}\left[\gamma\right]=\Lambda_{\Omega,B}\left[F_{*}\gamma\right],
\eeq
where $F_*\gamma$ is the pushforward conductivity:
\begin{equation}\label{eq:cloak_mat}
\left(F_{*}\gamma\right)(y)=\frac{DF(x)\,\gamma(x)\, (DF(x))^{T}}{\text{det}\left(DF(x)\right)},\quad x=F^{-1}(y)
\end{equation}
and $DF$ is the Jacobian matrix of $F$, i.e., $DF_{ij}=\frac{\partial y_{i}}{\partial x_{j}}$ for $i,j=1,\dots,d$.
We use this property to construct anisotropic cloaking structure.

\subsection{Cloaking design with anisotropic materials via pushforward of GPT-vanishing structures}\label{section:pushforward}

Let ${\sigma}$ be a radially stratified, piecewise-constant conductivity in $\RR^d$, $d=2,3$, defined by finitely many layers of isotropic materials with a perfectly insulating core. For radii $1=r_{L+1}<r_L<\cdots<r_1=2$ ($L$ denotes the number of coating layers), we set
\begin{equation}\label{eq:coating_sigma:reference}
\sigma=\sigma_{0}\,\chi_{\mathbb{R}^{d}\setminus\overline{B_2}}+\sum_{j=1}^{L}\sigma_{j}\chi_{A_{j}}\quad \mbox{with }\sigma_0=1,\quad \sigma_{L+1}=0,
\end{equation}
where
$A_j=\{x\in\RR^d: r_{j+1}<|x|<r_j\}$ for $j=1,\dots,L$ are the coating layers and $A_{L+1}=\{x\in\RR^d: |x|<r_{L+1}\}$ is the core. We assume that the parameters $\sigma_j$ for $1\le j \le L$ are chosen such that ${\sigma}$ is a GPT-vanishing structure of order $N$.  

Now, we consider the boundary value problem in $B_1$ with a perfectly insulating inner core of radius $\rho$: 
\beq\label{cond:smallcore:insul}
\begin{cases}
\ds\nabla\cdot(\sigma\circ\Psi_{1/\rho})\nabla u=0 & \text{in }B_{1}\setminus\overline{B}_{\rho},\\
\ds\frac{\partial u}{\partial n}=0 & \text{on }\p B_\rho,\\
\ds u=f\in H^{\frac{1}{2}}\left(\p B_1\right) & \text{on }\p B_1,
\end{cases}
\eeq
where $\sigma$ is a GPT-vanishing structure with a perfectly insulating core given in \eqnref{eq:coating_sigma:reference} with 
$r_1=2,\ r_{L_1}=1,\ \sigma_0=1$, and $\sigma_{L+1}=0$ and, thus,
$$\sigma\circ\Psi_{1/\rho}
=\begin{cases}
\ds 1 &\quad\mbox{in }B_{1}\setminus B_{2\rho},\\
\ds\mbox{coating layers}&\quad \mbox{in }B_{2\rho}\setminus\overline{B_\rho},\\
\ds 0 &\quad\mbox{in }B_\rho.
\end{cases}
$$

We apply a transformation $y=F(x)$ on $B_1$ (mapping virtual to physical coordinates). This transformation performs a core blow-up, expanding the sub-region $B_\rho$ to $B_{1/2}$. In particular, we set $F:\overline{B_1}\to\overline{B_1}$ to be a radial map defined by 
\begin{align} \label{eq:diffeo_radial}
F(x)=g\left(\left|x\right|\right)\frac{x}{\left|x\right|} \quad \mbox{for }x\neq 0;\quad F(0)=0,
\end{align}
where
\begin{equation} \label{eq:diffeo_radial_g}
g\left(t\right)
=\begin{cases}
\ds\frac{t}{2\rho}\quad & \text{if }t\le\rho,\\[2mm]
\ds\frac{1}{2}\left(\frac{t}{\rho}\right)^{\alpha(\rho)} \quad& \text{if }\rho\le t\le\frac{3}{4},\\[2mm]
\ds t \quad & \text{if }\frac{3}{4}\le t\le1
\end{cases}
\end{equation}
with $\alpha(\rho)$ chosen so that $g^-(\frac{3}{4})=\frac{1}{2}\left(\frac{\frac{3}{4}}{\rho}\right)^{\alpha(\rho)} =\frac{3}{4}$, that is,
\beq\label{def:alpha}
\alpha(\rho)=\frac{\ln\frac{3}{2}}{\ln\frac{3}{4}-\ln\rho}>0.
\eeq
Assuming that $0<\rho<\frac{1}{2}$, we have $0<\alpha(\rho)<1$.

Note that $F:\overline{B_{1}}\rightarrow \overline{B_{1}}$ is a radial, piecewise diffeomorphism such that $F\left(B_{\rho}\right)= B_{\frac{1}{2}}$ and $F(x)=x$ on $\p B_1$.
As in \eqnref{DtN:trans}, one can show that
\[
\Lambda_{B_{1},B_{\rho}}\left[\sigma\circ \Psi_{1/\rho}\right]=\Lambda_{B_{1},B_{\frac{1}{2}}}\left[F_{*}(\sigma\circ\Psi_{1/\rho})\right].
\]
Invoking (\ref{eq:cloaking_gpt_fin}) yields the following theorem for $d=2,3$ (the two-dimensional case was established in \cite{Ammari:2013:ENC_1,Heumann:2014:AEA}.)

\begin{thm} \label{thm:cloaking_gpt}
Let $\sigma\circ\Psi_{1/\rho}$ be the conductivity profile in $\RR^d$ ($d=2,3$) given as in \eqnref{eq:coating_sigma:reference} and \eqnref{cond:smallcore:insul}, yielding a GPT-vanishing structure of order $N$ with a perfectly insulating core $B_\rho$ and coating layers in $B_{2\rho}\setminus\overline{B_\rho}$. Let $F$ be given by \eqnref{eq:diffeo_radial}.

Then there exists a constant $C$ independent of $\rho$ such that 
\[
\left\Vert \Lambda_{B_{1},B_{{1}/{2}}}\left[F_{*}(\sigma\circ\Psi_{1/\rho})\right]-\Lambda_{B_{1}}\left[1\right]\right\Vert _{\mathcal{L}\left(H^{{1}/{2}}\left(\partial B_{1}\right),\,H^{-{1}/{2}}\left(\partial B_{1}\right)\right)}\le C\rho^{2N+d}.
\]
\end{thm}

We compute $F_{*}(\sigma\circ\Psi_{1/\rho})$ by applying the pushforward conductivity formula in \eqnref{eq:cloak_mat}. First observe that, for $x\neq 0$,
\begin{equation}\notag%\label{eq:DF}
\begin{aligned} 
DF(x)
&= g'(|x|)\,\frac{x\otimes x}{|x|^2}
  + \frac{g(|x|)}{|x|}\left(I_d-\frac{x\otimes x}{|x|^2}\right),\\
  \text{det}\left(DF(x)\right)&=g'(|x|)\left(\frac{g(|x|)}{|x|}\right)^{d-1},
\end{aligned}
\end{equation}
where $I_d$ is the $d\times d$ identity matrix and $x\otimes x=x x^T$, with $x$ viewed as a column vector.

Set
$$ y=F(x),\ r=|x|,\ s=|y|,\ \hat{y}=\frac{y}{|y|}.$$
Since $F$ is a radial map, we have 
$\frac{x\otimes x}{|x|^2}=\hat{y}\otimes\hat{y}$. 
For the isotropic and piecewise constant conductivity $\sigma\circ\Psi_{1/\rho}$, the pushforward $F_* (\sigma\circ\Psi_{1/\rho})$ on $B_{1}\setminus B_{\frac{1}{2}}$ can be written explicitly as follows:
\begin{align}\notag
&\left(F_* (\sigma\circ\Psi_{1/\rho})\right)(y)\\\notag
=&\, (\sigma\circ\Psi_{1/\rho})(x)\bigg(g'(|x|)\Big(\frac{g(|x|)}{|x|}\Big)^{d-1}\bigg)^{-1}
\bigg(  \big(g'(|x|)\big)^2\,\frac{x\otimes x}{|x|^2}
  + \Big(\frac{g(|x|)}{|x|}\Big)^2\Big(I_d-\frac{x\otimes x}{|x|^2}\Big) \bigg)\bigg\lvert_{x=F^{-1}(y)}\\ \notag%\label{eqn:Fstar1}
=&\, \sigma\left(\rho^{-1}|x|\right)\left[\Big(\frac{|x|}{g(|x|)}\Big)^{d-2}
  \bigg( \frac{|x|g'(|x|)}{g(|x|)}\,\hat{y}\otimes\hat{y}
  + \frac{g(|x|)}{|x|g'(|x|)}\left(I_d-\hat{y}\otimes\hat{y}\right) \bigg)\right]\bigg\lvert_{x=F^{-1}(y)}.
  \end{align}
This can be simplified as
\begin{align}
\label{eq:cloak}
\left(F_* (\sigma\circ\Psi_{1/\rho})\right)(y)
&= \sigma\left(\rho^{-1}g^{-1}(|y|)\right) \Big[\lambda_1(y)\,\hat{y}\otimes\hat{y}+\lambda_2(y) \left(I_{d}-\hat{y}\otimes\hat{y}\right)\Big].
\end{align}

For notational convenience, we regard $\lambda_1$ and $\lambda_2$ as radial functions and write $\lambda_1(y)=\lambda_1(|y|)=\lambda_1(s)$ and $\lambda_2(y)=\lambda_2(|y|)=\lambda_2(s)$ for $s=|y|\in[0,1]$. Since $g'(|x|)>0$, both $\lambda_1$ and $\lambda_2$ are positive.

Note that $\lambda_1 $ is the eigenvalue of $F_*1$ corresponding to the radial direction $\hat{y}$ (with multiplicity $1$), and $\lambda_2 $ is the eigenvalue of $F_*1$ corresponding to the tangential directions to the unit sphere $\mathbb{S}^{d-1}$ at $\hat{y}$ (with multiplicity $d-1$). 

We further set 
\begin{align}
\label{def:sigma_12}
\sigma_j^*(y)=\sigma^*_{j}(s):=\sigma\left(\rho^{-1}|x|\right) \lambda_j(s),\ j=1,2.
\end{align}
and simplify the forward conductivity as
\begin{align}\notag
\left(F_* (\sigma\circ\Psi_{1/\rho})\right)(y)
&=\sigma^*_1 \,\left(\hat{y}\otimes\hat{y}\right) + \sigma^*_2\left(I_{d}-\hat{y}\otimes\hat{y}\right).
\end{align}

We can express the eigenvalues $\lambda_1,\lambda_2$ in terms of the inverse function of $g$ given in \eqnref{eq:diffeo_radial_g} as follows. 
Let $s\mapsto g^{-1}(s)$ be the inverse function of $g$, that is,
\begin{equation} \label{eq:g_inv}
g^{-1}\left(s\right)
=\begin{cases}
\ds2s\rho\quad & \text{if }s\le\frac{1}{2},\\[2mm]
\ds\rho\, (2s)^{\frac{1}{\alpha(\rho)}} \quad& \text{if }\frac{1}{2}\le s\le\frac{3}{4},\\[2mm]
\ds s \quad & \text{if }\frac{3}{4}\le s\le1
\end{cases}
\end{equation}
with $\alpha(\rho)$ defined in \eqnref{def:alpha}. Since $s=g(|x|),\ |x|=g^{-1}(s)$, and $(g^{-1})'(s)=\frac{1}{g'(g^{-1}(s)}=\frac{1}{g'(|x|)}$, we have
\begin{align}\label{def:lambda12}
\lambda_1(s)=\left(\frac{g^{-1}\left(s\right)}{s}\right)^{d-2} \frac{1}{\lambda(s)} 
\quad\mbox{and}\quad
\lambda_2(s)=\left(\frac{g^{-1}\left(s\right)}{s}\right)^{d-2}\lambda(s)
\end{align}
with
\begin{align}\label{eq:lambda}
\ \lambda(s):=s\,\frac{\left(g^{-1}\right)^{\prime}\left(s\right)}{g^{-1}\left(s\right)}.
\end{align}
For $g$ given by \eqnref{eq:diffeo_radial_g} and its corresponding inverse in \eqnref{eq:g_inv}, we have
\begin{align}\label{eq:est_lambda}
\lambda(s)
=\begin{cases}
\ds 1,\quad &s<\frac{1}{2},\\
\ds \frac{1}{\alpha(\rho)}, \quad &\frac{1}{2}<s<\frac{3}{4},\\
\ds 1,\quad &\frac{3}{4}<s\le1.
\end{cases}
\end{align}

We define the maximal anisotropy as 
\beq\label{def:chi_max}
\chi_{\text{max}}:=\underset{s\in\left[\frac{1}{2},1\right]}{\operatorname{ess\ sup}}
\left\{
\frac{\lambda_{1}(s)}{\lambda_{2}(s)},\,
\frac{\lambda_{2}(s)}{\lambda_{1}(s)}\right\}.
\eeq
In addition, we define the minimal and maximal directional eigenvalues by
\begin{equation*}\lambda_{\text{min}}:=\underset{s\in\left[\frac{1}{2},1\right]}{\operatorname{ess \ inf}}\left\{ \lambda_{\text{1}}(s),\lambda_{\text{2}}(s)\right\} \quad\text{and}\quad\lambda_{\text{max}}:=\underset{s\in\left[\frac{1}{2},1\right]}{\operatorname{ess\ sup}}\left\{ \lambda_{\text{1}}(s),\lambda_{\text{2}}(s)\right\} .\end{equation*}

It can be verified that, for $F$ given by \eqnref{eq:diffeo_radial} and \eqnref{eq:diffeo_radial_g},
\begin{equation}
\chi_{\text{max}} =\underset{s\in\left[\frac{1}{2},1\right]}{\text{ess sup }} \lambda^2=O \left(\lvert \ln \rho \rvert^2\right),\quad \lambda_{\text{min}}= O\left(\frac{\rho^{d-2}}{\lvert\ln\rho\rvert}\right),\quad \lambda_{\text{max}}=O\left(\lvert \ln\rho \rvert \right).
\end{equation}

\begin{remark}
As shown in \cite{Griesmaier:2014:EAC}, the map $F$ defined by \eqnref{eq:diffeo_radial} and \eqnref{eq:diffeo_radial_g} is the optimal transformation for minimizing $\chi_{\text{max}}$ among all radial transformations $H_{\rho}:\overline{B_{1}}	\rightarrow\overline{B_{1}}$ of the form 
\begin{align*}
H_\rho(x)=\begin{cases}
\ds\frac{x}{2\rho},\quad & x\in B_{\rho},\\
\ds\psi_{\rho}\left(\left|x\right|\right)\frac{x}{\left|x\right|}, \quad & x\in {B_{\frac{3}{4}}}\setminus B_{\rho},\\
\ds x,\quad & x\in\overline{B_{1}}\setminus B_{\frac{3}{4}},
\end{cases}
\end{align*}
where $\psi_{\rho}:\left[\rho,\frac{3}{4}\right]\rightarrow\left[\tfrac{1}{2},\frac{3}{4}\right]$ is a $C^{1}$ diffeomorphism satisfying $$\psi_{\rho}\left(\rho\right)=\tfrac{1}{2},\ \psi_\rho\left(\tfrac{3}{4}\right)=\tfrac{3}{4},\ \mbox{and }\psi_\rho^{\prime}\left(t\right)>0\ \mbox{for all }t\in\left(\rho,\tfrac{3}{4}\right).$$
\end{remark}

\smallskip

In the virtual domain $B_{1}\setminus B_{\rho}$, prior to the pushforward by $F$, the conductivity $\sigma\circ\Psi_{1/\rho}$ varies within the annular region $\left\{ \rho<\left|x\right|<2\rho\right\}$, forming concentric piecewise-constant layers. Consequently, the transformed conductivity $F_{*}(\sigma\circ\Psi_{1/\rho})$ is piecewise smooth and anisotropic, exhibiting jump discontinuities across the interfaces of the multilayered coating. In addition, a discontinuity occurs along $\left|x\right|=\frac{3}{4}$, where $\left(g^{-1}\right)^{\prime}$ is not continuous. 

In \cite{Capdeboscq:2025:CCC}, the authors present a homogenization-based method for designing laminate structures composed of isotropic and homogeneous materials that approximate the anisotropic cloak. The key idea is to construct a sequence of laminated configurations whose effective behavior, in the sense of $H$-convergence, converges to the desired anisotropic conductivity matrix. In Section~\ref{subsec:homogenization}, we refine the approach of \cite{Capdeboscq:2025:CCC} by incorporating GPT-vanishing structures in order to enhance the cloaking effect.

\subsection{Cloaking design with isotropic materials via homogenization}\label{subsec:homogenization}

Let $\sigma\circ \Psi_{1/\rho}$ be as in Theorem~\ref{thm:cloaking_gpt}, and let $\alpha(\rho),\lambda_1,\lambda_2$, and $\sigma^*_1$, $\sigma^*_2$ be defined by \eqnref{def:alpha}, \eqnref{def:lambda12} and \eqnref{def:sigma_12}, respectively.

For $\varepsilon >0$, let $N_{\varepsilon}$ be the smallest positive integer such that 
$$\bigcup_{k=0}^{N_\varepsilon-1}I_k =\left[\tfrac{1}{2},1\right],\quad
I_k:=\left(\tfrac{1}{2}+\varepsilon [ k,\, k+1)\right)\cap\left[\tfrac{1}{2},1\right].$$ 
For each $k$, denote by $s_{k}$ and $s_{k+1}$ the left and right endpoints of $I_{k}$, respectively.

We define the function $A_\varepsilon:\RR^d\to\RR$ by
\begin{equation}\label{eq:Aeps_def}
A_{\varepsilon}\left(y\right)
:=\sum_{k=0}^{N_{\varepsilon}-1}a_{k}\left(\tfrac{s}{\varepsilon}\right)\chi_{I_{k}}(s),\quad s=|y|,
\end{equation}
where, for each $k=0,\ldots,N_{\varepsilon}-1$, the function $a_{k}$ is $1$-periodic and defined by
\[
a_{k}\left(s\right):=\begin{cases}
\ds\alpha, \quad& \ds s\in \left[0,\,l_0\right),\\
\ds 1,\quad &  \ds s\in l_0+[0,\,l_1),\\
\ds\gamma, \quad & \ds s\in l_0+l_1+[0,\,1-l_0-l_1)
\end{cases}
\]
with  material parameters $0<\alpha<\gamma<\infty$.
For each radial position $s_k$, we choose $l_0$ and $l_1$, depending on $s_k$,  $\alpha$, $\gamma$, $\sigma^*_{1}$ and $ \sigma^*_{2}$ so that

\begin{equation}
\begin{cases}
\ds \alpha\, l_0+l_1+\gamma\left(1-l_0-l_1\right)=\sigma^*_{2}(s_k),\\[2mm]
\ds \frac{1}{\alpha}\, l_0+l_1+\frac{1}{\gamma}\left(1-l_0-l_1\right)=\frac{1}{\sigma^*_{1}(s_k)}.
\end{cases}\label{eq:system-parameters}
\end{equation}
For $I_k$ contained in $(\frac{3}{4}, 1]$, we take $l_1=1$ so that $l_0=1-l_0-l_1=0$. In other words, laminates with conductivities $\alpha$ and $\gamma$ appears up to $r=\frac{3}{4}$.

Recall that $\sigma_1^*(s)=\sigma(x) \lambda_1(s)=\sigma\left(\rho^{-1}g^{-1}(s)\right) \lambda_1(s)$ , where $\lambda_1(s)$ is independent of $\sigma$. 
We distinguish two cases:
\begin{itemize}
    \item Case 1. $\sigma_1^*(s)\le 1$ for all $s\in [\frac{1}{2},1]$.
    \item Case 2. $\sigma_1^*(s)> 1$ for some $s\in [\frac{1}{2},1]$.
\end{itemize}
Assuming that $\rho$ is sufficiently small, it follows from \eqnref{def:lambda12} and \eqnref{eq:lambda} that $\lambda_1\le 1$. Case 1 is the generic case that holds if $\rho$ is sufficiently small. We analyze the two cases separately below.

\smallskip

{\bf Case 1 ($\sigma_1^*(s)\le 1$ for all $s\in [\frac{1}{2},1]$).}
In this case, we further assume that $\alpha$ and $\gamma$ are chosen so that
\beq\label{eq:parameter_conditions1}
\max_{s\in[\frac{1}{2},1]} \frac{1-\sigma^*_2}{\frac{1}{\sigma^*_1}-1}<\alpha<\frac{1}{{\max}_{s\in[\frac{1}{2},1]} \,\frac{1}{\sigma^*_{1}(s)}}
\eeq
and
\beq\label{eq:parameter_conditions12}
\gamma>\underset{s\in[\frac{1}{2},1]}{\max}\left(\tfrac{\sigma^*_{2}(s)-\alpha}{1-\frac{1}{\sigma^*_{1}(s)}\alpha}\right).
\eeq
Such a choice of $\alpha$ is possible when $\rho$ is sufficiently small, and, having fixed $\alpha$, we then choose $\gamma$ accordingly so that it satisfies \eqnref{eq:parameter_conditions12}. 

Under the conditions \eqnref{eq:parameter_conditions1} and \eqnref{eq:parameter_conditions12}, the following inequalities hold in particular at each $s_k$ appearing in \eqnref{eq:system-parameters}:
\begin{equation}\label{cond:alpha;beta;good}
\frac{1-\sigma^*_2(s_k)}{\frac{1}{\sigma^*_1(s_k)}-1}<\alpha<\frac{1}{\sigma^*_{1}(s_k)}
\quad\mbox{ and }
\gamma>\tfrac{\sigma^*_{2}(s_k)-\alpha}{1-\frac{1}{\sigma^*_{1}(s_k)}\alpha}.
\end{equation}
Consequently, the system \eqnref{eq:system-parameters} admits solutions $l_0$ and $l_1$ satisfying
\begin{equation}\label{eq:proportions_3mat}
\begin{cases}
l_0=\tfrac{\alpha\left(\sigma^*_{2}+\frac{1}{\sigma^*_{1}}\gamma-\gamma-1\right)}{\left(1-\alpha\right)\left(\gamma-\alpha\right)},\\[3mm]
l_1=\tfrac{\sigma^*_{2}+\frac{1}{\sigma^*_{1}}\alpha\gamma-\alpha-\gamma}{\left(\alpha-1\right)\left(\gamma-1\right)},\\[3mm]
1-l_0-l_1=\tfrac{\gamma\left(\sigma^*_{2}+\frac{1}{\sigma^*_{1}}\alpha-\alpha-1\right)}{\left(\gamma-\alpha\right)\left(\gamma-1\right)},\\[3mm]
0\le l_0,\,l_1,\, \left(1-l_0-l_1\right)\le 1,
\end{cases}
\end{equation}
where, for brevity, $\sigma_1^*$ and $\sigma_2^*$ are understood to be evaluated at $s_k$. 

\smallskip

{\bf Case 2 ($\sigma_1^*(s)> 1$ for some $s\in [\frac{1}{2},1]$).} 
From \eqnref{def:lambda12} and \eqnref{eq:lambda}, it holds that $\lambda_2> \lambda_1$. For $s_k$ satisfying $\sigma_1^*(s_k)>1$, we have $\sigma_2^*(s_k)>1$ and 
\begin{equation*}
\frac{1-\sigma^*_2(s_k)}{\frac{1}{\sigma^*_1(s_k)}-1}>\frac{1-\sigma^*_1(s_k)}{\frac{1}{\sigma^*_1(s_k)}-1}=\frac{1}{\sigma^*_{1}(s_k)}.
\end{equation*}
In the case where $\sigma_1^*(s_k)>1$, one cannot choose $\alpha$ and $\gamma$ to satisfy \eqnref{cond:alpha;beta;good}. Instead, we use three isotropic materials with conductivity values $\alpha_k$, $1$, and $\gamma_k$, where $\alpha_k$ and $\gamma_k$ are chosen to satisfy
\begin{equation}\label{eq:parameter_conditions2}
0<\alpha_k<1 \quad\text{and} \quad  \frac{\sigma^*_{2}(s_k)-\alpha_k}{1-\frac{1}{\sigma^*_{1}(s_k)}\alpha_k} < \gamma_k < \frac{1-\sigma^*_2(s_k)}{\frac{1}{\sigma^*_1(s_k)}-1},
\end{equation}
if $ \tfrac{\sigma^*_{2}(s_k)-\alpha_k}{1-\frac{1}{\sigma^*_{1}(s_k)}\alpha_k} <\tfrac{1-\sigma^*_2(s_k)}{\frac{1}{\sigma^*_1(s_k)}-1}$. 
The proportions $l_0$, $l_1$, and $(1-l_0-l_1)$ of the conductivities $\alpha_k$, $1$, and $\gamma_k$, respectively, are then computed from  \eqnref{eq:proportions_3mat} with $\alpha=\alpha_k$ and $\gamma=\gamma_k$. The condition \eqnref{eq:parameter_conditions2} ensures that these proportions lie in $(0,1)$. 

We can relax the condition on $\alpha$ in \eqnref{eq:parameter_conditions1} so that it applies to both {\bf Case 1} and {\bf Case 2}. 
Note that \eqnref{eq:parameter_conditions1} implies $\alpha <1$, so $\alpha$ represents a low conducting isotropic material compared to the background conductivity $1$. Moreover, in \eqnref{eq:parameter_conditions2} we require $\alpha_k<1$. We may therefore choose a low conductivity $\alpha$ as $\alpha_k$  satisfying (where this $\alpha$ is independent of $k$)

\beq\label{eq:alpha_cond}
	\max_{s\in[\frac{1}{2},1],~ \sigma_1^*(s)<1} \frac{1-\sigma^*_2}{\frac{1}{\sigma^*_1}-1}
	<\alpha<
\min_{s\in[\frac{1}{2},1],~ \sigma_1^*(s)<1}	\sigma_1^*(s),
\eeq
where the maxima and minima are taken over those $s$ for which $\sigma_1^*(s)<1$. 
This choice of $\alpha$ is valid in both {\bf Case 1} and {\bf Case 2}.

Now we turn to the high conducting constituent isotropic material $\gamma$. Take $\alpha$ satisfying \eqnref{eq:alpha_cond}. We choose $\gamma$ independent of $k$ so that it satisfies either \eqnref{eq:parameter_conditions12} or \eqnref{eq:parameter_conditions2}, whenever such a choice is possible; this situation is illustrated in Figure~\ref{fig:L=4} in Section \ref{sec:Illustrative-examples}. 

However, in {\bf Case 2}, it may happen that no such $\gamma$ exists. 

Note that
$$b_1(s):=\frac{\sigma^*_{2}(s)-\alpha}{1-\frac{1}{\sigma^*_{1}(s)}\alpha} <\frac{1-\sigma^*_2(s)}{\frac{1}{\sigma^*_1(s)}-1}:=b_2(s)\qquad\mbox{whenever }\sigma_1^*(s)>1.$$

 Suppose $s_k$ and $s_{k^\prime}$ are two radial values such that  $\sigma_1^*(s_k)>1$ and $\sigma_1^*(s_{k^\prime})>1$. This means that on $(s_k,s_{k+1})$ and $(s_{k^\prime},s_{k^\prime+1})$, we require $b_1(s_k)<\gamma_k<b_2(s_k)$ and $b_1(s_{k^\prime})<\gamma_{k^\prime}<b_2(s_{k^\prime})$ respectively, by condition \eqnref{eq:parameter_conditions2}. If $b_1(s_k)>b_2(s_{k^\prime})$ or $b_2(s_k)<b_1(s_{k^\prime})$, then the construction requires two highly conducting materials, say, $\gamma^{(1)}$ and $\gamma^{(2)}$. This situation is illustrated in Figure \ref{fig:L=6}, where two high conducting materials are required.

Since $a_k(\frac{s}{\varepsilon})$ is $\varepsilon$-periodic, for each $s_k$ we have
\beq\label{property:a_k}
\begin{aligned}
\int_{s_k}^{s_k+\varepsilon} a_k\left(\tfrac{s}{\varepsilon}\right)\,ds
&=\varepsilon\left(\alpha\, l_0+l_1+\gamma\left(1-l_0-l_1\right)\right)
=\varepsilon\,\sigma^*_2(s_k),\\
\int_{s_k}^{s_k+\varepsilon}\left(a_k\left(\tfrac{s}{\varepsilon}\right)\right)^{-1}\,ds
&=\varepsilon\left(\frac{1}{\alpha}\, l_0+l_1+\frac{1}{\gamma}\left(1-l_0-l_1\right)\right)
=\varepsilon\,\frac{1}{\sigma^*_1(s_k)}.
\end{aligned}
\eeq

We now quantify the weak convergence of the coefficients of the heterogeneous medium defined by \eqnref{eq:Aeps_def}. For this purpose, we introduce the following notation.
\begin{notation}\label{notation:1}
We define
\beq\label{def:kappa}
\kappa:=\max\bigg\{ \underset{|x|\in[\frac{1}{2},1]}{\max}\sigma(x),\ \Big( \underset{|x|\in[\frac{1}{2},1]}{\min}\sigma(x)\Big)^{-1}\bigg\}
\eeq
and
\begin{equation}\label{notation_ab}
a_{\rho,\varepsilon}= 1+\varepsilon \left|\ln\rho\right|,
\quad
b_{\rho,\varepsilon}=1+\varepsilon \,\frac{\left|\ln\rho\right|}{\rho}.
\end{equation}
\end{notation}

\begin{lemma} \label{lem:weak_conv} 
Fix $\rho>0$ sufficiently small, and let $0<\varepsilon\ll \rho$.  

Let  $\sigma\circ\Psi_{1/\rho}$ be as in Theorem~\ref{thm:cloaking_gpt}, $\sigma_{1}^{*}$, $\sigma_{2}^{*}$ defined by \eqnref{def:sigma_12} with $\lambda_{1},\lambda_{2}$ defined by  \eqnref{def:lambda12}. Set $\kappa, a_{\rho,\varepsilon}, b_{\rho,\varepsilon}$ as in Notation~\ref{notation:1}. 
Finally, let $A_\varepsilon$ be given by \eqnref{eq:Aeps_def}, defined on the domain $\Omega:=B_1\setminus\overline{B}_{\frac{1}{2}}$.
There exists a constant $C$ such that, for any $f\in W^{1,\infty}(\Omega)$, 
\begin{align} 
\left|\int_{\Omega}\big(A_{\varepsilon}-\sigma^*_2\big)(y)f(y)\, dy\right|
& \leq
C\kappa\, \varepsilon
\cdot
\begin{cases}
\ds\left|\ln \rho \right| \left\Vert \partial_{r}f\right\Vert_{L^{\infty}(\Omega)}+\left|\ln \rho\right|\left\Vert f\right\Vert_{L^{\infty}(\Omega)}, & d=2,\\[2mm]
\ds a_{\rho,\varepsilon}\left\Vert \partial_{r}f\right\Vert_{L^{\infty}(\Omega)}+\left|\ln\rho\right|\left\Vert f\right\Vert_{L^{\infty}(\Omega)}, & d=3
\end{cases}\label{eq:weak_lim_1}
\end{align}
and
\begin{align}
\left|\int_{\Omega}\left(\frac{1}{A_{\varepsilon}}-\frac{1}{\sigma^*_1}\right)(y)f(y)\,dy\right| &\leq
C\kappa\, \varepsilon\cdot
\begin{cases}
\ds\left|\ln \rho\right|\left\Vert \partial_{r}f\right\Vert_{L^{\infty}(\Omega)}+\left|\ln \rho\right|\left\Vert f\right\Vert_{L^{\infty}(\Omega)},
 & d=2,\\[2mm]
\ds b_{\rho,\varepsilon}\left\Vert \partial_{r}f\right\Vert_{L^{\infty}(\Omega)}+\frac{\left|\ln\rho\right|}{\rho}\left\Vert f\right\Vert_{L^{\infty}(\Omega)}, & d=3.
\end{cases}\label{eq:weak_lim_2}
\end{align}
\end{lemma}

\begin{proof}
Set $x=F^{-1}(y)$ as before. For notational convenience, we write $A_\varepsilon(s)=A_\varepsilon(y)$ and $\sigma(t)=\sigma(x)$ with $s=|y|$,  $t=|x|$. This is unambiguous since both functions are radial. Recall that the coating layer radii are ordered as $1<r_L<\cdots<r_1=2$.

We first consider the case $d=2$.   In this setting, $\lambda_2\left(t\right)=\frac{1}{\lambda_1}(t)=\lambda\left(t\right)$, as given in \eqnref{eq:lambda}. Note that $A_\varepsilon(t)$ is a real positive valued. 
From \eqnref{eq:Aeps_def} and \eqnref{property:a_k}, we have
\beq
\begin{aligned}
&\int_{\frac{1}{2}}^{1}A_{\varepsilon}\left(s\right)\, dt-\int_{\frac{1}{2}}^{1}\sigma^*_2\left(s\right)\, ds\\ \notag
=&\sum_{k=0}^{N_{\varepsilon}-2}\int_{I_{k}}\Big[a_k\left(\tfrac{s}{\varepsilon}\right)-\sigma^*_2(s)\Big]\, ds
 +\int_{I_{N_\varepsilon-1}}\Big[a_k\left(\tfrac{s}{\varepsilon}\right)-\sigma^*_2(s)\Big]\, ds\\
=&\sum_{k=0}^{N_{\varepsilon}-2}\int_{I_{k}}\Big[\sigma^*_2(s_k)-\sigma^*_2(s)\Big]\, ds
 +\int_{I_{N_\varepsilon-1}}\Big[a_k\left(\tfrac{t}{\varepsilon}\right)-\sigma^*_2(s)\Big]\, ds.
\end{aligned}\label{eq:Aeps-ave}
\eeq

The function $s\mapsto \sigma^*_2(s)$ is piecewise constant since both $\sigma$ and $\lambda$ are piecewise constant. In particular, $s\mapsto\sigma^*_2(s)$ has jump discontinuities at 
$${\widetilde{r}}_j:=g(r_j),$$ 
which corresponds to the radial locations of the transformed coating interfaces, and at $s=\frac{3}{4}$.
In the physical domain--that is, after applying the pushforward by $F$--the coating structure occupies the annulus $B_{g(2\rho)}\setminus B_{\frac{1}{2}}$. By \eqnref{eq:est_lambda}, it holds that for $s\in[\tfrac{1}{2},1]$,
 \begin{equation}\label{eq:lambda_sigma}
\sigma^*_2(s)
=\begin{cases}
  \ds   \tfrac{1}{\alpha(\rho)}\,\sigma_{j},&s\in \left(\widetilde{r}_{j+1},\widetilde{r}_{j}\right)\mbox{ for } 1\le j\le L,\\[3mm]
 \ds   \tfrac{1}{\alpha(\rho)},& s\in \left(g(2\rho),\,\frac{3}{4}\right),\\[3mm]
\ds  1,  &s\in\left(\frac{3}{4},\, 1\right).
\end{cases}
\end{equation}

If $\widetilde{r}_{j}\notin\text{int}\left(I_{k}\right)$ for every $j$, then $\sigma^*_2$ is constant on $I_{k}$. Hence,
\begin{equation*}
\int_{I_{k}}\Big[\sigma^*_2\left(s_{k}\right)\chi_{I_{k}}(s)-\sigma^*_2(s)\Big]\, ds=0.\label{eq:integral_cell}
\end{equation*}
On the other hand, if $\widetilde{r}_{j}\in\text{int}\left(I_{k}\right)$ for some $j$, then
\begin{equation*}
\int_{s_{k}}^{s_{k+1}}\Big|\sigma^*_2\left(s_{k}\right)\chi_{I_k}(s)-\sigma^*_2(s)\Big|\, ds
\le\varepsilon\max\{\sigma_j,\sigma_{j+1}\} \,  \frac{1}{\alpha(\rho)}\leq \varepsilon\max\sigma\max\lambda
\end{equation*}
and such cases can occur in at most $L$ microcells.  
This implies
\begin{align}
 \bigg\Vert \sum_{k=0}^{N_{\varepsilon}-2}\sigma^*_2\left(s_{k}\right)\chi_{I_{k}}-\sigma^*_2\bigg\Vert_{L^{1}\left(\Omega\right)} & \le C\varepsilon\max\sigma\max\lambda.\label{eq:int1_est}
\end{align}
Furthermore, by construction, on $B_1\setminus B_{\frac{3}{4}}\supset I_{N_\varepsilon-1}$, $a_k(y)=1=\sigma \lambda$. Thus, 
$$
\begin{aligned}
\int_{I_{N_\varepsilon-1}}\left|a_k\left(\tfrac{s}{\varepsilon}\right)\chi_{I_{N_\varepsilon-1}}(s)-\sigma^*_2(s)\right|\, ds
=0.
\end{aligned}
$$
Hence, using \eqnref{eq:est_lambda}, \eqnref{eq:Aeps-ave} and (\ref{eq:int1_est}), one has 
\begin{align}
\left\Vert A_{\varepsilon}\right\Vert_{L^{1}\left([\frac{1}{2},1]\right)} 
\leq C\max\sigma\max\lambda\leq C\max\sigma\left|\ln\rho\right|.\label{eq:L1_a_d=2}
\end{align}

We decompose the integral in \eqnref{eq:weak_lim_1} as
\begin{align}\label{S1plusS2}
\int_{\Omega}\left(A_{\varepsilon}-\sigma^*_2\right)(y)f(y)\, dy & =S_{1}+S_{2}+S_3
\end{align}
with

\begin{align*}
S_{1} &:= \sum_{k=0}^{N_{\varepsilon}-2}\int_{|y|\in I_{k}}
\Big(a_{k}\!\left(\tfrac{y}{\varepsilon}\right)-\sigma^*_2(s_{k})\Big)\,f(y)\,dy, \\
S_{2} &:= \sum_{k=0}^{N_{\varepsilon}-2}\int_{|y|\in I_{k}}
\Big(\sigma^*_2(s_{k})-\sigma^*_2(y)\Big)\,f(y)\,dy,
\end{align*}
and $S_3$ corresponding to the integral for $|x|\in I_{N_\varepsilon -1}$ so that
\begin{align*}
\left|S_3\right|=0.
\end{align*}
From the averaging relations in \eqnref{property:a_k}, we have
$
\int_{ I_{k}}\big(a_{k}\!\left(\tfrac{s}{\varepsilon}\right)-\sigma^*_2(s_{k})\big)\,ds = 0,
$
which yields 
\begin{align*}
\left|S_{1} \right| 
& =\left|\sum_{k=0}^{N_{\varepsilon}-2}\int_{|y|\in I_{k}}
\Big(a_{k}\left(\tfrac{y}{\varepsilon}\right)
-\sigma^*_2(s_{k})\Big)\Big(f(y)-f(s_k\hat{y})\Big)\, dy\right|\\
&\leq\sum_{k=0}^{N_{\varepsilon}-2}\int_{|y|\in I_{k}}
\Big(a_{k}\left(\tfrac{y}{\varepsilon}\right)
+\sigma^*_2(s_{k})\Big) \varepsilon \left\Vert \partial_{r}f\right\Vert_{L^{\infty}\left(\Omega\right)}\, dy\\
&\leq C\left(\left\Vert A_{\varepsilon}\right\Vert_{L^{1}\left(\frac{1}{2},1\right)}+\left\Vert \sum_{k=0}^{N_\varepsilon-2} 1_{I_k}\sigma_2^*\left(s_k\right)\right\Vert_{L^{1}\left(\frac{1}{2},1\right)}\right)\varepsilon\,\left\Vert \partial_{r}f\right\Vert_{L^{\infty}\left(\Omega\right)}\\[2mm]
&\leq C\max\sigma\cdot \varepsilon\,\left|\ln\rho\right|\left\Vert \partial_{r}f\right\Vert_{L^{\infty}\left(\Omega\right)}.
\end{align*}
Moreover, by \eqnref{eq:int1_est}, we simply bound
\begin{align*}
S_{2} & \le\bigg\Vert \sum_{k=0}^{N_{\varepsilon}-1}\sigma^*_2\left(s_{k}\right)\chi_{I_{k}}-\sigma^*_2\bigg\Vert_{L^{1}\left(\Omega\right)}\left\Vert f\right\Vert_{L^{\infty}\left(\Omega\right)}
\leq C\max\sigma\cdot\varepsilon\left|\ln \rho\right|\left\Vert f\right\Vert_{L^{\infty}\left(\Omega\right)}.
\end{align*}
Hence, we prove \eqnref{eq:weak_lim_1} for $d=2$. 

As an analogue of \eqnref{eq:L1_a_d=2}, we obtain
\begin{align*}
\bigg\Vert \frac{1}{A_{\varepsilon}}\bigg\Vert_{L^{1}\left([\frac{1}{2},1]\right)} & \le C\max\left(\frac{1}{\sigma}\right)\left|\ln\rho\right|.
\end{align*}
Indeed, if $\widetilde{r}_{j}\in\text{int}\left(I_{k}\right)$ for some $j$, which can occur for at most $L$ microcells, then \begin{equation}
\int_{s_{k}}^{s_{k+1}}\Big|\sigma^*_1\left(s_{k}\right)\chi_{I_{k}}(s)-\sigma^*_1(s)\Big|\,ds
\le\varepsilon\max\left\{\frac{1}{\sigma_{j}},\frac{1}{\sigma_{j+1}} \right\}\max \lambda\leq\varepsilon\max\frac{1}{\sigma}\,\frac{1}{\alpha(\rho)}.
\end{equation}
Therefore, by \eqnref{eq:est_lambda} and \eqnref{eq:lambda_sigma}, we have
\begin{equation}\label{L1_1/a_d=2}
\begin{aligned}
 \left\Vert \frac{1}{A_{\varepsilon}}\right\Vert_{L^{1}\left(\left[\frac{1}{2},1\right]\right)}	&\le\left\Vert \sigma^*_1\right\Vert_{L^{1}\left(\left[\frac{1}{2},1\right]\right)}+\bigg\Vert{\sum_{k=0}^{N_{\varepsilon}-1}\sigma^*_1\left(s_{k}\right)\chi_{I_{k}}}-\sigma^*_1\bigg\Vert_{L^{1}\left(\Omega\right)}\\
&\le\left\Vert \sigma^*_1\right\Vert _{L^{1}\left(\left[\frac{1}{2},1\right]\right)}+C\varepsilon\max\frac{1}{\sigma}\max\lambda\\
&\leq C\max\frac{1}{\sigma}\left|\ln\rho\right|.
\end{aligned}
\end{equation}
Following the arguments using $S_{1}$ and $S_{2}$ in \eqnref{S1plusS2}, we prove \eqnref{eq:weak_lim_2} for $d=2$.

Now, we prove the claim for for $d=3$. From \eqnref{def:lambda12} and \eqnref{eq:lambda},
\begin{equation*}
\lambda_1(s)=\frac{\left(g^{-1}\right)^{\prime}}{\lambda^2}(s),\quad
\lambda_2(s)  =\left(g^{-1}\right)^{\prime}(s)\quad\mbox{for }s\in \left[\frac{1}{2},1\right].
\end{equation*}

From \eqnref{eq:Aeps_def} and \eqnref{property:a_k}, we have
\beq\notag
\begin{aligned}
&\int_{\frac{1}{2}}^{1}A_{\varepsilon}(s)\, ds-\int_{\frac{1}{2}}^{1}\sigma^*_2(s)\, ds\\ \notag
=&\sum_{k=0}^{N_{\varepsilon}-2}\int_{I_{k}}\Big[\sigma^*_2\left(s_{k}\right)-\sigma^*_2(s)\Big]\, ds
 +\int_{I_{N_\varepsilon-1}}\Big[a_k\left(\tfrac{s}{\varepsilon}\right)-\sigma^*_2(s)\Big]\, ds\\
=&\sum_{k=0}^{N_{\varepsilon}-2}\int_{I_{k}}\left[\left(\sigma\left(g^{-1}\right)^{\prime}\right)\left(s_{k}\right)-\left(\sigma\left(g^{-1}\right)^{\prime}\right)(s)\right]\, ds
 +\int_{I_{N_\varepsilon-1}}\left[a_k\left(\tfrac{s}{\varepsilon}\right)-\left(\sigma\left(g^{-1}\right)^{\prime}\right)(s)\right]\, ds.
\end{aligned}
\eeq

For $g$ given by \eqnref{eq:diffeo_radial_g}, the mapping
$t\mapsto \left(g^{-1}\right)^{\prime}(t)$ is increasing on $[\frac{1}{2},\frac{3}{4}] $, and we have
\begin{equation}\label{lambda_est_3d}\begin{cases}
   \ds \underset{t\in[\frac{1}{2},1]}{\text{ess sup}}\left(g^{-1}\right)^{\prime}(t)
    = \left(g^{-1}\right)^{\prime}\left(\frac{3}{4}\right)&\le C\lvert \ln \rho \rvert,\\[2mm]
 \ds   \left\Vert \left(g^{-1}\right)^{\prime}(\cdot)\right\Vert _{L^{1}\left([\frac{1}{2},1]\right)} &\le C.
\end{cases} 
\end{equation}
Furthermore, the mapping $t\mapsto \left(g^{-1}\right)^{\prime}\left(t\right)$ is piecewise increasing for $t$ on the interval $[\frac{1}{2}, g(2\rho)]$, that corresponds the coating structure.

Fix $k$. If $\widetilde{r}_{j}:=g(r_j)\notin\text{int}\left(I_{k}\right)$ for every $j$, then $\sigma$ contantly takes the same value on $I_k$ and
\begin{equation}\label{int_same sigma}
\int_{I_k}\Big|\sigma^*_2\left(s_{k}\right)-\sigma^*_2(s)\Big| \,ds
\leq 
\varepsilon\,\sigma\big|_{I_{k}}\cdot
\Big[\left(g^{-1}\right)^{\prime}\left(s_{k+1}\right)-\left(g^{-1}\right)^{\prime}\left(s_k\right)\Big].  
\end{equation}
On the other hand, if $\widetilde{r}_{j}:=g(r_j)\in\text{int}\left(I_{k}\right)$ for some $j$, then we can split $I_{k}$ into two microcells, namely, $I_{k,1}:=[t_{k},\widetilde{r}_{j}]$ and $I_{k,2}:=[\widetilde{r}_{j},t_{k+1}]$ each being mapped to a constant value via $\sigma$. The number of such microcells is at most $2L$.

Now, fix $j$. Define $$\mathcal{I}_j=\bigcup J,\quad\mbox{by taking the union of all }J\in\big\{I_k,I_{k,1},I_{k,2}\,:\, \sigma\big|_J =\sigma_j\big\}.$$

Assuming the values $\sigma_j$ are distinct values, $\mathcal{I}_j$ forms an interval, say $\mathcal{I}_j=[a_j, b_j]\subset[\frac{1}{2},1]$. Then, by a telescoping argument,
\begin{align*}
\int_{\mathcal{I}_j}\Big|\sigma^*_2\left(s_{k}\right)-\sigma^*_2(s)\Big| \,ds
&\le \sum_{J\in \mathcal{I}_j}\int_J\Big|\sigma^*_2\left(s_{k}\right)-\sigma^*_2(s)\Big| \,ds \\
&\le 
\varepsilon\sigma_j\Big[\left(g^{-1}\right)^{\prime} \left(b_j\right)-\left(g^{-1}\right)^{\prime}\left( a_j\right) \Big]\\
&\le C \varepsilon \left|\ln \rho\right|,
\end{align*}
where the last inequality follows from \eqnref{lambda_est_3d}.

A similar argument applies to the region outside the coating structure $B_1\setminus B_{g(2\rho)}$, where $\sigma=1$ and $t\mapsto\left(g^{\prime}(t)\right)^{-1}$ is increasing.
By \eqnref{int_same sigma}, \eqnref{lambda_est_3d}, and a telescoping argument, one obtains
\begin{equation*}
\sum_{k\text{ s.t. }\sigma\lvert_{I_{k}}=1}\int_{I_{k}}\Big|\left(g^{-1}\right)^{\prime}\left(s_{k}\right)-\left(g^{-1}\right)^{\prime}(s)\Big|\,ds\le C\varepsilon\underset{t\in[\frac{1}{2},1]}\,{\text{ess sup}}\left(g^{\prime}(t)\right)^{-1}\le C\varepsilon\left|\ln\rho\right|.    
\end{equation*}
Therefore, with the $L^{1}$ estimate in \eqnref{lambda_est_3d}, 
\begin{equation}\label{eq:L1_a_d=3}
    \left\Vert A_{\varepsilon}\right\Vert _{L^{1}\left(\left[\frac{1}{2},1\right]\right)}\le C\max\sigma\left(1+\varepsilon\left|\ln\rho\right|\right).
\end{equation}

On the other hand, $t\mapsto\frac{\lambda^{2}}{\left(g^{-1}\right)^{\prime}}\left(t\right)$ is decreasing on $\left[\frac{1}{2},\frac{3}{4}\right]$ and 
\begin{equation}\label{est_siglam_d=3}\begin{cases}
\underset{\left[\frac{1}{2},1\right]}{\text{ess sup}}\frac{\lambda^{2}}{\left(g^{-1}\right)^{\prime}}(t)=\frac{\lambda^{2}}{\left(g^{-1}\right)^{\prime}}\left(\frac{1}{2}\right) & \le C\frac{\left|\ln\rho\right|}{\rho},\\
\left\Vert \frac{\lambda^{2}}{\left(g^{-1}\right)^{\prime}}\right\Vert _{L^{1}\left([\frac{1}{2},1]\right)} & \le C .
\end{cases}\end{equation}
Using similar arguments as above, we have
\begin{align*}
\left\Vert \frac{1}{A_{\varepsilon}}\right\Vert_{L^{1}\left(\frac{1}{2},1\right)} \le& \left\Vert \frac{1}{\sigma^*_{1}}\right\Vert _{L^{1}\left(\left[\frac{1}{2},1\right]\right)}+C\varepsilon\max\frac{1}{\sigma}\max\frac{1}{\lambda_{1}}
\le C\max\frac{1}{\sigma}\left(1+\varepsilon\frac{\left|\ln\rho\right|}{\rho}\right).
\end{align*}
Following the arguments using $S_{1}$ and $S_{2}$ in \eqnref{S1plusS2}, we prove \eqnref{eq:weak_lim_2} for $d=3$. 
\end{proof}

\section{Proof of the main result}\label{sec:main_result}
Throughout this section, given $\varphi\in H^{\frac{1}{2}}\left(r=1\right)$, let $u$ and $u_\varepsilon$ be the solutions to
\begin{equation}\label{eq:def_u}
\begin{cases}
\ds\nabla\cdot \left(F_{*}(\sigma\circ\Psi_{1/\rho})\, \nabla u\right)=0 \quad& \text{in }B_{1}\setminus{B}_{\frac{1}{2}},\\
\ds\frac{\partial u}{\partial\nu}=0 \quad& \text{on }|x|=\frac{1}{2},\\
\ds u=\varphi \quad& \text{on }|x|=1
\end{cases}
\end{equation}
and
\begin{equation}\label{eq:def_u_ep}
\begin{cases}
\ds\nabla\cdot\left(A_{\varepsilon}\nabla u_{\varepsilon}\right)=0 \quad& \text{in }B_{1}\setminus{B}_{\frac{1}{2}},\\
\ds\frac{\partial u_{\varepsilon}}{\partial\nu}=0 \quad& \text{on }|x|=\frac{1}{2},\\
\ds u_{\varepsilon}=\varphi \quad& \text{on }|x|=1,
\end{cases}
\end{equation}
respectively, where $F_{*}(\sigma\circ\Psi_{1/\rho})$ is defined in (\ref{eq:cloak}) and $A_{\varepsilon}$ in \eqnref{eq:Aeps_def}.

For $r=|x|>\frac{3}{4}$, it follows from \eqnref{eq:diffeo_radial_g} and \eqnref{eq:proportions_3mat} that $F(x)=x$, $\sigma^*_1(r)=\sigma^*_2(r)=1$, and hence $l_0(r)=0$ and $l_1(r)=1$.
Consequently, $$A_\varepsilon(y)=I_d \quad\mbox{throughout }B_1\setminus B_{1-\delta},$$ provided that $\delta$ is sufficiently small.

By the boundary conditions on $r=1$ and $r=\frac{1}{2}$, the radial symmetry of $F_* \sigma$ and $A_\varepsilon$, and the fact that $F_* \sigma=A_\varepsilon$ in  $B_{1}\setminus B_{1-\delta}$, we obtain 
\begin{align}\notag
  &\int_{B_{1}\setminus B_{\frac{1}{2}}} F_{*}(\sigma\circ\Psi_{1/\rho})\,\nabla u\cdot\nabla u\, dy-\int_{B_{1}\setminus B_{\frac{1}{2}}}A_{\varepsilon}\nabla u_{\varepsilon}\cdot\nabla u_{\varepsilon}\, dy\\ \notag
  =\,&\int_{B_{1}\setminus B_{\frac{1}{2}}}F_{*}(\sigma\circ\Psi_{1/\rho})\,\nabla u\cdot\nabla u_{\varepsilon}\, dy-\int_{B_{1}\setminus B_{\frac{1}{2}}}A_{\varepsilon}\nabla u_{\varepsilon}\cdot\nabla u\, dy\\ \label{energy:differ}
  =\,&\int_{\Omd}\left(F_{*}(\sigma\circ\Psi_{1/\rho})-A_{\varepsilon}\right)\nabla u_{\varepsilon}\cdot\nabla u\, dy
\end{align}
with
\[
\Omega_{\delta}:=B_{1-\delta}\setminus B_{\frac{1}{2}}.
\]
We also introduce the notation: for a function $f$ on $\Omega_\delta$, 
\[
\nabla_{\phi}f:=\left(I_{d}-\hat{x}\otimes\hat{x}\right)\nabla f,
\]
whenever the derivatives exist. 
Note that $\nabla_{\phi}f$ represents the component of the gradient vector orthogonal to the radial direction. 
We also use the notation  $\Delta_{\phi}f$ to denote the tangential Laplace--Beltrami operator which is $\partial_{\theta\theta}$ when $d=2$.
In the following derivations, we employ harmonic techniques to address boundary measurements by localizing our estimates to $\Omega_\delta$.

The coefficients $A_\varepsilon$ in the cloaking domain depends only on the radial variable. This radial symmetry ensures higher regularity of the solution, as established in earlier works (see, for instance, \cite{Chipot:1986:SLL,Li:2000:GES}). We now state the result in the following lemma, adapted to our setting with precise estimates, which corresponds to the one derived in \cite[Lemma 14]{Capdeboscq:2025:CCC}.
\begin{lemma}[\cite{Capdeboscq:2025:CCC}]\label{lem:rad-reg}
For any real functions $m_{1},m_{2}\in L^{\infty}\left((\tfrac{1}{2},1)\right)$, let $M\in L^{\infty}\left(B_{1},\mathbb{R}^{d\times d}\right)$
be an uniformly elliptic, positive definite matrix-valued function such that
$$M(x)=
\begin{cases}
\ds I_{d},\quad &x\in B_{1}\setminus B_{1-\delta},\\
\ds
m_{1}(|x|)\,\hat{x}\otimes\hat{x}+m_{2}(|x|)\bigl(I_{d}-\hat{x}\otimes\hat{x}\bigr),\quad &x\in B_{1}\setminus B_{\frac{1}{2}}\\
0&x\in B_{\frac{1}{2}}.
\end{cases}
$$
 For $\varphi\in H^{\frac{1}{2}}\left(\partial B_{1}\right)$, let $f\in H^{1}(B_{1}\setminus B_{\frac{1}{2}})$ be the weak solution of the boundary value problem 
\begin{equation}\label{eq:regul-eq}
\begin{cases}
\ds\nabla\cdot\left(M\nabla f\right)  =0\quad&\text{in }B_{1}\setminus B_{\frac{1}{2}},\\
\ds \frac{\partial f}{\partial n}  =0\quad&\text{on }\partial B_{\frac{1}{2}},\\
\ds f  =\varphi\quad &\text{on }\partial B_{1}.
\end{cases}
\end{equation}
Then we have 
$$f,\nabla_{\phi}f, \Delta_{\phi}f\in C(\overline{\Omd};\mathbb{R}^{d}),
\quad m_{1}\partial_{r}f\in W^{1,\infty}\left(\Omd\right), \quad m_{1}\partial_{r}\left(\nabla_{\phi}f\right)\in W^{1,\infty}(\Omd;\mathbb{R}^{d}).$$
Moreover, the following estimates hold with a constant $C$ independent of $\varphi$:

\beq\label{Lemma:harmonic:eeq}
\begin{aligned}
&\max_{\overline{\Omd}}\,\bigl(\left|f\right|+\left|\nabla_{\phi}f\right|\bigr)
\leq C\left\Vert \varphi\right\Vert_{H^{{1}/{2}}\left(\partial B_{1}\right)};\\
&\max_{\overline{\Omd}}\,\bigl(\left|m_{1}\partial_{r}f\right|+\left|m_{1}\partial_{r}(\nabla_{\phi}f)\right|\bigr)
\leq C\Bigl(1+\left\Vert m_{2}\right\Vert_{L^{1}(\frac{1}{2},1)}\Bigr)\left\Vert \varphi\right\Vert_{H^{{1}/{2}}\left(\partial B_{1}\right)};\\
&\bigl|\partial_{r}(m_{1}\partial_{r}f)(x)\bigr| 
+ \bigl|\partial_{r}(m_{1}\partial_{r}\nabla_{\phi}f)(x)\bigr|
\le C\,|m_{2}(x)|\,\|\varphi\|_{H^{1/2}(\partial B_{1})} 
\quad \text{for a.e. } x \in \Omega_\delta.
\end{aligned}
\eeq
\end{lemma}
Note that both $F_*\sigma$ and $A_\varepsilon$ satisfy the assumptions on $M$ in Lemma~\ref{lem:rad-reg}. This allows us to apply Lemma~\ref{lem:rad-reg} and obtain the required estimates for $u$ and $u_\varepsilon$ in the following proposition. More precisely, we establish an error estimate between the energies of the solutions corresponding to the anisotropic material and its piecewise isotropic laminate approximation:
\begin{prop}\label{prop:homres}
Let $\varphi\in H^{\frac{1}{2}}\left(r=1\right)$. Let $u$ and $u_\varepsilon$ be the solutions to \eqnref{eq:def_u} and \eqnref{eq:def_u_ep}, respectively. 

There exists a constant $C>0$, independent of $\rho$, such that for all sufficiently small $\varepsilon>0$, 
\begin{align*}
 & \bigg|\int_{B_{1}\setminus B_{\frac{1}{2}}}F_{*}(\sigma\circ\Psi_{1/\rho})\,\nabla u\cdot\nabla u\,dy-\int_{B_{1}\setminus B_{\frac{1}{2}}}A_{\varepsilon}\nabla u_{\varepsilon}\cdot\nabla u_{\varepsilon}\, dy\bigg|\\
 \leq  &
 \begin{cases}
\ds 
C\varepsilon\kappa^{2}\lvert \ln \rho \rvert^2\left(\frac{1}{\alpha}+\gamma_{\max}\right)\left\Vert \varphi\right\Vert ^{2}{}_{H^{\frac{1}{2}}\left(\partial B_{1}\right)}
\quad&\mbox{for }d=2,\\[2mm]
\ds C\varepsilon\kappa^{2}\left(\frac{1}{\alpha} a_{\rho,\varepsilon}^2+\gamma_{\max} a_{\rho,\varepsilon}b_{\rho,\varepsilon}+a_{\rho,\varepsilon}\kappa \frac{\lvert \ln \rho \rvert}{\rho }\right)\left\Vert \varphi\right\Vert ^{2}_{H^{\frac{1}{2}}\left(\partial B_{1}\right)}\quad &\mbox{for }d=3.
 \end{cases}
\end{align*}
Here $\alpha$ and $\gamma$ are as in \eqnref{eq:system-parameters} (see also \eqnref{eq:parameter_conditions1} and \eqnref{eq:parameter_conditions12}), and $\kappa,\, a_{\rho,\varepsilon},\,b_{\rho,\varepsilon}$ are given in \eqnref{def:kappa} and \eqnref{notation_ab}.
\end{prop}
\begin{proof}

By \eqnref{energy:differ}, we obtain
\begin{align}\label{int:Er:Et}
\int_{B_{1}\setminus B_{\frac{1}{2}}}F_{*}(\sigma\circ\Psi_{1/\rho})\,\nabla u\cdot\nabla u\, dy-\int_{B_{1}\setminus B_{\frac{1}{2}}}A_{\varepsilon}\nabla u_{\varepsilon}\cdot\nabla u_{\varepsilon}\, dy
=E_{r}+E_{t},
\end{align}
where $E_{r}$ is the integral with the radial components of $\nabla u_\varepsilon$ and $\nabla u$, and $E_{t}$ with the tangential components, that is,
\beq\label{eq:defABpfhom-1}
\begin{aligned}
E_{r} & :=\int_{\Omd}\left(\frac{1}{A_{\varepsilon}}-\frac{1}{\sigma^*_1}\right)A_{\varepsilon}\,\partial_{r}u_{\varepsilon}\,\sigma^*_1\, \partial_{r}u\, dy,\\
E_{t} & :=\int_{\Omd}\left(\sigma^*_2-A_{\varepsilon}\right)\nabla_{\phi}u_{\varepsilon}\cdot\nabla_{\phi}u\, dy.
\end{aligned}
\eeq
We estimate $E_{r}$ and $E_{t}$ by applying Lemma \ref{lem:weak_conv}, which requires bounds on $\left\Vert \partial_{r}f\right\Vert_{L^{\infty}\left(\Omega_{\delta}\right)}$ and $\left\Vert f\right\Vert _{L^{\infty}\left(\Omega_{\delta}\right)}$, with $$f=\nabla_{\phi}u_{\varepsilon}\cdot\nabla_{\phi}u,\qquad f=A_{\varepsilon}\,\partial_{r}u_{\varepsilon}\,\sigma^*_1\, \partial_{r}u.$$

First, by Lemma \ref{lem:rad-reg}, we easily obtain 
\begin{align*}
\left\Vert \nabla_{\phi}u_{\varepsilon}\cdot\nabla_{\phi}u\right\Vert _{L^{\infty}\left(\Omega_{\delta}\right)} & \le C\left\Vert \varphi\right\Vert_{H^{\frac{1}{2}}\left(\partial B_{1}\right)}.
\end{align*}
By applying Lemma \ref{lem:rad-reg} to these functions (see \eqnref{Lemma:harmonic:eeq}) together with (\ref{eq:est_lambda}), \eqnref{eq:lambda_sigma}, \eqnref{eq:L1_a_d=2}, \eqnref{lambda_est_3d}, \eqnref{eq:L1_a_d=3}, and \eqnref{est_siglam_d=3}, one obtains
\begin{align*}
 & \left\Vert \partial_{r}\left(\nabla_{\phi}u_{\varepsilon}\cdot\nabla_{\phi}u\right)\right\Vert _{L^{\infty}\left(\Omd\right)}\\
\leq\,  & C\bigg(\Big\Vert \frac{1}{A_{\varepsilon}}\Big\Vert _{L^{\infty}\left(\Omd\right)}\left(1+\left\Vert A_{\varepsilon}\right\Vert _{L^{1}\left(\Omd\right)}\right)+\Big\Vert \frac{1}{\sigma_1^*}\Big\Vert _{L^{\infty}\left(\Omd\right)}\left(1+\left\Vert \sigma_2^*\right\Vert _{L^{1}\left(\Omd\right)}\right)\bigg)\left\Vert \varphi\right\Vert ^{2}_{H^{\frac{1}{2}}\left(\partial B_{1}\right)}
\end{align*}
and
\begin{align*}
\left\Vert\partial_{r}\left(\nabla_{\phi}u_{\varepsilon}\cdot\nabla_{\phi}u\right)\right\Vert _{L^{\infty}\left(\Omd\right)}
 & \le\begin{cases}
\ds C\,\frac{1}{\alpha}\,\kappa\left|\ln\rho\right|\left\Vert \varphi\right\Vert ^{2}{}_{H^{\frac{1}{2}}\left(\partial B_{1}\right)}\quad & \text{for }d=2,\\[3mm]
\ds C\,\frac{1}{\alpha}\,\kappa\left(1+\varepsilon\left|\ln\rho\right|\right)\left\Vert \varphi\right\Vert ^{2}{}_{H^{\frac{1}{2}}\left(\partial B_{1}\right)} \quad & \text{for }d=3.
\end{cases}
\end{align*}

Similarly, one obtains
\begin{align*}
\left\Vert A_{\varepsilon}\,\partial_{r}u_{\varepsilon}\,\sigma^*_1\, \partial_{r}u\right\Vert _{L^{\infty}\left(\Omega_{\delta}\right)} & \le C\left\Vert A_{\varepsilon}\right\Vert _{L^{1}\left(\Omega\right)}\,\left\Vert \sigma^*_1 \right\Vert _{L^{1}\left(\Omega\right)}\,\left\Vert \varphi\right\Vert ^{2}_{H^{\frac{1}{2}}\left(\partial B_{1}\right)}\\
 & \leq\begin{cases}
\ds C\,\kappa^{2}\left|\ln\rho\right|^{2}\left\Vert \varphi\right\Vert^{2}_{H^{\frac{1}{2}}\left(\partial B_{1}\right)} \quad& \mbox{for }d=2,\\[2mm]
\ds C\,\kappa^{2}\left(1+\varepsilon\left|\ln\rho\right|\right)\left\Vert \varphi\right\Vert^{2}_{H^{\frac{1}{2}}\left(\partial B_{1}\right)} \quad &\mbox{for } d=3.
\end{cases}
\end{align*}
Recall that in {\bf Case 2} of Section \ref{subsec:homogenization}, it may be necessary to use more than one highly conducting materials $\gamma$ satisfying \eqnref{eq:parameter_conditions2}. We then define $\gamma_{\max}$ to be the largest conductivity among the values of $\gamma$ determined by \eqnref{eq:parameter_conditions2}.

We again apply Lemma \ref{lem:rad-reg} to the right-hand side of the product rule
\begin{align*}
&\left\Vert \partial_{r}\left(A_{\varepsilon}\partial_{r}u_{\varepsilon}\, \sigma^*_1\,\partial_{r}u\right)\right\Vert_{L^{\infty}\left(\Omega_{\delta}\right)} \\
\leq & \left\Vert \sigma^*_1\partial_{r}u\right\Vert _{L^{\infty}\left(\Omega_{\delta}\right)}\left\Vert \partial_{r}\left(A_{\varepsilon}\partial_{r}u_{\varepsilon}\right)\right\Vert_{L^{\infty}\left(\Omega_{\delta}\right)}
+\left\Vert A_{\varepsilon}\partial_{r}u_{\varepsilon}\right\Vert _{L^{\infty}\left(\Omega_{\delta}\right)}\left\Vert \partial_{r}\left(\sigma^*_1\partial_{r}u\right)\right\Vert _{L^{\infty}\left(\Omega_{\delta}\right)}
\end{align*}
and derive
\begin{align*}
\left\Vert \partial_{r}\left(A_{\varepsilon}\partial_{r}u_{\varepsilon}\, \sigma^*_1\,\partial_{r}u\right)\right\Vert_{L^{\infty}\left(\Omega_{\delta}\right)} 
\leq &
\begin{cases}
\ds C\gamma_{\max}\,\kappa\,\left|\ln\rho\right|\left\Vert \varphi\right\Vert ^{2}_{H^{\frac{1}{2}}\left(\partial B_{1}\right)} \quad & \mbox{for }d=2,\\[2mm]
\ds C\gamma_{\max}\,\kappa\,\left(1+\varepsilon\lvert \ln \rho \rvert \right)\left\Vert \varphi\right\Vert ^{2}_{H^{\frac{1}{2}}\left(\partial B_{1}\right)} \quad &\mbox{for } d=3.
\end{cases}
\end{align*}

Finally, by applying Lemma \ref{lem:weak_conv} to \eqnref{eq:defABpfhom-1} with the estimates derived above, we obtain 
\begin{equation}\label{estimate:Er+Et}
 \begin{aligned}
E_{r}+E_{t}  \le 
\begin{cases}
\ds 
C\varepsilon\kappa^{2}\lvert \ln \rho \rvert^2\left(\frac{1}{\alpha}+\gamma_{\max}\right)\left\Vert \varphi\right\Vert ^{2}{}_{H^{\frac{1}{2}}\left(\partial B_{1}\right)}
\quad&\mbox{for }d=2,\\
\ds C\varepsilon\kappa^{2}\left(\frac{1}{\alpha} a_{\rho,\varepsilon}^2+\gamma_{\max} a_{\rho,\varepsilon}b_{\rho,\varepsilon}+a_{\rho,\varepsilon}\kappa \frac{\lvert \ln \rho \rvert}{\rho }\right)\left\Vert \varphi\right\Vert ^{2}_{H^{\frac{1}{2}}\left(\partial B_{1}\right)}\quad &\mbox{for }d=3.
 \end{cases}
\end{aligned}   
\end{equation}

This completes the proof by \eqnref{int:Er:Et}.
\end{proof}

We are now ready to prove Theorem \ref{thm:main}.
\begin{proof}[\foreignlanguage{american}{Proof of Theorem\foreignlanguage{british}{ \ref{thm:main}} }]
\label{proof:thm1} Let $\rho_{EC}$ be the radius of the small inclusion
in the enhanced cloaking scheme with GPT-vanishing structure of order $N$. From Theorem
\ref{thm:cloaking_gpt}, we have
\[
\left\lVert\Lambda_{B_{1},B_{\frac{1}{2}}}\left[
F_{*}(\sigma\circ\Psi_{1/\rho})\right]-\Lambda_{B_{1},\emptyset}\left[1\right]\right\rVert_{\mathcal{L}\left(H^{\frac{1}{2}}\left(\partial B_{1}\right),H^{-\frac{1}{2}}\left(\partial B_{1}\right)\right)}\leq C\rho_{EC}^{d+2N}
\]
and, by Proposition \ref{prop:homres},
 \begin{equation*}
 \begin{aligned}
&\left\lVert\Lambda_{B_{1},B_{\frac{1}{2}}}\left[A_\varepsilon
\right]-  \Lambda_{B_{1},B_{\frac{1}{2}}}\left[F_{*}(\sigma\circ\Psi_{1/\rho})\right]\right\rVert_{\mathcal{L}\left(H^{\frac{1}{2}}\left(\partial B_{1}\right),H^{-\frac{1}{2}}\left(\partial B_{1}\right)\right)} \\ 
&\le 
\begin{cases}
\ds 
C\varepsilon\kappa^{2}\lvert \ln \rho_{EC}\rvert^2\left(\frac{1}{\alpha}+\gamma_{\max}\right)
\quad&\mbox{for }d=2,\\
\ds C\varepsilon\kappa^{2}\left(\frac{1}{\alpha} a_{\rho_{EC},\varepsilon}^2+\gamma_{\max} a_{\rho_{EC},\varepsilon}b_{\rho_{EC},\varepsilon}+a_{\rho_{EC},\varepsilon}\kappa \lvert \ln \rho_{EC} \rvert \rho_{EC}^{-1}\right)\quad &\mbox{for }d=3.
 \end{cases}
\end{aligned}   
\end{equation*}
To match the level of invisibility $\rho^d$ of the near-cloak obtained by the regular blow-up transformation without the GPT-vanishing coating,
we can set
\begin{equation}\label{rho_enh}
  \rho_{EC}=\rho^{\frac{d}{d+2N}}.  
\end{equation}

It is worth emphasizing that $$\rho_{EC}\gg \rho.$$

Taking conditions \eqnref{eq:boundalpha}  and \eqnref{eq:boundgamma}, we deduce that $\alpha=O\left(\kappa|\ln \rho|^{-1} \rho^{\frac{d}{d+2N}(d-2)}\right)$ and $\gamma_{\max}=O\left(\kappa |\ln \rho|\right)$. 
Set
\begin{equation}\label{epsilon_level}
\varepsilon=\begin{cases}
O\left(\rho^2\left|\ln\rho\right|^{-3}\kappa^{-3}\right) & \text{when }d=2\\
O\left(\rho^{3+\frac{3}{2N+3}}\left|\ln\rho\right|^{-1}\kappa^{-1}\right) & \text{when }d=3
\end{cases}
\end{equation}

 one has
\[
\lVert\Lambda_{B_{1},B_{\frac{1}{2}}}\left[A_{\varepsilon}\right]-\Lambda_{B_{1},B_{\frac{1}{2}}}\left[F_{*}\sigma\right]\rVert_{\mathcal{L}\left(H^{\frac{1}{2}}\left(\partial B_{1}\right),H^{-\frac{1}{2}}\left(\partial B_{1}\right)\right)}\leq C\rho^{d}.
\]
The conclusion follows from the triangle inequality.
\end{proof}

\section{A remark on the case with arbitrary inclusion}\label{sec:arbitrary_inclusion}

Theorem \ref{thm:main} establishes a cloaking scheme based on
GPT-vanishing coating structures, specifically designed to conceal the
perfectly insulated ball $B_{\frac{1}{2}}$. We note, however, that this
result only applies when the conductivity in the “hidden” region is zero. A complementary strategy was developed by H. Heumann and M.S. Vogelius \cite{Heumann:2014:AEA} and extended this cloaking approach to a more general
transmission setting, thereby enabling the cloaking of arbitrary objects within $B_{\frac{1}{4}}$. Their construction introduces a thin layer of material with very low conductivity in the annular region $B_{\rho_{}}\setminus B_{\frac{\rho}{2}}$ (in the virtual
domain).

In this section, we aim to construct
near-cloak structures analogous to those in \cite{Heumann:2014:AEA}, but based on the homogenization techniques developed in
Section \ref{sec:laminate}. For simplicity, we restrict the discussion to the two-dimensional case ($d=2$).

\subsection{Review of cloaking arbitrary an inclusion}
Set $N\in\NN$. Let $\sigma_N$, which corresponds to $\sigma\circ\Psi_{1/\rho}$, denote the multi-coating having an insulating core in the virtual domain (i.e., before applying the push-forward mapping $F_*$) with GPT-vanishing properties up to order $N$, constructed in Section \ref{sec:GPT}. 
Let $\zeta>0$ be a small parameter, and define the conductivity distribution $\sigma_{\zeta}$ by
\beq\label{def:sigma_zeta}
\sigma_{\zeta}:=\begin{cases}
\sigma_N\quad& \text{in }B_{1}\setminus B_{\rho^{}},\\
\zeta \quad & \text{in }B_{\rho^{}}\setminus B_{\frac{\rho}{2}}\\
\widetilde{B} \quad& \text{in }B_{\frac{\rho}{2}},
\end{cases},
\eeq
where $\widetilde{B}\in L^{\infty}\left(B_{\frac{\rho}{2}}\right)$
is a symmetric, positive definite matrix. Recall the  radial transformation
$F$ introduced in (\ref{eq:diffeo_radial}). The resulting  the near-cloak imposed in $B_{1}\setminus B_{\frac{1}{4}}$ (in the physical domain) is given by $F_{*}\sigma_{\zeta}$. 
The anisotropy measure given in \eqnref{def:chi_max} becomes
\[
\chi_{\text{max}}=\max_{x\in\bar{B}_{1}\setminus B_{\frac{1}{4}}}\frac{\tau_{2}}{\tau_{1}}(x),
\]
where for each $x\in\bar{B}_{1}\setminus B_{\frac{1}{4}}$, $\tau_{1}(x):=\frac{\sigma_{\zeta}}{\lambda\left(\left|x\right|\right)}$
and $\tau_{2}(x):=\sigma_{\zeta}\lambda\left(\left|x\right|\right)$,
are the eigenvalues of $F_{*}\left[\sigma_{\zeta}\right]$ with 
\begin{align*}
\lambda(s)=\begin{cases}
\ds1, & s<\frac{1}{2},\\
\ds\frac{1}{\alpha(\rho)}, & \frac{1}{2}<s<\frac{3}{4},\\
\ds1, & \frac{3}{4}<s<1.
\end{cases}
\end{align*}
With respect to $\chi_{\text{max}}$, $F_{*}\left[\sigma_{\zeta}\right]$
configuration has the same level of anisotropy with the near-cloak $F_*\sigma$ in
\eqnref{eq:cloak}. However, this comes at the cost of 
\[
\max_{x\in\bar{B}_{1}\setminus B_{\frac{1}{4}}}\frac{1}{\tau_{1}}=\rho_{EC}^{-\left(2N+2\right)}.
\]

The following cloaking result from \cite{Heumann:2014:AEA} states that the near-cloak $F_{*}\sigma_{\zeta}$, imposed in $B_{1}\setminus B_{\frac{1}{4}}$ (in the physical domain), conceals the arbitrary conducting material $$B:=F_{*}\widetilde{B}$$ (also symmetric and positive definite)
placed in the cloaking region $B_{\frac{1}{4}}$. 

\begin{theorem}[\cite{Heumann:2014:AEA}]\label{thm:HV:estimates}
Set $N\in\NN$ and $\sigma_\zeta$ be given by \eqnref{def:sigma_zeta} with arbitrary symmetric, positive definite matrix $\widetilde{B}$ in the core. Set $\zeta=\rho^{2N+2}$. Given $\varphi\in H^{\frac{1}{2}}(\p B_1)$, let $u_{\zeta}$ and $U$ be the solutions to 
\[
\begin{cases}
\nabla\cdot\left(F_{*}\sigma_{\zeta}\,\nabla u_{\zeta}\right)=0 \quad& \text{in }B_{1},\\
u_{\zeta}=\varphi \quad & \text{on }\partial B_{1}
\end{cases}
\]
and
\[
\begin{cases}
\Delta U=0 \quad& \text{in }B_{1},\\
U=\varphi \quad& \text{on }\partial B_{1},
\end{cases}
\]
respectively. We denote by $\Lambda_{B_1}[F_{*}\sigma_{\zeta}]$ and $\Lambda_{B_1}[1]$ the DtN map associated with these problems. 
Then, it holds that
\[
\left\Vert \Lambda_{B_1}[F_{*}\sigma_{\zeta}]-\Lambda_{B_1}[1]\right\Vert _{\mathcal{L}\left(H^{\frac{1}{2}}\left(\partial B_{1}\right),H^{-\frac{1}{2}}\left(\partial B_{1}\right)\right)}\le C\rho^{2N+2}.
\]
\end{theorem}

\subsection{Near-cloaking laminate structure for arbitrary objects} 

From the definition of diffeomorphism $F$ in \eqnref{eq:diffeo_radial}, $F_{*}\left[\sigma_{\zeta}\right]=\zeta$ in  $B_{\frac{1}{2}}\setminus B_{\frac{1}{4}}$, thus isotropic in that region. Define 
\[
A_{\varepsilon}^{\zeta}=\begin{cases}
B&\text{ in } B_{\frac{1}{4}},\\
\zeta & \text{in }B_{\frac{1}{2}}\setminus B_{\frac{1}{4}},\\
A_{\varepsilon} & \text{in }B_{1}\setminus B_{\frac{1}{2}},
\end{cases}
\]
where $A_{\varepsilon}$ is the piecewise constant and isotropic laminate
material proposed in (\ref{eq:Aeps_def}). We state in the following result that $A_{\varepsilon}^{\zeta}$ can cloak arbitrary material
in $B_{\frac{1}{4}}$.
The proof follows from the proof of Theorem \ref{thm:main} discussed in Section \ref{sec:main_result}.

\begin{theorem} Let   $\varphi$ and $U$ be the functions defined in Theorem \ref{thm:HV:estimates}.
 Let $u^\zeta_\varepsilon$ be the solution of 
 \[
\begin{cases}
\nabla\cdot\left(A^\zeta_\varepsilon\,\nabla u^\zeta_{\varepsilon}\right)=0 \quad& \text{in }B_{1},\\
u^\zeta_{\varepsilon}=\varphi \quad & \text{on }\partial B_{1}.
\end{cases}
\]
Then, taking $\varepsilon$ to be in \eqnref{epsilon_level} and the size of the small hole in the blow-up transformation by $F$ to be $\rho_{EC}=\rho^{\frac{1}{1+N}}$, we have 
\[
\left\Vert \Lambda_{A^\zeta_\varepsilon}-\Lambda_{1}\right\Vert _{\mathcal{L}\left(H^{\frac{1}{2}}\left(\partial B_{1}\right),H^{-\frac{1}{2}}\left(\partial B_{1}\right)\right)}\le C\rho^{2}.
\]
\end{theorem}

This means that at the price of adding a low
conducting material $\zeta=\rho_{EC}^{2N+2}=\rho^{2}$ in $B_{\frac{1}{2}}\setminus B_{\frac{1}{4}}$,
$A_{\varepsilon}^{\zeta}$ is an approximate cloak of order $\rho^2$. 

\section{Illustrative examples}\label{sec:Illustrative-examples}

In this section, we present numerical examples illustrating the near-cloak laminates constructed in Section \ref{subsec:homogenization} for a perfectly insulating inclusion in dimension $d=2$. 

\subsection{Computational setup}
We construct the near-cloaking laminates through the following steps.
Recall that we use the coordinate transformation $y=F(x)$ given by \eqnref{eq:diffeo_radial}, which maps the virtual domain $B_1$ (using coordinates $x$) onto the physical domain $B_1$ (using coordinates $y$), where the transformation is designed to expand the insulating cores is $B_\rho$ in the virtual domain and to $B_{\frac{1}{2}}$ in the physical domain. 

\begin{enumerate}[label=$\bullet$\, {Step \arabic*.}, leftmargin=3.5em, align=left]
  \item Numerical computation of a GPT-vanishing structure $\sigma$ with an insulating core.
  \item (Theoretical) Define the small coated structure $(\sigma\circ\Psi_{1/\rho})(x)$.
  \item (Theoretical) Define the anisotropic structure $F_*\left(\sigma\circ\Psi_{1/\rho}\right)(y)$.
  \item Numerical computation of a near-cloaking laminate in the physical domain that approximates $F_*\left(\sigma\circ\Psi_{1/\rho}\right)(y)$.
\end{enumerate}

\smallskip

In {Step 1}, given a prescribed number of coating layers $L$, we take a uniform partition $(r_1,\dots,r_{L+1})$ of the interval $\left[1,2\right]\subset\RR$, defined by $$r_{j}=1+\tfrac{j-1}{L}, \quad j=1,\dots,L+1.$$ 
For a target order $N\in\NN$, we numerically compute the conductivity profile $\sigma=\left(\sigma_{1},\dots,\sigma_{L}\right)$, more precisely, a function of the form \eqnref{eq:coating_sigma:reference}, such that the corresponding coated disk in $\RR^2$ has vanishing CGPTs up to order $N$. In other words, $\sigma$ is a solution of the non-linear system
$$
M_{k}\left[\sigma\right]=0\quad\text{for }k=1,\dots,N, 
$$
where $M_{k}$ denotes the $k$-th CGPT defined in \eqnref{eq:expansion_d=2_simple}.

Following \cite{Ammari:2013:ENC_1}, we employ a gradient descent scheme to find a joint root of the CGPT values $M_k[\sigma]$ for $k\leq N$, iteratively updating $\sigma$ {color{magenta}starting from  an initialization $\sigma_j=2^{(-1)^j}$, $j=1,\dots,L$.}

\smallskip

In {Step 2} and {Step 3}, we fix a parameter $\rho>0$ and define $F_*\left(\sigma\circ\Psi_{1/\rho}\right)(y)$ as in Section~\ref{section:pushforward}, where $\sigma$ is the GPT-vanishing structure computed in {Step 1}.

\smallskip

In Step 4, we fix the parameter $\varepsilon$ and compute the eigenvalues $\sigma_1^*(s_k)$ and $\sigma_2^*(s_k)$ for each $k=0,\dots, N_\varepsilon-1$, where the number $N_\varepsilon$ and the sampling points $s_k$ are defined depending on $\varepsilon$ as outlined in Section~\ref{subsec:homogenization}. 
We then use these eigenvalues to select the constituent isotropic conductivities, $\alpha$ (low conducting material) and $\gamma$ (one or more high conducting materials), according to the homogenization scheme detailed in Section~\ref{subsec:homogenization}. 

\subsection{Examples}
In all examples, we set the GPT-vanishing order $N=4$ or $6$, take the number of coating layers $L$ to be equal to $N$, and fix the scale of lamination at $\varepsilon=\tfrac{1}{50}$ and the near-cloak parameter $\rho=0.0001$.

\subsubsection{Non-coated insulating core}\label{subsub:non_coated}
To establish a baseline for comparison, we first consider a non-coated structure characterized by 
$$\sigma(x)=\chi_{\RR^2\setminus \overline{B_1}}\,(x),\quad (\sigma\circ\Psi_{1/\rho})\big|_{B_1}(x)=\chi_{\{\rho \le |x|\le 1\}}\,(x).$$

Then the radial and tangential directional eigenvalues of the cloaking structure $F_*(\sigma\circ\Psi_{1/\rho}))(y)$ in the physical variable $y$, with $s=|y|\in[\tfrac{1}{2},1]$, are respectively
\beq\label{sigma_star:noncoating}
\begin{aligned}
\sigma_1^*(y)&=\frac{\sigma(\rho^{-1}F^{-1}(y))}{\lambda(s)}=\frac{1}{\lambda(s)},\\
\sigma_2^*(y)&=\sigma(\rho^{-1}F^{-1}(y))\lambda(s)={\lambda(s)},
\end{aligned}
\eeq
where $\lambda$ is given by \eqnref{eq:est_lambda}. Note that $1/\lambda$ and $\lambda$ and are the radial and tangential directional eigenvalues of $F_* 1$, the push-forward of the constant conductivity, respectively.
The eigenvalues of $\lambda$  and ${1}/{\lambda}$ are presented in Figure \ref{fig:eigen_noncoated}. 

In view of condition \eqnref{eq:parameter_conditions1}, we can choose $\alpha>0$ satisfying $-1<\alpha<1/\max(\lambda)$, that is, $0<\alpha <0.0454$. With the choice $\alpha=0.0227$, \eqnref{eq:parameter_conditions12} reduces to $\gamma > 43.9665$. Figure \ref{fig:lam_noncoated} shows a cloaking laminate without coating, given by \eqnref{eq:Aeps_def}, constructed from three isotropic materials with conductivities $\alpha(=0.0227)$, $1$, and $\gamma(=65.9498)$, and lamination scale $\varepsilon=\tfrac{1}{50}$.

\begin{figure}[h!]
	\centering
	\begin{subfigure}{1\textwidth}
		\centering
		\begin{minipage}{0.48\linewidth}
			\centering
			\includegraphics[width=\linewidth]{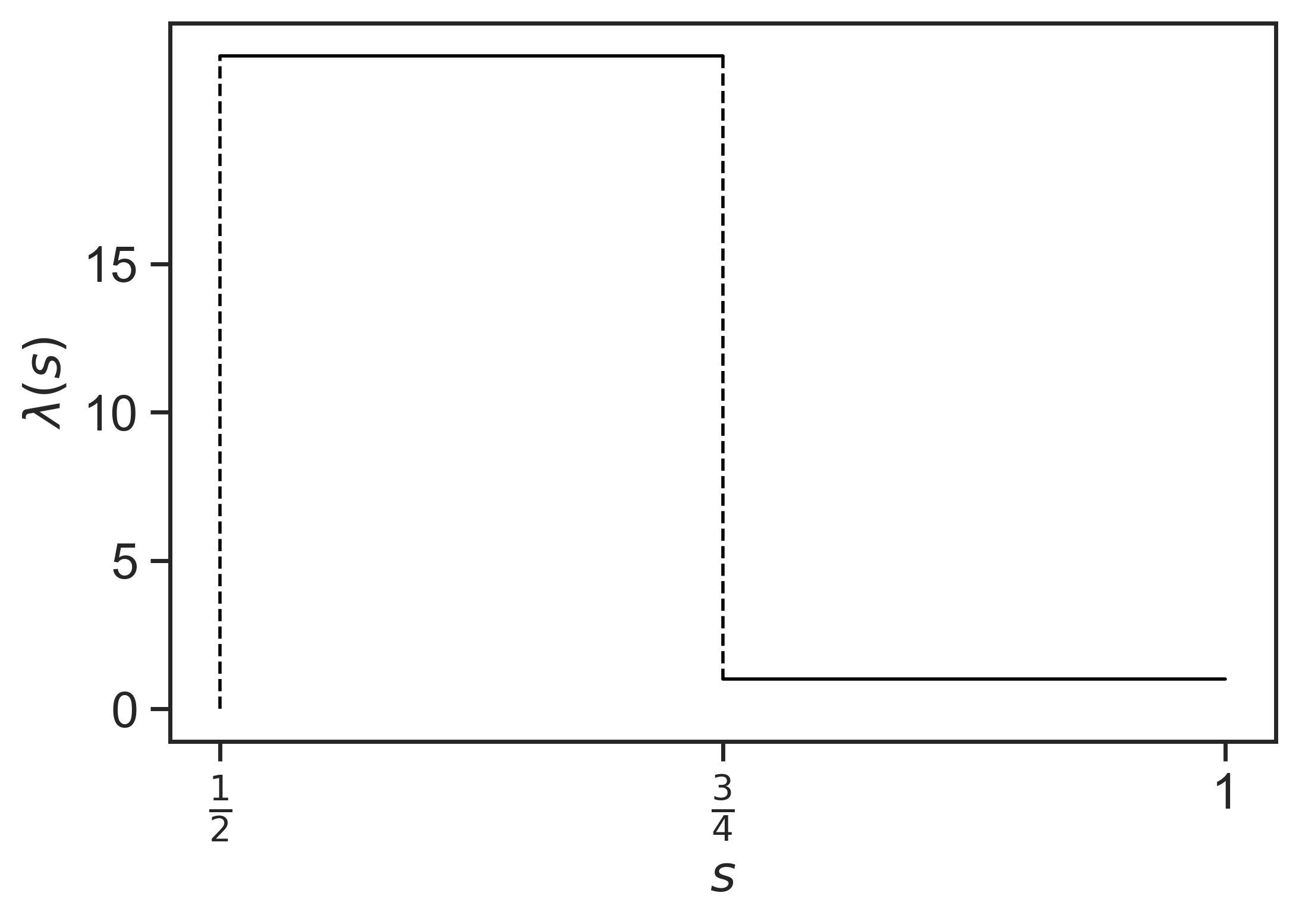}
		\end{minipage}\hfill
		\begin{minipage}{0.48\linewidth}
			\centering
			\includegraphics[width=\linewidth]{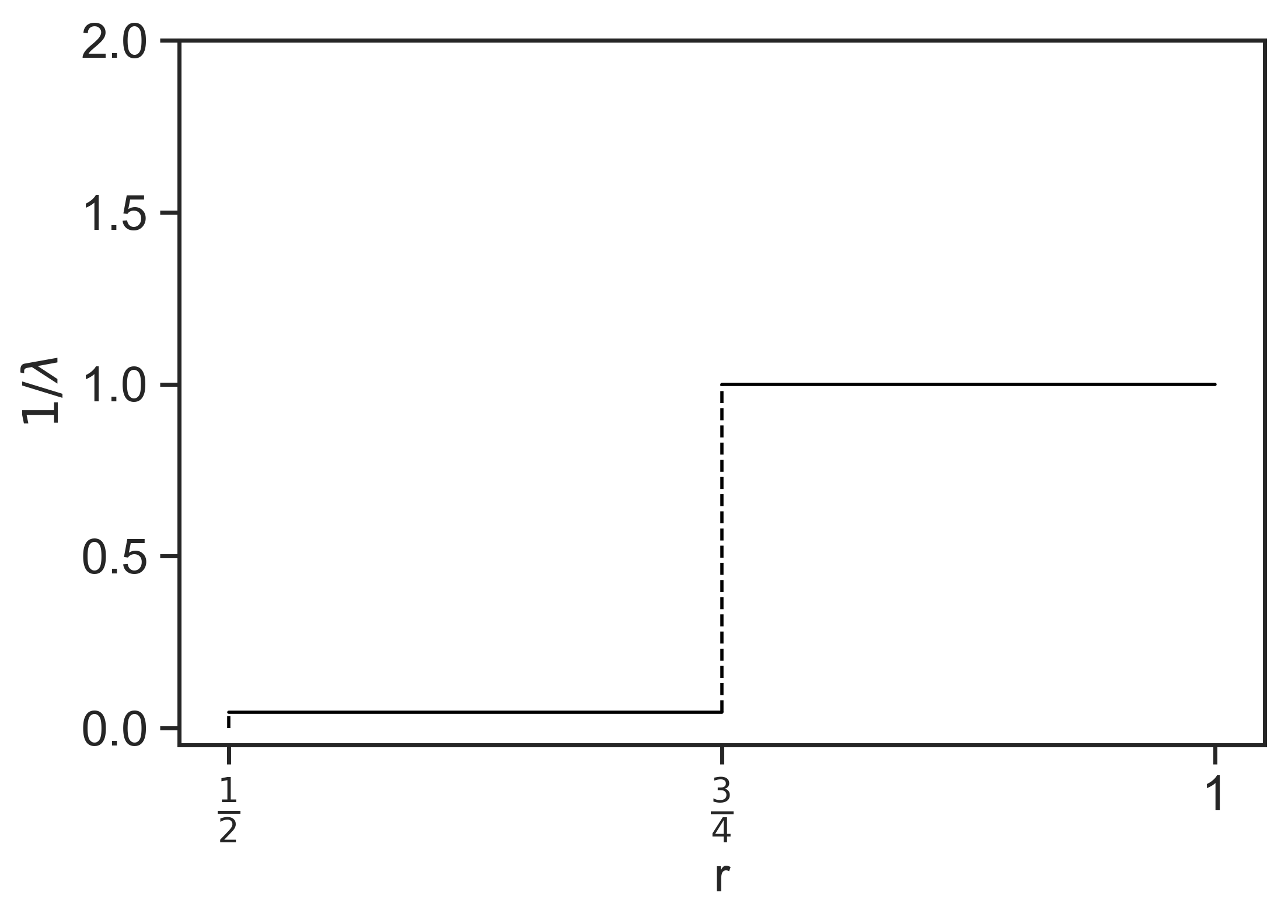}
		\end{minipage}
		\caption{The tangential (left) and radial (right) eigenvalues $\lambda(y)$ and $1/\lambda(y)$ of $(F_{*}1)(y)$ with the physical variable $y$. For $s=|y|\in[\frac{1}{2},1]$, $\sigma_1^*(y)=1/\lambda(y)$ and $\sigma_2^*(y)=\lambda(y)$  (see \eqnref{sigma_star:noncoating}).}
		\label{fig:eigen_noncoated}
	\end{subfigure}
	\vspace{0.5cm}
	\begin{subfigure}{1\textwidth}
		\centering
		\includegraphics[width=0.6\linewidth]{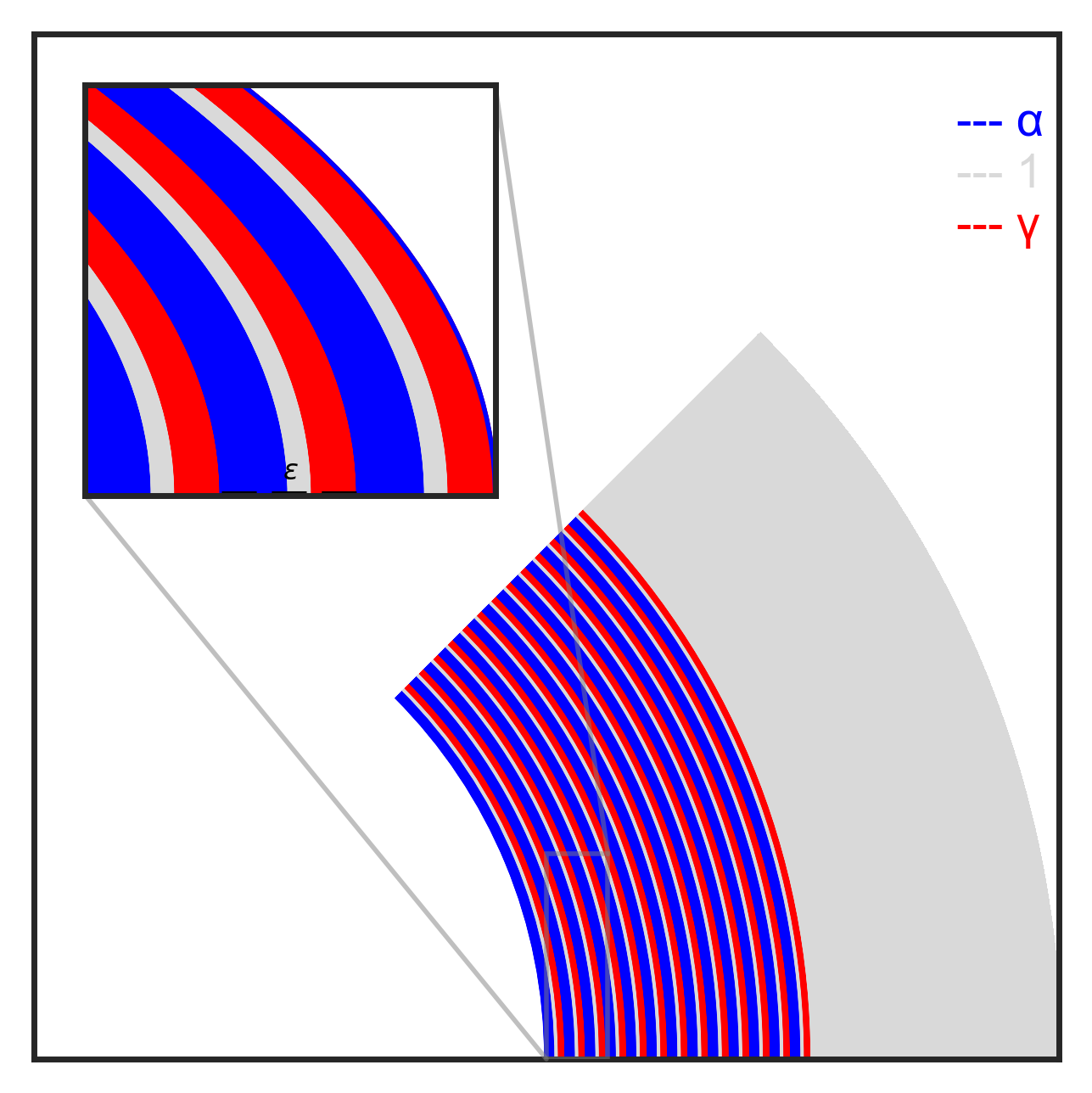}    
		\caption{A laminate in the annulus $\{\frac{1}{2}\leq |y|\leq \frac{3}{4}\}$ with a uniform lamination scale $\varepsilon=\tfrac{1}{50}$ and three isotropic materials with conductivities $\alpha(=0.0227)$, $1$, and $\gamma(=65.9498)$.}
		\label{fig:lam_noncoated}
	\end{subfigure}
	\caption{Cloaking laminate using a non-coated insulated core in $\RR^2$.}
\end{figure}

\subsubsection{Cloaking laminate using GPT-vanishing structure of order $N=4$}

Figure \ref{fig:L=4} shows an illustration for $N=L=4$ number of coating layers in two dimensions. 
In Figure \ref{fig:sigma_L=4}, we observe a fluctuating behaviour of the conductivities of the multi-coat with a maximum value of $7.6021$ (leftmost layer) and minimum of $0.2811$ (second layer from the left). 
We use the small parameter $\rho_{EC}$ computed by \eqnref{rho_enh} in the blow-up transformation. Specifically, $\rho=0.0001$ and $L=4$, we have $\rho_{EC}=0.1585$.

Figure \ref{fig:eigen_L=4} shows the tangential and radial eigenvalues of $F_{*}(\sigma\circ\Psi_{1/\rho})$.
Recall that \eqnref{eq:alpha_cond} restricts the permissible values of $\alpha$. In this particular example, this corresponds to the interval $0<\alpha<0.0734$.
Consequently, we achieve a slight relaxation of the constraints on the low-conducting constituent material compared to the non-coated case in Section \ref{subsub:non_coated}.
We take  $\alpha=0.05$. Given this choice of $\alpha$, we determine the permissible range for $\gamma$. In the leftmost layer where $\sigma$ is maximal, we have $\sigma_1^*>1$; consequently, condition \eqnref{eq:parameter_conditions2} implies that $29.8458<\gamma<56.7726$ for this layer. In all other regions where $\sigma_1^*<1$, \eqnref{eq:parameter_conditions12} requires $\gamma>7.7902$. To satisfy these constraints, we select $\gamma=43.3092$. The resulting laminate structure, composed of conductivity phases $\alpha$, $1$, and $\gamma$ with a uniform scale $\varepsilon=\tfrac{1}{50}$, is illustrated in Figure \ref{fig:lam_enh_L=4}.
\begin{figure}
    \centering
    \begin{subfigure}{1\textwidth}
        \centering
        \includegraphics[width=0.5\linewidth]{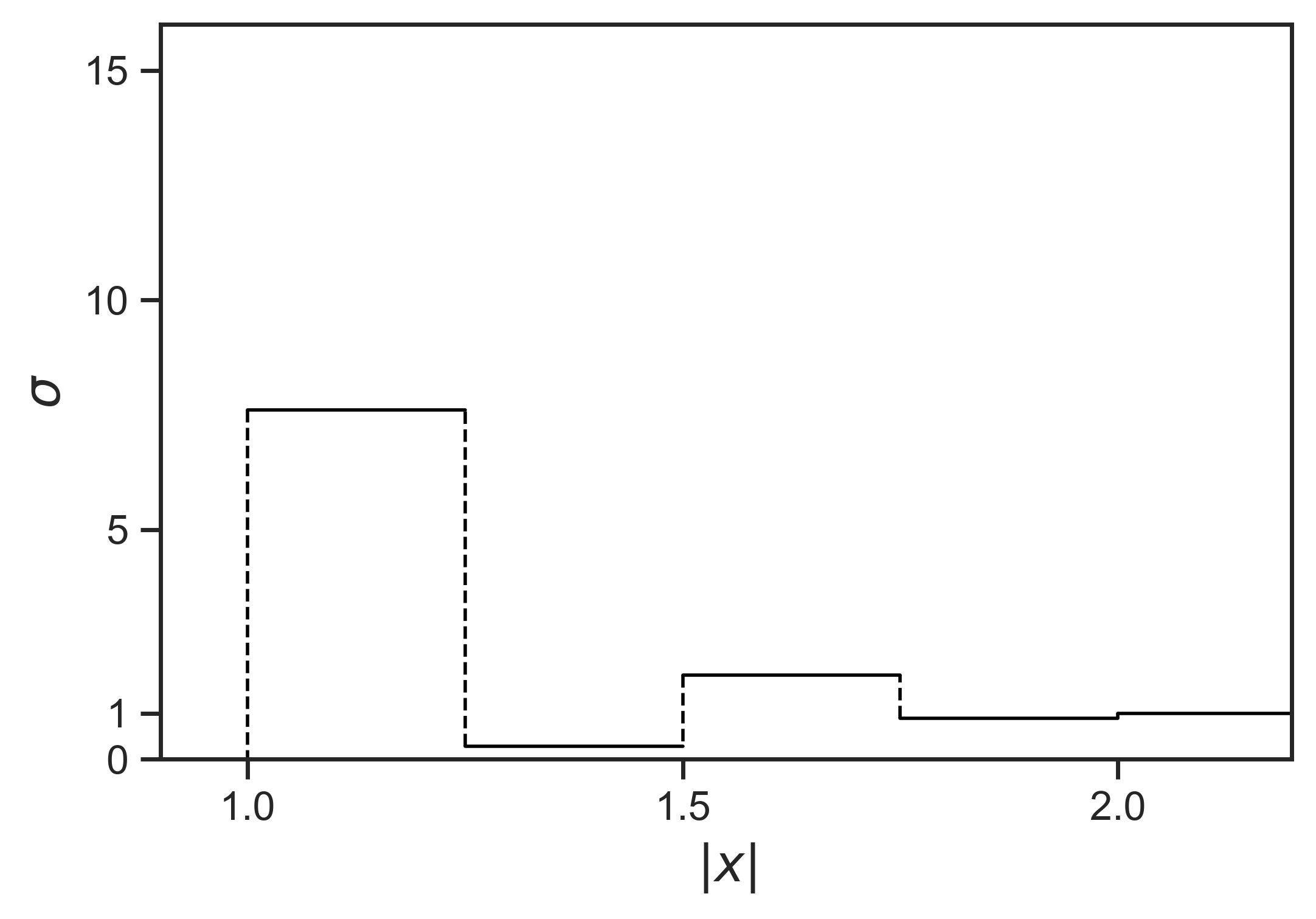} 
        \caption{Conductivity profile with $L=4$ for an insulated inclusion}
        \label{fig:sigma_L=4}
    \end{subfigure}
    
     \vspace{0.5cm}
     
  \begin{subfigure}{1\textwidth}
    \centering
   \begin{minipage}{0.48\linewidth}
        \centering
        \includegraphics[width=\linewidth]{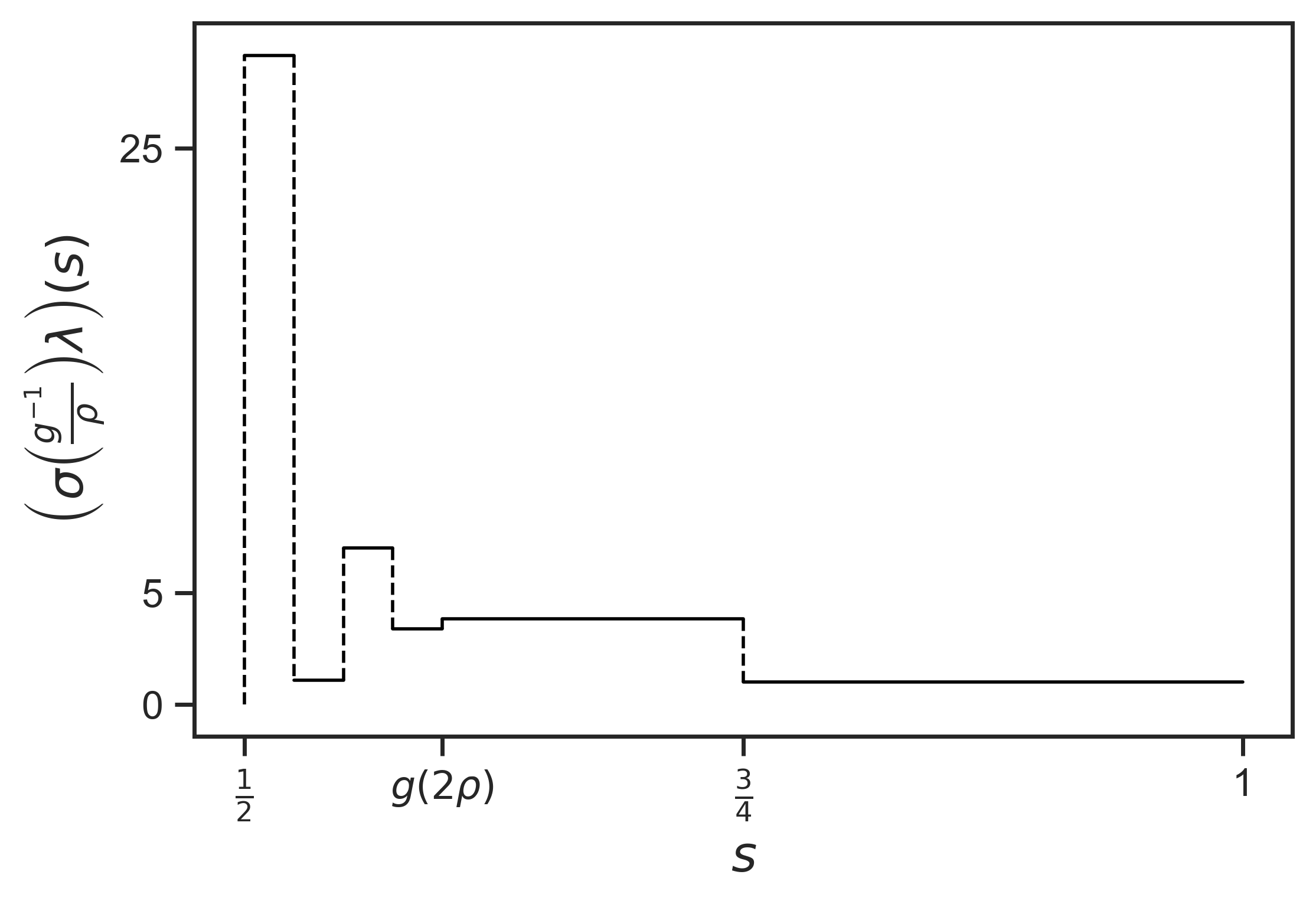}
    \end{minipage}\hfill
    \begin{minipage}{0.48\linewidth}
        \centering
        \includegraphics[width=\linewidth]{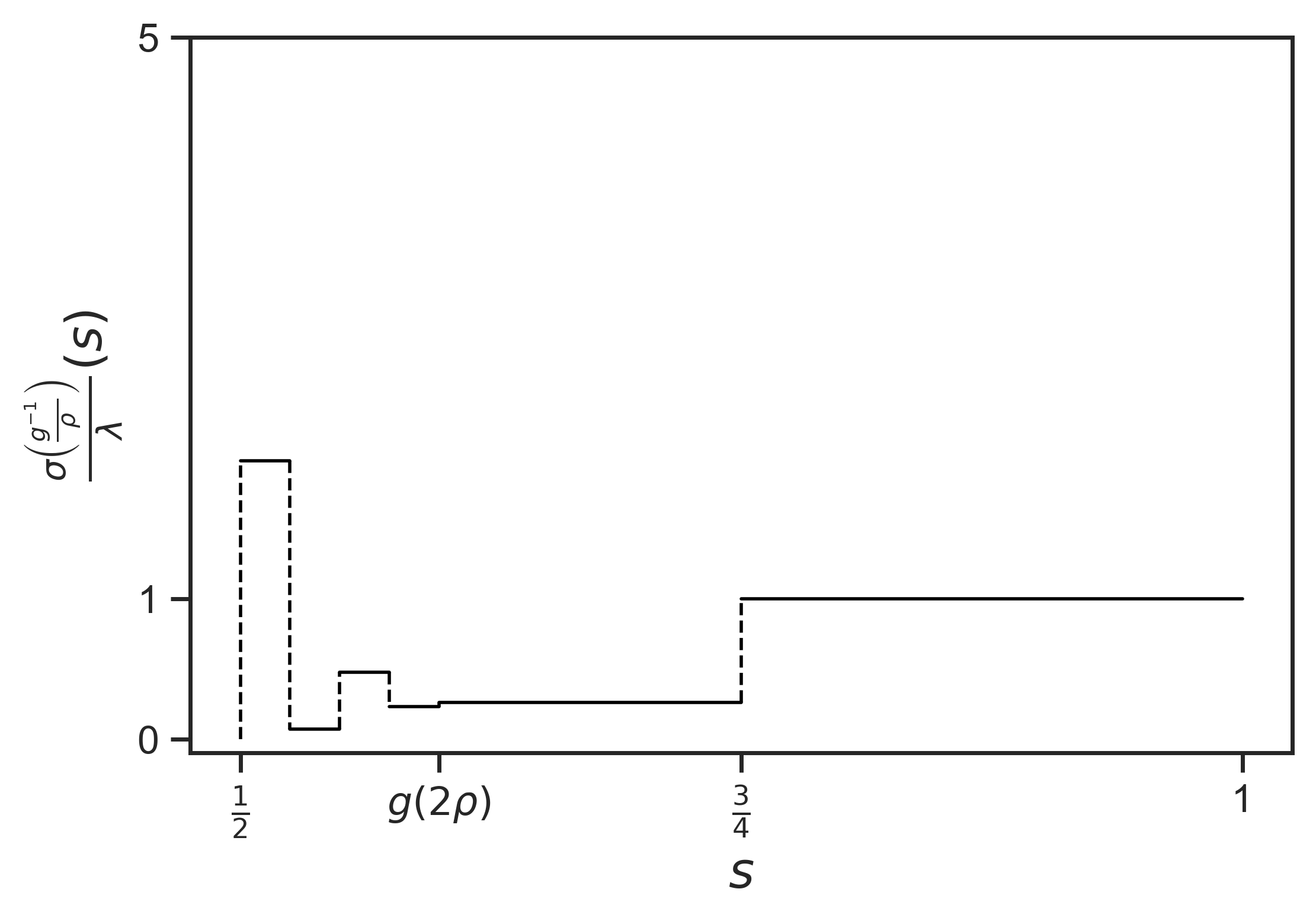}
    \end{minipage}
    \caption{The tangential (left) and radial (right) eigenvalues of $F_{*}(\sigma\circ\Psi_{1/\rho})$}
    \label{fig:eigen_L=4}
\end{subfigure}
     \vspace{0.5cm}  
    \begin{subfigure}{1\textwidth}
        \centering
        \includegraphics[width=0.6\linewidth]{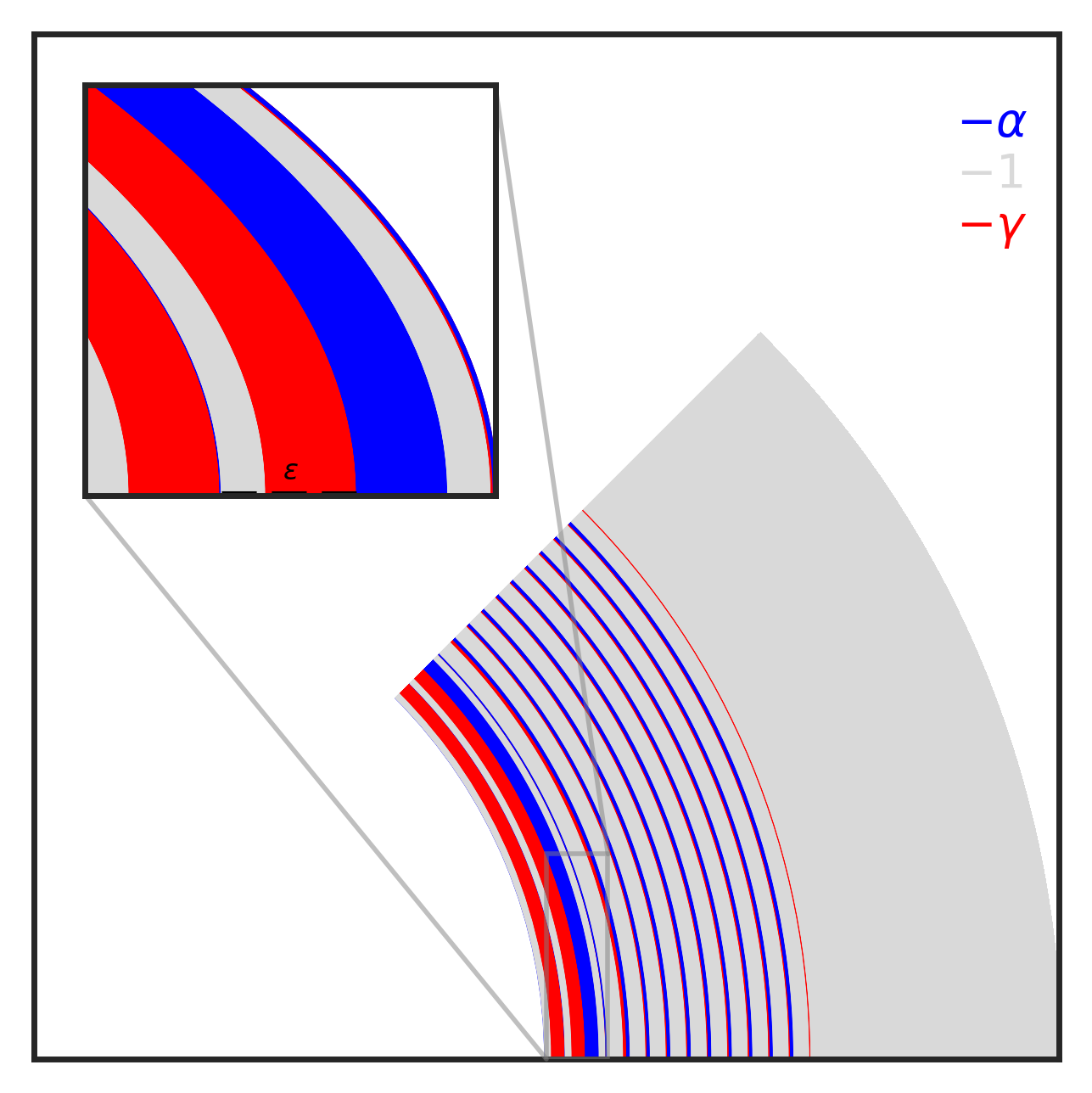}    
        \caption{A laminate with with a uniform scale and three materials}
        \label{fig:lam_enh_L=4}
    \end{subfigure}
    \caption{Cloaking laminate based on a GPT-vanishing structure of order $N=4$ in $\RR^2$.}
\label{fig:L=4}
\end{figure}

\subsubsection{Near-Cloaking laminate using GPT-vanishing structure of order $N=6$}
 For the case of $L=6$ layers in two dimensions, Step 1 of the computational setup yields a GPT-vanishing structure with a maximum conductivity of $11.6827$ and a minimum of $0.1706$. Following \eqnref{rho_enh}, we use $\rho_{EC}=0.2683$. By \eqnref{eq:alpha_cond}, one finds the permissible range for $\alpha$ as $0.0409<\alpha<0.0673$. We select $\alpha=0.05$. As shown in Figure \ref{fig:eigen_L=6}, the first and third coating layers (from the left) exhibit $\sigma_1^*>1$. Condition \eqnref{eq:parameter_conditions2} implies that $29.8968<\gamma<36.5563$ on the leftmost layer, and $9.4240<\gamma<27.4342$ for the third layer. For all other layers where $\sigma_1^*<1$, \eqnref{eq:parameter_conditions12} requires $\gamma>5.5576$. To satisfy these constraints, we employ two highly conducting isotropic materials; we choose $\gamma^{(1)}=32$ for the leftmost layer and $\gamma^{(2)}=15$ for the remaining regions. The resulting laminate structure is illustrated in Figure \ref{fig:lam_enh_L=6}. 

\subsubsection{Near-cloaking laminate in dimensions $d=3$ with $N=3$}
For the three-dimensional case ($d=3$) with $\rho=0.0001$, the non-coated cloaking laminate requires a low conducting isotropic material with $\alpha< 9.088 \times 10^{-6}$. Employing the multi-coating technique relaxes this constraint, for instance, allowing $0.0054<\alpha<0.0097$ for the GPT-vanishing structure of order $N=3$. The conductivity profile $\sigma$ with $L=3$ layers is shown in Figure \ref{fig:sigma_d=3}. The eigenvalues $\sigma_1^*$ and $\sigma_2^*$ of $F_*(\sigma\circ\Psi_{1/\rho})$ are presented in Figure \ref{fig:eigen_d=3}. Since $\sigma_1^*<1$ in the cloaking structure, we only need three constituent isotropic materials with conductivities $\alpha$ (low), $1$, and $\gamma$ (high), given by conditions   \eqnref{eq:parameter_conditions1} and \eqnref{eq:parameter_conditions12}. 

The laminate in Figure \ref{fig:d=3} is composed of layers of $\alpha=0.0075$, $1$, and $\gamma=10.8401$.

\begin{figure}
    \centering
    \begin{subfigure}{\textwidth}
        \centering
        \includegraphics[width=0.5\linewidth]{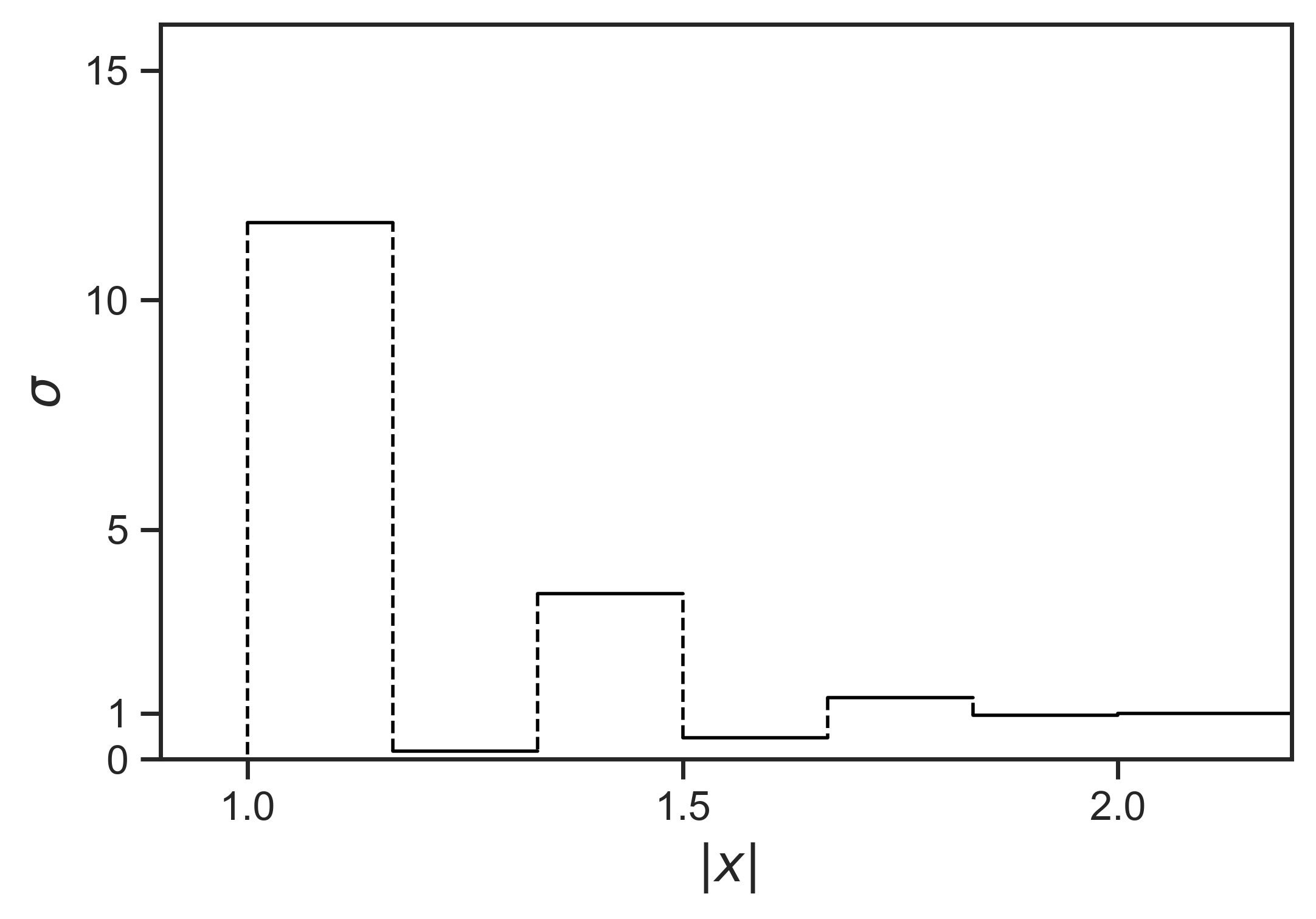} 
        \caption{Conductivity profile with $L=6$ for an insulated inclusion}
        \label{fig:sigma_L=6}
    \end{subfigure}
    
     \vspace{0.5cm}
     
   \begin{subfigure}{1\textwidth}
    \centering
   \begin{minipage}{0.48\linewidth}
        \centering
        \includegraphics[width=\linewidth]{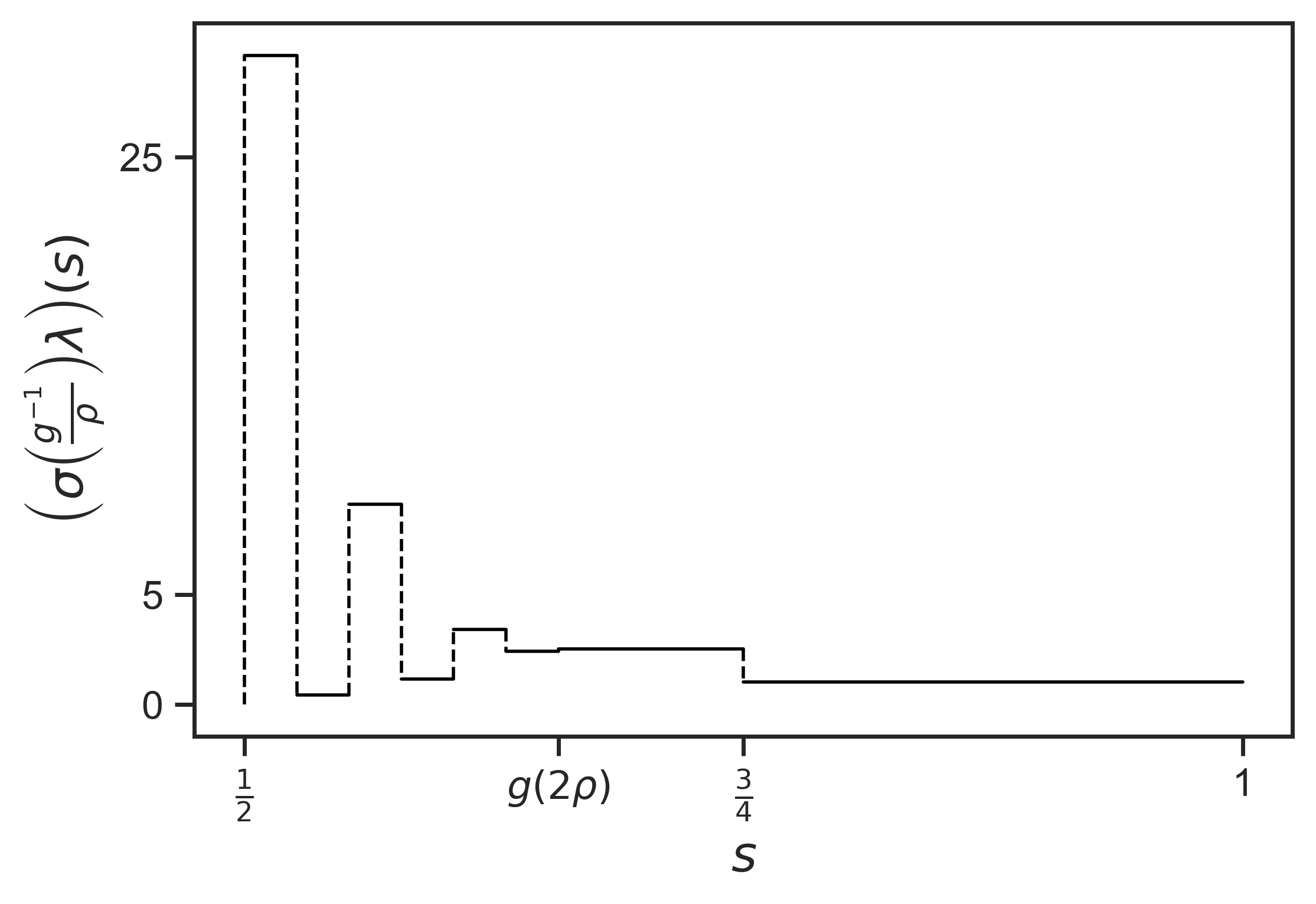}
    \end{minipage}\hfill
    \begin{minipage}{0.48\linewidth}
        \centering
        \includegraphics[width=\linewidth]{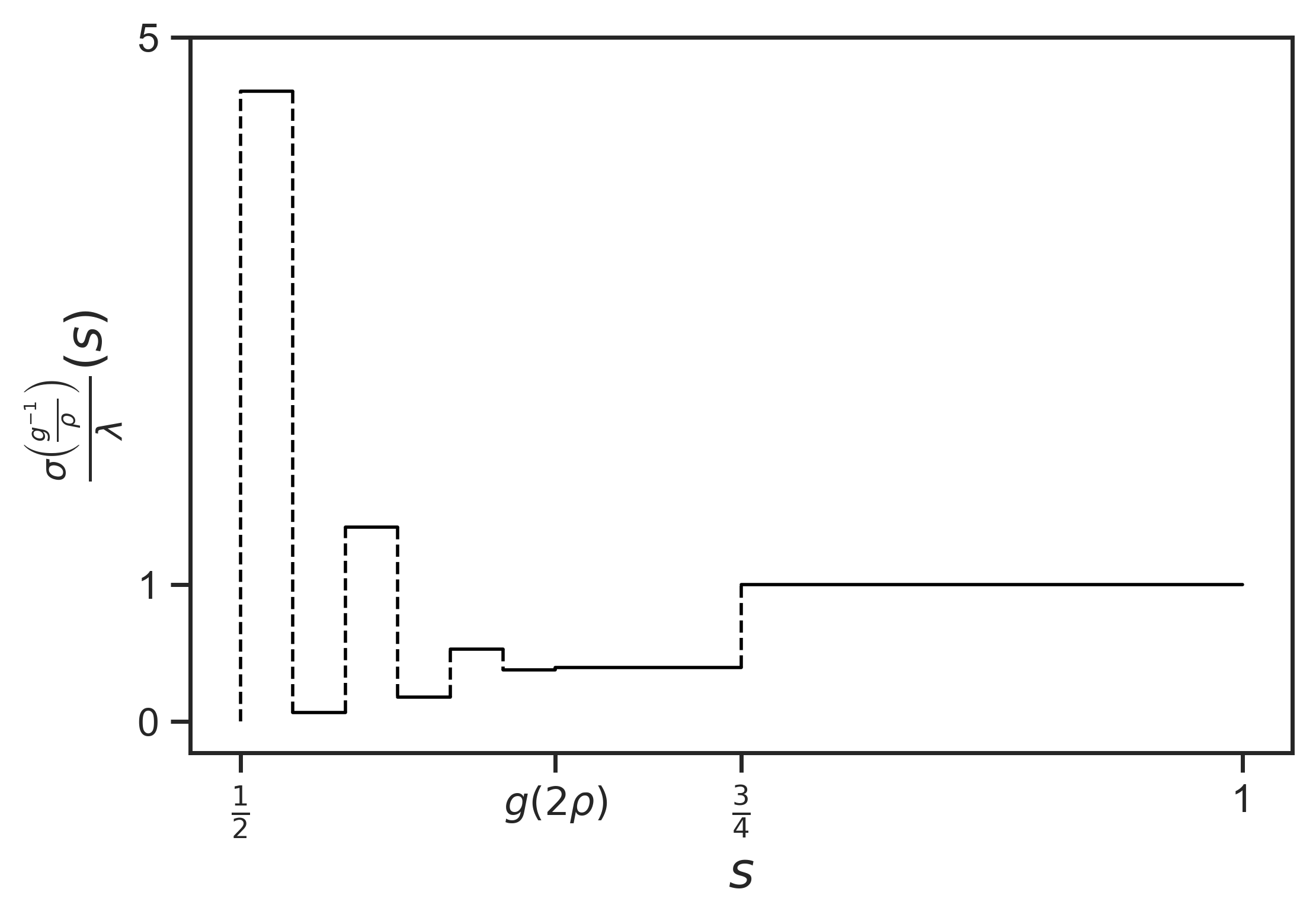}
    \end{minipage}
    \caption{The tangential (left) and radial (right) eigenvalues of $F_{*}(\sigma\circ\Psi_{1/\rho})$}
    \label{fig:eigen_L=6}
\end{subfigure}
     \vspace{0.5cm}
    \begin{subfigure}{\textwidth}
        \centering
        \includegraphics[width=0.6\linewidth]{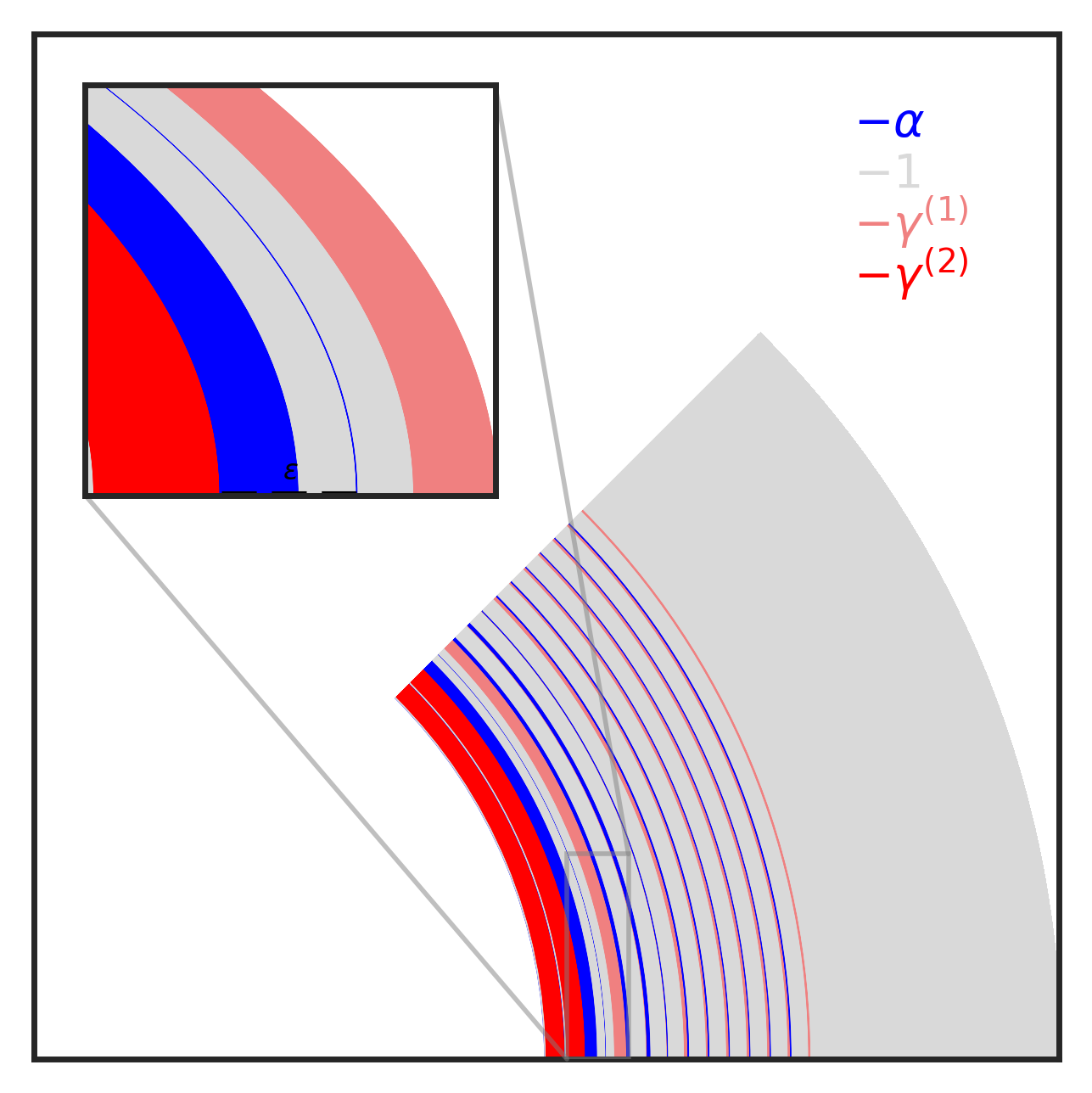}    
        \caption{A laminate with with a uniform scale and four materials}
        \label{fig:lam_enh_L=6}
    \end{subfigure}
    \caption{Cloaking laminate based on a GPT-vanishing structure of order $N=6$ in $\RR^2$.}
    \label{fig:L=6}
\end{figure}

\begin{figure}
    \centering
    \begin{subfigure}{\textwidth}
        \centering
        \includegraphics[width=0.5\linewidth]{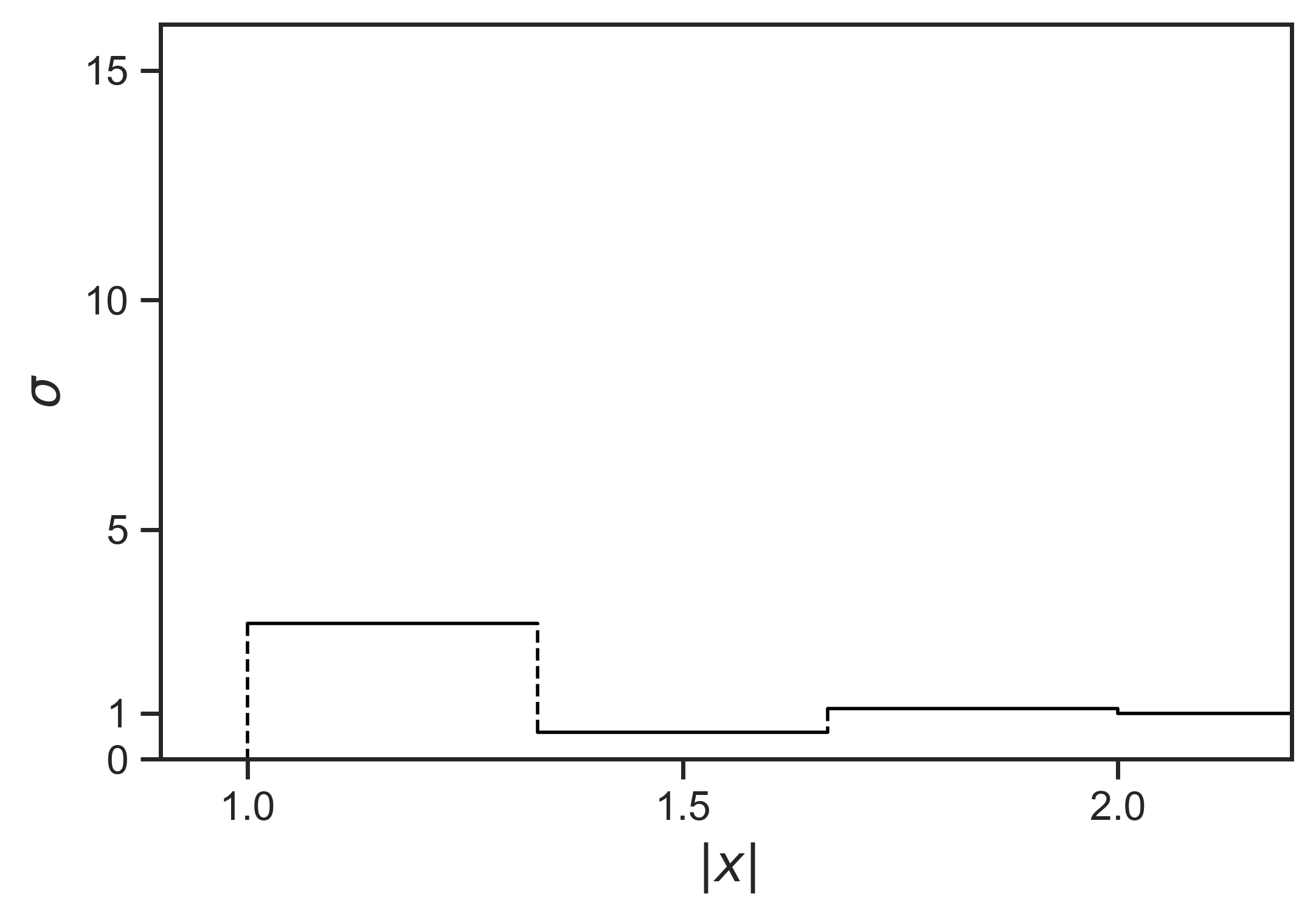} 
        \caption{Conductivity profile with $L=3$ for an insulated inclusion}
        \label{fig:sigma_d=3}
    \end{subfigure}
    
     \vspace{0.5cm}
     
   \begin{subfigure}{1\textwidth}
    \centering
   \begin{minipage}{0.48\linewidth}
        \centering
        \includegraphics[width=\linewidth]{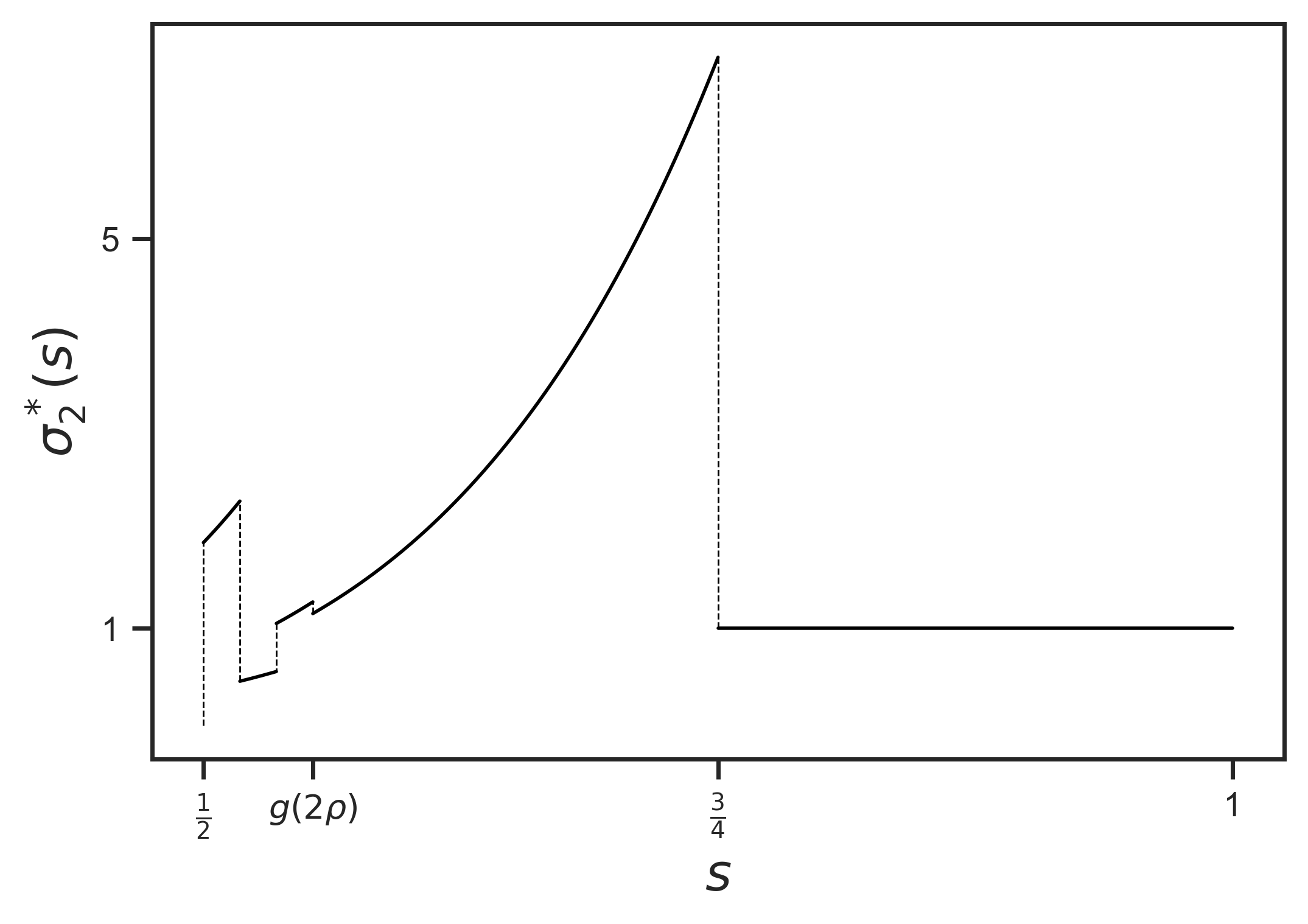}
    \end{minipage}\hfill
    \begin{minipage}{0.48\linewidth}
        \centering
        \includegraphics[width=\linewidth]{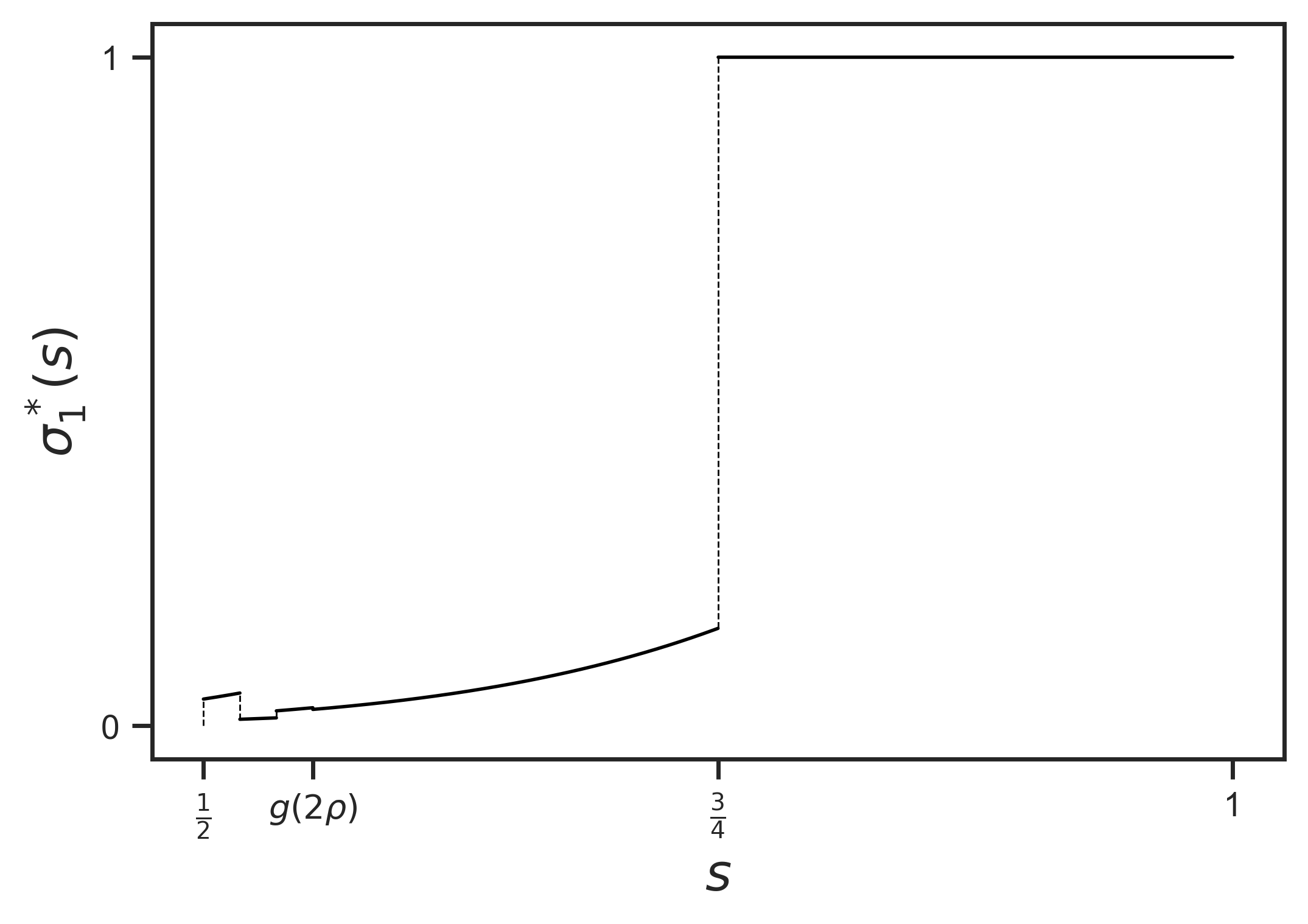}
    \end{minipage}
    \caption{The tangential (left) and radial (right) eigenvalues of $F_{*}(\sigma\circ\Psi_{1/\rho})$}
    \label{fig:eigen_d=3}
\end{subfigure}
     \vspace{0.5cm}
    \begin{subfigure}{\textwidth}
        \centering
        \includegraphics[width=0.6\linewidth]{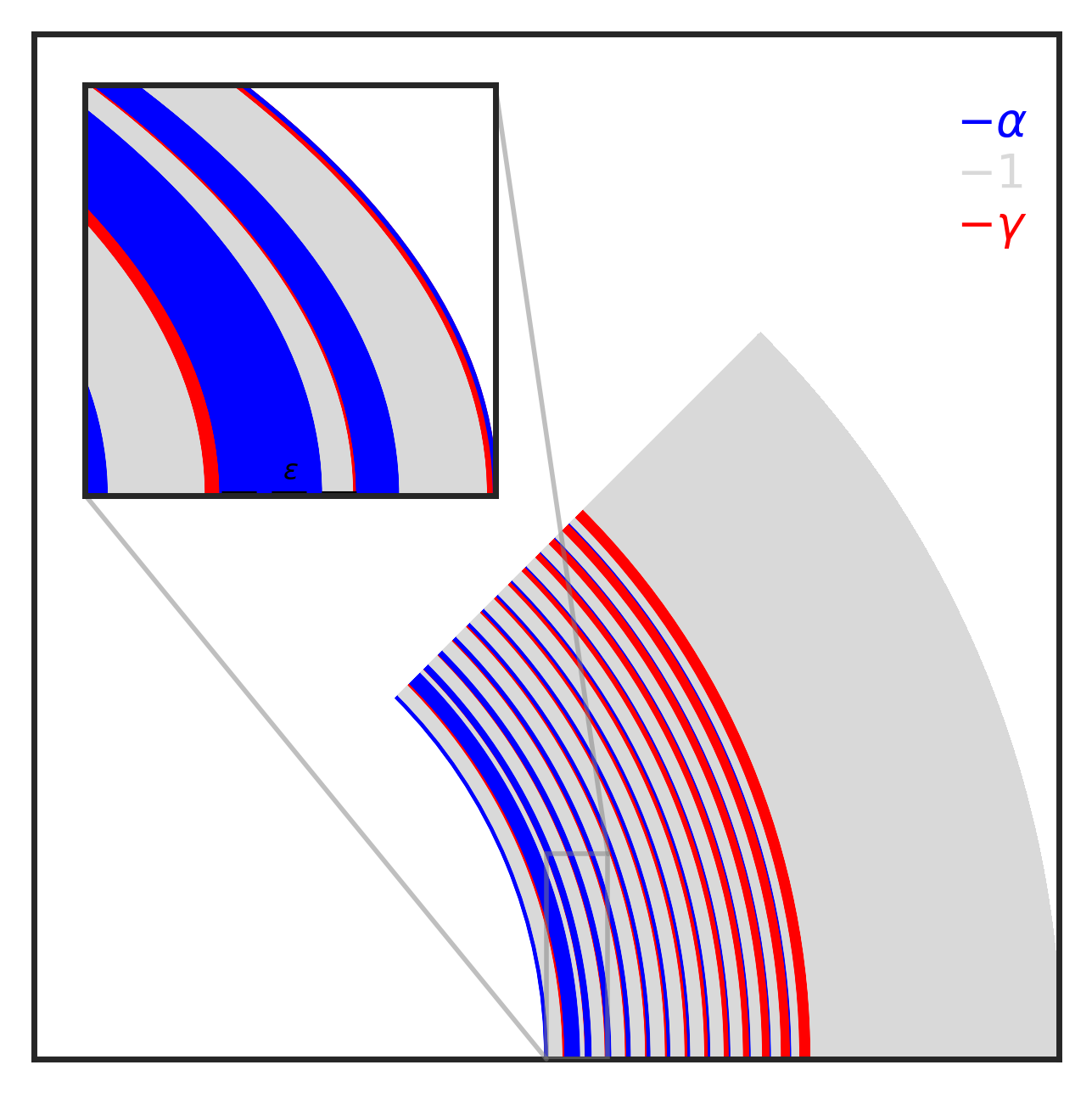}    
        \caption{A laminate with with a uniform scale and three materials}
        \label{fig:lam_enh_d=}
    \end{subfigure}
    \caption{Cloaking laminate based on a GPT-vanishing structure of order $N=3$ in $\RR^3$.}
    \label{fig:d=3}
\end{figure}

 \section{Conclusion}\label{sec:conclusion}
The paper shows that that  radial laminates constructed via homogenization can be further improved by  employing a multi-coating structure that cancels GPTs of lower order. The improvement lies either in reducing the contrast required of the constituent materials or allowing a coarser microstructure. In two dimensions,  the multi-coating strategy offers at most linear gain in reducing the number of required layers compared to the non-coated regime. While the near-cloak based on a GPT-vanishing coating scheme was constructed for a perfectly insulated inclusion, arbitrary inclusion can be cloaked at the cost of introducing a thin layer of low-conducting material.
\bibliographystyle{amsalpha}
\bibliography{GemidaLim_bib2}

\end{document}